\documentclass[autodetect-engine, a4paper]{amsart}
\usepackage[utf8]{inputenc}
\usepackage[margin=1.5in]{geometry}
\usepackage{amsmath}
\usepackage{amscd}
\usepackage{amssymb}
\usepackage{amsthm}
\usepackage{ascmac}
\usepackage{url}
\usepackage[X2,T1]{fontenc} 
\DeclareTextSymbolDefault{\cyrsdsc}{X2}
\usepackage{mathtools}
\usepackage{tikz-cd}
\usetikzlibrary{cd}
\usepackage{mathrsfs}
\usepackage{hyperref}
\hypersetup{breaklinks=true} 
\usepackage[all]{xy} 
\usepackage{color}
\usepackage{enumerate}
\theoremstyle{definition}
\newtheorem{thm}{Theorem}[section]
\newtheorem{lem}[thm]{Lemma}
\newtheorem{cor}[thm]{Corollary}
\newtheorem{prop}[thm]{Proposition}
\newtheorem{defn-prop}[thm]{Definition-Proposition}
\newtheorem{defn-thm}[thm]{Definition-Theorem}

\newtheorem{rem}[thm]{Remark}

\newtheorem{defn}[thm]{Definition}

\newcommand{\rank}{\mathop{\mathrm{rank}}}
\def\p{{\mathfrak p}}
\def\F{{\mathbb F}}
\def\G{{\mathbb G}}

\def\P{{\mathbb P}}
\def\Q{{\mathbb Q}}
\def\C{{\mathbb C}}
\def\Z{{\mathbb Z}}
\def\cO{\mathcal{O}}
\def\R{\mathbb R}
\DeclareMathOperator\CH{\mathrm{CH}}
\DeclareMathOperator\A{\mathrm{A}}
\DeclareMathOperator\Ab{\mathrm{Ab}}
\DeclareMathOperator\GL{\mathrm{GL}}

\DeclareMathOperator\SO{\mathrm{SO}}
\DeclareMathOperator{\Spin}{\mathrm{Spin}}

\DeclareMathOperator{\coker}{Coker}
\DeclareMathOperator\Aut{\mathrm{Aut}}

\DeclareMathOperator\End{\mathrm{End}}
\DeclareMathOperator\Gal{\mathrm{Gal}}

\DeclareMathOperator\Pic{\mathrm{Pic}}
\DeclareMathOperator\Sym{\mathrm{Sym}}
\DeclareMathOperator\Frac{\mathrm{Frac}}
\DeclareMathOperator\Spec{\mathrm{Spec}}
\DeclareMathOperator\Proj{\mathrm{Proj}}
\DeclareMathOperator\Shaf{\mathrm{Shaf}}

\DeclareMathOperator\sgn{\mathrm{sgn}}

\DeclareMathOperator\Gr{\mathrm{Gr}}
\DeclareMathOperator\OGr{\mathrm{OGr}}
\DeclareMathOperator\OFl{\mathrm{OFl}}
\DeclareMathOperator\Br{\mathrm{Br}}
\DeclareMathOperator\disc{\mathrm{disc}}
\DeclareMathOperator{\image}{Im}
\DeclareMathOperator{\PGL}{PGL}
\def\et{\mathop{\text{\rm \'et}}}

\newcommand{\cred}{\color{black}}

\newcommand{\cora}{\color{black}}

\usepackage[
backend=biber,
bibstyle=ieee-alphabetic,
citestyle=ieee-alphabetic,
sorting=nyt,
isbn=false,
url=false,
doi=false,
giveninits=true,
maxnames=10,
labelalpha=true,
maxalphanames=4,
dashed=false,
]{biblatex}
\DeclareCaseLangs{}
\DeclareFieldFormat*{date}{\mkbibparens{#1}}
\DeclareFieldFormat*{volume}{\mkbibbold{#1}}
\DeclareFieldFormat*{title}{\mkbibemph{#1\isdot}}
\DeclareFieldFormat*{journaltitle}{#1}
\DeclareFieldFormat*{pages}{{#1}}

\AtEveryBibitem{%
  \ifentrytype{online}{%
    \DeclareFieldFormat*{date}{#1}%
  }{%
  }%
}

\title[Arithmetic finiteness of Mukai varieties of genus 7]{Arithmetic finiteness of Mukai varieties of genus 7}

\author{Tetsushi Ito}
\address{Department of Mathematics, Graduate School of Science, Kyoto University, Kyoto 606-8502, Japan}
\email{tetsushi@math.kyoto-u.ac.jp}

\author{Akihiro Kanemitsu}
\address{Department of Mathematical Sciences, Graduate School of Science, Tokyo Metropolitan University, 1-1 Minami-Osawa, Hachioji-shi, Tokyo 192-0397, Japan}
\email{kanemitsu@tmu.ac.jp}

\author{Teppei Takamatsu}
\address{Department of Mathematics (Hakubi center), Graduate School of Science, Kyoto University, Kyoto 606-8502, Japan}
\email{teppeitakamatsu.math@gmail.com}

\author{Yuuji Tanaka}
\address{Beijing Institute of Mathematical Sciences and Applications, No. 544, Hefangkou Village, Huaibei Town, Huairou District, Beijing 101408, China}
\email{ytanaka@bimsa.cn}

\begin{document}

\numberwithin{equation}{subsection}

\begin{abstract}
We study arithmetic finiteness of prime Fano threefolds of genus 7 and their higher dimensional generalization, called Mukai varieties of genus 7.
For prime Fano threefolds of genus 7,  we provide an arithmetic refinement of the Torelli theorem,  
obtain Shafarevich-type finiteness results, and show the failure of the N\'eron--Ogg--Shafarevich criterion of good reduction.
For Mukai varieties of genus 7, we prove that Shafarevich-type finiteness results hold in dimensions 9 and 10, but fail in dimension 6.
In addition, we show that Mukai $n$-folds of genus 7 over $\mathbb{Z}$ do not exist for $n \leq 4$,
whereas they exist for $5 \leq n \leq 10$.
\end{abstract}


\maketitle


\section{Introduction}

Finiteness of algebraic varieties over number fields and rings of integers is one of the important subjects in arithmetic geometry.
Faltings (\cite{FaltingsShaf}) and Zarhin (\cite{Zarhin})
proved the Shafarevich conjecture for curves and abelian varieties,
which states the finiteness of isomorphism classes of
curves of genus $g \geq 2$ or abelian varieties of dimension $d$
over a number field with good reduction outside
a fixed finite set of places.
Fontaine proved that there is no abelian variety over $\Z$ (\cite{Fontaineabelian}).
Since the works of Faltings, Zarhin, and Fontaine,
analogous results have been studied
for other classes of varieties by several authors;
see, for example,
\cite{Scholl}, \cite{Andre}, \cite{She}, \cite{Takamatsu}, \cite{Fu-Li-Takamatsu-Zou}, \cite{Javanpeykar-Loughran:flag},
\cite{Nagamachi-Takamatsu},
\cite{Javanpeykarcanonicallypolarized},
\cite{TakamatsuEnriques},
\cite{Takamatsubielliptic},
\cite{Lawrence-Sawin},
\cite{Kramer-Maculan}
for Shafarevich-type results,
and see
\cite{Abrashkin}, \cite{Fontaine}, \cite{Schroer}
for Fontaine-type results.

In this paper, we address these problems for prime Fano threefolds of genus 7, and their higher dimensional generalizations \emph{Mukai varieties of genus 7}.
These varieties are Fano varieties obtained as linear sections of the spinor tenfold, and, because of 
the specific geometric properties of the spinor tenfold (e.g., self-duality), these varieties have been studied intensively in several works.
For example, in his seminal works, Mukai proved several important results for them including global Torelli theorem for prime Fano threefolds of genus 7 over $\C$ (see \cite{Mukaimukain-fold}, 
\cite{Mukaicurve},
\cite{MukaiBrill},
\cite{MukaiSugakuShintenkaiEnglish},
\cite{Mukaicurve3}).
Building on Mukai's work, we study arithmetic of
Mukai varieties of genus $7$ over non-closed fields
and in mixed characteristic.
The results are summarized as follows.
In the first part,
we concentrate on prime Fano threefolds of genus $7$, and prove the following results:
\begin{itemize}
    \item
(Arithmetic Torelli theorem) \ The arithmetic period map defined by the intermediate Jacobians is injective
(see Theorem \ref{thm:introarithmeticTorelli}).
\item
(Shafarevich conjecture) \ For a number field $K$ and a finite set of places $S$,
there exist only finitely many $K$-isomorphism classes of prime Fano threefolds of genus $7$ over $K$
which have good reduction outside $S$ (see Theorem \ref{introMainTheoremshaf}).
\item
(Counter-examples to the N\'eron--Ogg--Shafarevich criterion) \ There exist prime Fano threefolds of genus $7$ over
a $p$-adic field that do not have good reduction,
while their $\ell$-adic cohomology is unramified for any $\ell \neq p$
 (see Theorem \ref{thm:introgrcounterexample}).
\end{itemize}
In the second part,
we study higher dimensional Mukai varieties of genus $7$, and prove the following results:
\begin{itemize}
\item
(Shafarevich conjecture for Mukai $n$-folds of genus $7$) \ Shafarevich-type finiteness results for Mukai $n$-folds of genus $7$ hold
if $n = 9, 10$, but fail if $n = 6$
(see Theorem \ref{thm:introMukaifinite}).
\item 
(Fontaine-type results for Mukai $n$-folds of genus $7$) \ Mukai $n$-folds of genus 7 over $\mathbb{Z}$ do not exist for $n \leq 4$,
whereas they do exist for $5 \leq n \leq 10$
(see Theorem \ref{thm:intromukainoverZ}).
\end{itemize}
To the best of our knowledge,
our results give (1) the first example of Fano varieties of Picard rank $1$ for which the Shafarevich conjecture fail, and (2) the first non-trivial examples of smooth projective $n$-folds over $\mathbb{Z}$
that are not defined combinatorially. See Remark~\ref{rem:smproj_over_z}.
See also Remark~\ref{rem:intrommp} and
Remark~\ref{rem:nonexistence}
below for the various different techniques that are used to prove the above results.
In the following, we explain our results in more detail.

\subsection{Arithmetic Torelli theorem for prime Fano threefolds of genus $7$}

First, we study the structure of
the moduli stack of prime Fano threefolds of genus 7
in characteristic $0$,
following Mukai's global Torelli theorem of prime Fano threefolds of genus $7$ over $\C$ in \cite{MukaiBrill}, and Kuznetsov's study on families of Fano varieties \cite{Kuznetsovfamily}.
We briefly recall the definition of Prime Fano threefolds of genus 7.
Let $K$ be a field, and $\overline{K}$ an algebraic closure of $K$.
A smooth projective variety $X$ over $K$ is called a \emph{Fano variety}
if the anti-canonical divisor $-K_{X}$ is ample.
The rank of the Picard group of $X_{\overline{K}} \coloneqq X \otimes_K \overline{K}$
is called the \emph{(geometric) Picard rank} of $X$.
The \emph{index} of a Fano variety $X$ over $K$ is the 
largest positive integer $r$ such that $- K_{X_{\overline{K}}} \sim r L$ for some divisor $L$ on $X_{\overline{K}}$.
In this paper, we only consider the case where $r = \dim X -2$ (the case of \emph{coindex $3$}).
In this case, the \emph{genus} $g$ of $X$ is defined by 
$( -K_{X}/r)^{\dim X} = 2g-2$.
A \emph{prime Fano threefold of genus $7$} is a Fano threefold of Picard rank one, index one and genus $7$.
We prove an arithmetic analogue of the global Torelli theorem as follows:

\begin{thm}[Arithmetic Torelli theorem; see Theorem \ref{thm:moduliarithmetictorelli}]
\label{thm:introarithmeticTorelli}
\begin{enumerate}
\item
Let $\mathcal{F}_{\Q}$ (resp. $\mathcal{M}_{\mathrm{sm}, \Q}$) be the moduli stack of prime Fano threefolds of genus 7 (resp.\ smooth projective curves of genus 7) over $\Q$. 
Then we have an open embedding
\[
\Gamma
\colon
\mathcal{F}_{\Q} \hookrightarrow \mathcal{M}_{\mathrm{sm}, \Q},
\]
whose image is the moduli stack of smooth projective curves of genus 7 without $g_4^1$.
\item
Let $\mathcal{A}_{\Q}$ be the moduli stack of principally polarized abelian varieties of dimension $7$ over $\Q$,
and $j \colon \mathcal{M}_{\mathrm{sm}, \Q} \hookrightarrow \mathcal{A}_{\Q}$
the period map defined by the (polarized) Jacobian of curves.
For any field $k$ of characteristic $0$,
$j \circ \Gamma (k)$ is given by 
\[
X \mapsto (\Ab^2_{X/k}, \Theta_{X}),
\]
where $(\Ab^2_{X/k}, \Theta_{X})$ is the (polarized) algebraic representative of $\A^2$, which is a $k$-form of the intermediate Jacobian of $X$.
\end{enumerate}
\end{thm}

The arithmetic period map $\Gamma$ is
defined in Definition-Proposition \ref{defn-prop:dualvariety}.
For necessary backgrounds on the intermediate Jacobian of Fano threefolds over non-closed fields,
see Section \ref{section:intermediate}.

\begin{rem}
For cubic threefolds and odd-dimensional complete intersections of three quadrics,
analogous results have been obtained by
Achter, Ciurca, Javanpeykar--Loughran
(see \cite{Achter}, \cite{ciurca2024prymvarietiescubicthreefolds}, \cite{Javanpeykar-Loughran:completeintersection}).
\end{rem}

\subsection{Shafarevich conjecture for prime Fano threefolds of genus $7$}

As an application of the arithmetic Torelli theorem, we obtain a cohomological generalization of the Shafarevich conjecture (cf.\ \cite{Takamatsu}) for prime Fano threefolds of genus $7$ over finitely generated fields of characteristic $0$.

\begin{thm}[Cohomological Shafarevich conjecture; see Theorem \ref{thm:cohomshaf}]
\label{introMainTheoremshaf}
Let $K$ be a finitely generated field of characteristic $0$,
and $R$ a finitely generated normal $\Z$-algebra
with $K = \Frac R$.
Let $\Shaf(K,R)$ be the set of $K$-isomorphism classes of prime Fano threefolds $X$ of genus $7$ over $K$
satisfying the following condition:
\begin{quote}
For every prime ideal $\p \in \Spec R$ of height $1$,
the action of the absolute Galois group
$\Gal(\overline{K}/K)$ on the $\ell$-adic cohomology $H^3_{\et}(X_{\overline{K}}, \Q_{\ell})$
is unramified at $\p$ for some prime $\ell \notin \p$.
\end{quote}
Then the set $\Shaf(K,R)$ is finite.
\end{thm}

We recall the notion of good reduction of Fano varieties:
we say $X$ has \emph{good reduction at $\p$} (see \cite[Definition 4.3]{Javanpeykar-Loughran:GoodReductionFano} and Definition \ref{defn:goodreduction})
if there exists a Fano scheme (see \cite[Definition 2.1]{Javanpeykar-Loughran:GoodReductionFano})
$\mathfrak{X}$ over $\Spec R_{\p}$
such that $\mathfrak{X} \otimes_{R_{\p}} K \cong X$,
where $R_{\p}$ is the localization of $R$ at $\p$.
If a prime Fano threefold $X$ of genus $7$ has good reduction at $\p$,
the action of $\Gal(\overline{K}/K)$ on
$H^3_{\et}(X_{\overline{K}}, \Q_{\ell})$
is unramified at $\p$
for every prime $\ell \notin \p$.
Hence we have the following corollary,
which might be called the \emph{ordinary Shafarevich conjecture}.

\begin{cor}[Ordinary Shafarevich conjecture]
\label{introodrinaryshaf}
Let $K, R$ be as in Theorem \ref{introMainTheoremshaf}.
Let $\Shaf'(K,R)$ be the set of $K$-isomorphism classes of prime Fano threefolds $X$ of genus $7$ over $K$
such that $X$ has good reduction at all prime ideals $\p \in \Spec R$ of height $1$.
Then the set $\Shaf'(K,R)$ is finite.
\end{cor}

In fact, we have $\Shaf'(K,R) \subset \Shaf(K,R)$, and hence the finiteness of $\Shaf(K,R)$ implies that of $\Shaf'(K,R)$.

It is natural to ask whether
the inclusion $\Shaf'(K,R) \subset \Shaf(K,R)$ is the equality.
In other words, does the unramifiedness of the $\ell$-adic cohomology $H^3_{\et}(X_{\overline{K}}, \Q_{\ell})$
imply $X$ has good reduction at $\p$?
We shall show that the answer is negative in general.
Indeed, we construct a counter-example in mixed characteristic $(0,p)$ and equal-characteristic $0$ cases (see Theorem \ref{thm:introgrcounterexample}).
Therefore, our \emph{cohomological generalization} (Theorem \ref{introMainTheoremshaf})
is more general than the ordinary Shafarevich conjecture (Corollary \ref{introodrinaryshaf}).
See \cite{Takamatsu}, where similar results were obtained for K3 surfaces.

The Shafarevich conjecture for
some classes of Fano threefolds of Picard rank $1$
have been proved by Javanpeykar--Loughran (\cite{Javanpeykar-Loughran:GoodReductionFano}, \cite{Javanpeykar-Loughran:completeintersection})
and Licht \cite{Licht}.
More precisely, they proved the Shafarevich conjecture for Fano threefolds of Picard rank 1 and index greater than or equal to $2$, and prime Fano threefolds of genus $2,3,4,5$,
whereas the case of prime Fano threefolds of genus $7$ remained open.
See Remark~\ref{rem:separatedness}.

\begin{rem}
In the course of the proof of Theorem \ref{thm:introarithmeticTorelli} and Theorem \ref{introMainTheoremshaf},
we prove that for a smooth Fano threefold,
the action of $\Gal(\overline{K}/K)$ on
$H^3_{\et}(X_{\overline{K}}, \Q_{\ell})$
is unramified at $\p$ for some prime $\ell \notin \p$
if and only if it is unramified for every prime $\ell \notin \p$.
Moreover, if $\ell \in \p$ and $K$ is a number field,
we may use the condition on $p$-adic Hodge theory;
see Corollary \ref{cor:unramlindep}.
\end{rem}

\subsection{Counter-examples to the N\'eron--Ogg--Shafarevich criterion for prime Fano threefolds of genus $7$}

As stated in the previous subsection,
we also prove the following result.

\begin{thm}[Counter-examples to the N\'eron--Ogg--Shafarevich criterion; see Theorem \ref{thm:grcounterexample}]
\label{thm:introgrcounterexample}
Let $p$ be a prime number.
Then there exist a number field $K$, a finite place $\p$ lying above $p$, and a prime Fano threefold $X$ of genus $7$ over $K$
such that $X$ does not have potential good reduction at $\p$, while $H^3_{\et}(X_{\overline{K}}, \Q_{\ell})$ is unramified at $\p$ for any prime $\ell \neq p$.
\end{thm}

We say $X$ has \emph{potential good reduction at $\p$}
if it has good reduction at a prime ideal of height $1$ lying
above $\p$ after replacing $K$ by its finite extension.

The idea of the proof of Theorem \ref{thm:introgrcounterexample}
can already be found in
Mukai's work over $\C$ \cite[Section 8, (b)]{MukaiBrill}, where he claimed that some degeneration of smooth curves (such that the special fiber is also smooth) correspond to the degeneration of prime Fano threefolds of genus 7 to the Gorenstein (non-smooth) Fano threefolds via the {\it non-commutative Brill--Noether theory}.
We verify that the generic fiber of such a family satisfies the condition in Theorem \ref{thm:introgrcounterexample}.
It is known that a curve of genus 7 can be constructed
from a given prime Fano threefold of genus $7$
via certain birational transformations called \emph{two-ray games} (see \cite[Section 8]{MukaiBrill}).

\begin{rem}
\label{rem:intrommp}
In our proof of Theorem \ref{thm:introgrcounterexample},
it is essential to perform two-ray games for families in
mixed characteristic by using relative minimal model program (see Proposition \ref{prop:mixedtwo-ray}).
For this purpose, we heavily use Tanaka's recent results on Fano threefolds and the existence of flops in characteristic $p>0$ (\cite{Tanaka1}, \cite{Tanaka2}, \cite{Tanakaflop}).
\end{rem}

\subsection{Shafarevich conjecture for Mukai $n$-folds of genus $7$}

In the second part of this paper,
we study Mukai $n$-folds of genus $7$, which are defined as (relative) linear sections of the spinor tenfold. (See Definition \ref{defn:Mukaivariety}. See also Proposition \ref{prop:MukaivssplitMukai} for a comparison with the numerical definition of \emph{Fano manifolds of coindex three and genus 7}.)
Since the Hodge diamonds of Mukai $n$-folds of genus $7$ are diagonal when $n \geq 4$,
we consider ordinary, i.e., non-cohomological, Shafarevich conjecture
in dimension $\geq 4$.

\begin{thm}[Ordinary Shafarevich conjecture for Mukai $n$-folds of genus $7$]
\label{thm:introMukaifinite}
\hfill
\begin{enumerate}
\item ($n=10$) \ The (ordinary) Shafarevich conjecture holds true for Mukai tenfolds of genus $7$.

\item ($n=9$) \ 
For a number field $K$, the number of $K$-isomorphism classes of Mukai ninefolds of genus $7$ over $K$ is equal to $2^r$,
where $r$ denotes the number of real places of $K$.
In particular, the Shafarevich conjecture holds true for Mukai ninefolds of genus $7$.

\item ($n=6$) \ The (ordinary) Shafarevich conjecture fails for Mukai sixfolds of genus $7$.
\end{enumerate}
\end{thm}

See Theorem \ref{thm:counterexampleshaf}, Theorem \ref{thm:Mukai10shaf}, and Theorem \ref{thm:numberofmukai9} 
for precise statements.

Assertion (1) is a direct corollary of Javanpeykar--Loughran's result (\cite{Javanpeykar-Loughran:flag}).
Assertion (2) is stronger than the usual Shafarevich conjecture, and can be seen as an analogue of the classification of forms of del Pezzo quintics (\cite[Theorem 1]{DKM}) in our setting.
To the best of our knowledge, assertion (3) gives
the first example of Fano varieties of Picard rank 1
for which the (ordinary) Shafarevich conjecture fails.
This counter-example is motivated by the existence of a non-isotrivial family of Mukai sixfolds of genus 7
constructed by Kuznetsov \cite[Subsection 1.2]{Kuznetsovspinor} (see also \cite{mathoverflow}).
To obtain infinitely many integral models, we use Hassett--Tschinkel's results on the potential density of integral points \cite{Hassett-Tschinkel}.

\begin{rem}
For the remaining values of $n$ (namely, $n \in \{ 4,5,7,8 \}$),
it remains open whether the ordinary Shafarevich conjecture for Mukai $n$-folds of genus $7$ holds in these cases.
See Section \ref{section:Remarks} for further discussion.
\end{rem}

\subsection{Fontaine-type results for Mukai $n$-folds of genus $7$}

Another interesting and important problem is determining the existence or non-existence of smooth projective varieties over $\Z$.

This problem was originally raised by Grothendieck (\cite[page 242]{Mazur}).
The non-existence of elliptic curves over $\Z$ was proved by Ogg,
who attributed it to Tate (\cite{Ogg}),
and later Abrashkin and Fontaine generalized this to abelian varieties (\cite{Abrashkinabelian} and \cite{Fontaineabelian}).
More generally, Abrashkin and Fontaine proved that if a smooth projective variety over $\Z$ exists,
then the Hodge numbers of the generic fiber satisfy $h^{i,j} =0$ for any $i\neq j$ with $i+j \leq 3$ (\cite{Abrashkin},  \cite{Fontaine}). 
Recently, Schr\"{o}er (\cite{Schroer}) proved that there is no Enriques surface over $\Z$,
where the results of Abrashkin and Fontaine do not apply.

We completely determine the integers $n$
such that there exist Mukai $n$-folds of genus $7$ over $\Z$.

\begin{thm}[Fontaine-type results for Mukai $n$-folds of genus $7$; see Theorem \ref{thm:MukainoverZ}]
\label{thm:intromukainoverZ}
Mukai $n$-folds of genus 7 over $\mathbb{Z}$ do not exist for $n \leq 4$,
whereas they do exist for $5 \leq n \leq 10$.
\end{thm}

The results of Abrashkin and Fontaine imply the non-existence of
Mukai $n$-fold of genus $7$ over $\Z$ when $n \leq 3$
(for the cases of curves and surfaces, see Remark \ref{rem:Mukai1folds,2-folds}).
But their results do not apply to the case $n\geq 4$,
where Hodge diamonds are diagonal.
On the other hand, there is a Mukai $10$-fold over $\Z$,
which is the homogeneous space of the spinor group over $\Z$.
Therefore, it remains to determine for which integers $4 \leq n \leq 9$ Mukai $n$-folds of genus $7$ over $\Z$ exist. 

\begin{rem}
\label{rem:nonexistence}
The proof of the existence for $5 \leq n \leq 9$ is done by constructing (integral) linear sections of spinor tenfolds directly.
For this we use the self-duality of the spinor tenfolds.
On the other hand, the non-existence part requires a study of the geometry of the spinor tenfold over $\F_{2}$,
which is similar to that of Schr\"{o}er's proof of the non-existence of Enriques surfaces over $\Z$ \cite{Schroer};
he reduced the non-existence of Enriques surfaces over $\Z$
to the non-existence of certain Enriques surfaces over $\F_2$.
\end{rem}

\begin{rem}
\label{rem:smproj_over_z}
Only a few examples of smooth projective varieties over $\Z$ are known.
It seems that the only known examples of smooth projective varieties over $\Z$ are combinatorially defined,
such as homogeneous spaces, toric varieties, and varieties obtained from them by applying operations like products, blow-ups along linear centers, and taking the Hilbert schemes, etc.
See \cite[Introduction]{Schroer}.
To the best of our knowledge,
our results (Theorem \ref{thm:intromukainoverZ} for $5 \leq n \leq 9$) give
the first non-trivial examples of
smooth projective varieties over $\mathbb{Z}$ that are not defined combinatorially.
\end{rem}

\subsection*{Organization of this paper}

The organization of this paper is as follows.
In Section \ref{section:intermediate}, we recall the definitions of Fano threefolds and intermediate Jacobians. We also review recent works on intermediate Jacobians over non-closed fields, and we apply them to prove the $\ell$-independence results for unramifiedness. 
In Section \ref{section:fanovscurve}, we review the basic properties of prime Fano threefolds of genus $7$ and the correspondence with curves of genus $7$ without $g_4^1$.
By using Kuznetsov's method (\cite{Kuznetsovfamily}), we construct spinor embeddings as families and apply them to prove Theorem \ref{thm:introarithmeticTorelli}.
Moreover, as an application, we prove Theorem \ref{introMainTheoremshaf}.
In Section \ref{section:goodredcriteria}, we recall the two-ray game construction and, by generalizing that to a mixed characteristic setting, we prove Theorem \ref{thm:introgrcounterexample}.
In Section \ref{Section:Mukain}, we recall the basic properties of Mukai $n$-folds of genus $7$, and prove  Theorem \ref{thm:introMukaifinite} and Theorem \ref{thm:intromukainoverZ}.
Finally, in Section \ref{section:Remarks}, we give further remarks on the possible future research prospects.

\subsection*{Notations and terminology}

In this paper, a \emph{variety} over a field is a geometrically integral scheme separated and of finite type over the field.
A \emph{curve} means a variety of dimension $1$.
For a locally free sheaf $E$ on a scheme $B$,
we denote the projective bundle 
$
\Proj_{B} (\Sym (E))
$
by $\P_{B} (E)$, which parametrizes rank one locally free quotients of $E$.
Similarly $\Gr (E,r)$ denotes the (relative) Grassmann variety parametrizing rank $r$ locally free quotients of $E$. 
For an $\cO_{X}$-module $A$, we denote the degree 2 part of the symmetric algebra $\Sym (A)$ by $S^2 A$.
Also, we denote the dual $\cO_{X}$-module of $A$ by $A^{\vee}$.

For schemes $X$ over $S$ and $T$ over $S$, we denote the base change $X \times_{S} T$ by $X_{T}$.

For a scheme $X$, we denote the derived category of bounded complexes of coherent sheaves by $D^{\mathrm{b}} (X)$.

For any closed immersion $X \hookrightarrow Y$ of schemes, we denote the conormal sheaf of $X$ in $Y$ by $N^{\ast}_{X/Y}$.
We denote the normal sheaf $(N^{\ast}_{X/Y})^{\vee}$ by $N_{X/Y}$.

For any quadratic form $q$ of rank $n$ over a field $k$ of characteristic different from $2$, the determinant $\det q \in k^{\times}/ k^{\times 2}$ is the determinant of any matrix representation of $q$, 
and the discriminant is defined by
$\disc q \coloneqq (-1)^{\frac{n(n-1)}{2}} \det q \in k^{\times }/ k^{\times 2}$.
We denote the Clifford algebra (resp.\ even Clifford algebra) by $C(q)$ (resp.\ $C_{0} (q)$).
We denote the degree $n$-part of the Clifford algebra by $\mathrm{Cl}_{n} (q)$.
We say quadratic forms $q_{1}$ and $q_{2}$ over a field $k$
are \emph{$k$-similar} if there exists $c\in k^{\times}$ such that $q_{1} = cq_{2}$.

\section{Fano varieties and intermediate Jacobians}
\label{section:intermediate}

In this section, we review the generalities of Fano threefolds and intermediate Jacobians over non-closed fields.

\subsection{Definitions of Fano varieties and Fano threefolds}

A standard reference for complex Fano threefolds is \cite{Iskovskikh-Prokhorov}.

\begin{defn}
Let $k$ be a field, and $X$ a smooth projective variety over $k$.
\begin{enumerate}
\item
We say $X$ is a \emph{Fano variety} when the anti-canonical divisor $-K_{X}$ is ample.
We call a Fano variety of dimension $3$ a \emph{Fano threefold}.
\item
The \emph{index} of a Fano variety $X$ is the largest positive integer $r$ such that $-K_{X_{\overline{K}}} \sim r L$ for some divisor $L$ on $X_{\overline{K}}$.
\item
A Fano threefold $X$ is called a \emph{prime Fano threefold} when the (geometric) Picard rank and the index of $X$ are both $1$.
\item 
The \emph{genus} of a prime Fano threefold is the integer $\frac{(-K_{X})^3}{2} +1$.
\end{enumerate}
\end{defn}

\subsection{Intermediate Jacobians over non-closed fields}

In this subsection, we recall the results by Benoist--Wittenberg and Achter--Casalaina-Martin--Vial on intermediate Jacobian varieties over non-closed fields.
First, we recall the classical intermediate Jacobians of threefolds.

\begin{defn-prop}
Let $X$ be a smooth projective variety over $\C$ of dimension $3$.
We define the \emph{intermediate Jacobian} $J(X)$ by
\[
J(X) \coloneqq H^3(X,\C)/(F^2 H^3(X,\C) + H^3(X,\Z)),
\]
where $F^i$ is the Hodge filtration on $H^3(X,\C)$.
Then we can equip $J(X)$ with the structure of principally polarized complex torus of dimension $g = \frac{b_{3}}{2}$ (see \cite[Definition 3.5]{Clemens-Griffiths}). 
Moreover, if $h^{1,0} (X) = h^{3,0} (X)=0$ (complex Fano threefolds satisfy this condition), then the structure of polarized complex torus gives the structure of principally polarized abelian variety by \cite[Lemma 3.4]{Clemens-Griffiths}. 
 In this case, we denote its principal polarization by $\Theta (X)$. We call $\Theta (X)$ the \emph{theta polarization} on $J(X)$.
\end{defn-prop}

In order to formulate a model over a non-closed field, we begin by recalling the definition of Chow groups.

\begin{defn}
\label{defn:ChowandA}
Let $k$ be a perfect field, and $X$ a smooth projective variety over $k$.
For a non-negative integer $i \geq 0$,
the \emph{Chow group} $\CH^i(X)$ \emph{of degree} $i$ is a group consisting of codimension $i$-cycles modulo rational equivalence.
We denote the subgroup of $\CH^i(X)$ consisting of algebraically trivial cycles by $\A^i(X).$
\end{defn}

\begin{defn-thm}
\label{defn-thm:representative}
Let $k$ be a perfect field, and $X$ a smooth projective variety over $k$. 
Then $\A^2(X_{\overline{k}}) \subset \CH^2(X_{\overline{k}})$ admits the algebraic representative $\Ab^2_{X_{\overline{k}}/\overline{k}}$ in the sense of Murre \cite[Definition 1.6.1]{Murre}.
Moreover, $\Ab^2_{X_{\overline{k}}}$ uniquely descends to an abelian variety $\Ab^2_{X/k}$ over $k$ such that the regular morphism (see \cite[Definition 1.6.1]{Murre}) $\A^2(X_{\overline{k}}/\overline{k}) \rightarrow \Ab^2_{X/k} (\overline{k})$ is $\Gal (\overline{k}/k)$-equivariant. We call an abelian variety $\Ab^2_{X/k}$ the \emph{algebraic representative of $\A^2$ over $k$}.
The algebraic representatives satisfy the following.
\begin{enumerate}
\item 
 Let $L$ over $k$ be a separable field extension.
 Then we have a canonical isomorphism
 \[
 (\Ab^2_{X/k})_L \simeq \Ab^2_{X_{L}/L}.
 \]

\item
 If $k$ is a subfield of $\C$ and $X_{\C}$ is a uniruled threefold, then $(\Ab^2_{X/k})_{\C}$ is isomorphic to the intermediate Jacobian $J(X_{\C})$.
 \end{enumerate}
\end{defn-thm}

\begin{proof}
The first statement follows from \cite[Theorem A]{Murre}.
The second statement follows from \cite[Theorem 4.4]{ACMV1}.
The statement (1) follows from \cite[Theorem 1]{ACMVfunctorial} (see also \cite[Remark 1.1 and Remark 1.2]{ACMVdecomposition}).
The statement (2) follows from the proof of \cite[Theorem 5.1 and Corollary 5.2]{ACMV1}.
\end{proof}

\begin{prop}
\label{prop:representativeetale}
Let $k$ be a perfect field, $X$ a smooth projective variety of dimension $d$ over $k$, and $\Ab^2_{X/k}$ be the algebraic representative over $k$. Let $\ell$ be a prime number with $\ell \in k^{\times}$.
Then the following hold. 
\begin{enumerate}
\item 
Assume that there exists a positive integer $N$ such that 
\[
N \Delta_{X_{\overline{k}}} \in \CH^d((X \times X)_{\overline{k}})
\]
admits a cohomological decomposition of type $(d-1,1)$ with respect to $H^{\ast}_{\et} (\Z_{\ell})$ in the sense of \cite[Definition 6.1]{ACMVdecomposition}.
The regular morphism 
\[
\Phi^2 \colon
\A^2 (X_{\overline{k}}) \rightarrow \Ab^2_{X/k} (\overline{k})
\]
is a $\Gal(\overline{k}/k)$-equivariant isomorphism.
\item
Assume that there exists a positive integer $N$ such that 
\[
N \Delta_{X_{\overline{k}}} \in \CH^d((X \times X)_{\overline{k}})
\]
admits a decomposition of type $(d-1,1)$ in the sense of \cite[Subsection 3.1]{ACMVbloch}.
The second $\ell$-adic map 
\[
T_{\ell} \lambda^2 \colon T_{\ell} \A^2(X_{\overline{k}}) \rightarrow H^3_{\et} (X_{\overline{k}}, \Z_{\ell} (2))_{\tau}
\]
is a $\Gal(\overline{k}/k)$-equivariant isomorphism, where 
$T_{\ell} \A^2(X_{\overline{k}})$ is the $\ell$-adic Tate module and
$H^3_{\et} (X_{\overline{k}}, \Z_{\ell} (2))_{\tau}$ is the torsion-free part of $H^3_{\et} (X_{\overline{k}},\Z_{\ell} (2))$.
\end{enumerate}
\end{prop}

\begin{proof}
(1) follows from \cite[Proposition 10.1]{ACMVdecomposition}.

(2) follows from \cite[Proposition 5.1]{ACMVbloch}.
\end{proof}

\begin{rem}
\label{rem:representative}
\begin{enumerate}
\item 
If $X$ is geometrically rationally chain connected, then the assumptions of the decomposition of the diagonal in (1) and (2) are satisfied by \cite[Remark 2.8]{ACMVdecomposition}. Note that, by \cite[Ch.V, 2.14]{Kollarrationalcurves},
smooth Fano varieties over $k$ are geometrically rationally chain connected.
\item
If $k$ is a subfield of $\C$, then without the assumption of the decomposition of the diagonal, $\Phi^2$ induces an isomorphism on torsion by \cite[Proposition 1.5]{ACMVdecomposition}.
\end{enumerate}
\end{rem}

As an application, we have the following corollary.

\begin{cor}
\label{cor:unramlindep}
Let $R$ be a Henselian discrete valuation ring with the fraction field $K$ and the residue field $k$. 
Let $X$ be a (smooth) Fano variety over $K$.
Then the following are equivalent:
\begin{itemize}
\item The action of $\Gal(\overline{K}/K)$ on $H^3_{\et}(X_{\overline{K}}, \Q_{\ell})$
is unramified for a prime $\ell$ which is invertible in $k$.

\item The action of $\Gal(\overline{K}/K)$ on $H^3_{\et}(X_{\overline{K}}, \Q_{\ell})$
is unramified for any prime $\ell$ which is invertible in $k$.
\end{itemize}
Moreover, if $K$ is of characteristic $0$ and $k$ is a perfect field of characteristic $p >0$, then the above conditions are equivalent to the following:
\begin{itemize}
\item The action of $\Gal(\overline{K}/K)$ on $H^3_{\et}(X_{\overline{K}}, \Q_p)$
is crystalline.
\end{itemize}
\end{cor}

\begin{proof}
By replacing $K$ with its perfection, we may assume that $K$ is a perfect field.
Let $\Ab^2_{X/K}$ be the algebraic representative of $X$ over $K$, which exists by Theorem \ref{defn-thm:representative}.
By Proposition \ref{prop:representativeetale} and Remark \ref{rem:representative}, for any prime number $\ell \in K^{\times}$, we have a $\Gal(\overline{K}/K)$-equivariant isomorphism 
\[
H^1_{\et}(\Ab^2_{X/K, \overline{K}}, \Q_{\ell})^{\vee} \simeq T_{\ell} (\Ab^2_{X/K}) \otimes \Q_{\ell}
\simeq T_{\ell} (\A^2(X_{\overline{K}})) \otimes \Q_{\ell} 
\simeq H^3_{\et} (X_{\overline{K}}, \Q_{\ell} (2)).
\]
Therefore, we can reduce it to the same problem for the first \'{e}tale cohomology of abelian varieties over $K$, which follows from \cite[Theorem 1]{Serre-Tate} and \cite[Theorem 1]{Coleman-Iovita}.
\end{proof}

Next, we recall the existence of the polarization on the algebraic representative (\cite{Benoist-Wittenberg1}, \cite{ACMVdecomposition}).

\begin{defn-thm}
\label{defn-thm:representativepolarization}
Let $k$ be a perfect field, and $X$ a smooth Fano threefold over $k$. Let $\Ab^2_{X/k}$ be the algebraic representative of $X$ over $k$.
\begin{enumerate}
\item 
Suppose that $k$ is a subfield of $\C$.
Then $\Ab^2_{X/k}$ admits a polarization $\Theta_{X}$ such that $(\Theta_{X})_{\C}$ on $(\Ab^2_{X/k})_{\C} \simeq J(X_{\C})$ equal to the theta polarization $\Theta(X_{\C})$.
\item 
Suppose that $X$ is geometrically stably rational, then there exists a symmetric $k$-isomorphism
\[
\Lambda_{X} \colon \Ab^2_{X/k} \rightarrow \Ab^{2 \vee}_{X/k}.
\]
When $k \subset \C$, this comes from the principal polarization $\Theta_{X}$.
\item 
When $k$ is of characteristic $0$, then $\Lambda_{X}$ comes from a principal polarization. 
We denote this polarization also by $\Theta_{X}$.
\item 
Assume $k$ is of characteristic $0$.
Then the isomorphism 
\[
T_{\ell} (\Ab^2_{X/k}) \simeq H^3_{\et} (X_{\overline{k}}, \Z_{\ell}(2))
\]
given by Proposition \ref{prop:representativeetale} and Remark \ref{rem:representative} is compatible with respect to a skew-symmetric form
\[
T_{\ell} \Ab^2_{X/K}  \times T_{\ell} \Ab^2_{X/K} \rightarrow \Z_{\ell} (1)
\]
associated with $\Theta_{X}$ and a skew-symmetric form
\[
H^3_{\et} (X_{\overline{K}}, \Z_{\ell}(2)) \times H^3_{\et} (X_{\overline{K}}, \Z_{\ell} (2)) \rightarrow
\Z_{\ell} (1)
\]
given by the cup product.
\end{enumerate}
\end{defn-thm}
\begin{proof}
(1) The assertion follows from \cite[Example 12.5]{ACMVdecomposition}.

(2) The assertion follows from \cite[Theorem 12.12]{ACMVdecomposition}.

(3) The assertion is essentially proved in \cite{ACMVdecomposition}.
Indeed, by taking a finitely generated subfield $K' \subset K$ where $X$ is defined and choosing an embedding $K' \hookrightarrow \C$, we can reduce it to (1) and (2).
Note that, $\Lambda_{X}$ is compatible with the base change since
it is \emph{distinguished} in the sense of \cite[Section 12]{ACMVdecomposition} (see \cite[Lemma 12.3]{ACMVdecomposition}). 

(4) The assertion can also be reduced to the complex case, where the statement follows from the construction of the Theta divisor on intermediate Jacobians (see \cite[Section 3]{Clemens-Griffiths}).
\end{proof}

\section{Prime Fano threefolds of genus 7 and curves of genus 7 without $g_4^1$}
\label{section:fanovscurve}

\subsection{Basic properties of prime Fano threefolds of genus $7$}

First, we introduce basic properties of prime Fano threefolds of genus $7$ over $\C$.

\begin{prop}
\label{prop:basic}
Let $X$ be a prime Fano threefolds of genus $7$ over $\C$.
Then the following hold:
\begin{enumerate}
\item 
The Hodge diamond of $X$ is given by
\begin{center}
\begin{tabular}{ccccccc}
& & & 1 \\
& & 0 & & 0 \\
& 0 & & 1 & & 0 \\
0 & & 7 & & 7 & & 0 \\
& 0 & & 1 & & 0 \\
& & 0 & & 0 \\
& & & 1
\end{tabular}.
\end{center}
\item 
The dimension of $H^{0} (T_{X/\C})$ is $0$, i.e., there is no infinitesimal automorphism.
Since automorphisms preserve the ample line bundle $-K_{X}$ automatically, the automorphism group $\Aut(X)$ is finite.
\item 
The dimension of $H^1 (T_{X/\C})$ is $18$, i.e., the number of moduli of prime Fano threefolds of genus $7$ is 18.
\item The variety $X$ is rational.
\end{enumerate}
\end{prop}
\begin{proof}
These are well-known results.
(1) and (4) follow from \cite[Corollary 4.4.12 and Theorem 4.5.10]{Iskovskikh-Prokhorov} and the Lefschetz theorem (cf.\ \cite[Remark 3.9]{Kuznetsovspinor}). 
(2) follows from \cite{Prokhorov} or \cite[Corollary 4.3.5]{Kuznetsov-Prokhorov-Shramov}.
Note that, as in \cite[Remark 9.30]{Tanaka2}, the proof of the existence of smooth conics in \cite[Theorem 4.5.10]{Iskovskikh-Prokhorov} contains a gap. Instead, see \cite[Corollary 10.6]{Prokhorovconicnote}.

Let $c_i$ be the $i$-th Chern class of $T_X$.
By the Hirzebruch--Riemann--Roch theorem, we have 
\[
\frac{c_1c_2}{12} = \chi(\cO_X) =1
\]
and
\[
c_3 = \sum (-1)^{i+j} h^{i,j}=-10.
\]
Then, again by the Hirzebruch--Riemann--Roch theorem, we have
\[
\chi(T_X) = -18.
\]
Now (3) follows from (2) and the Kodaira--Akizuki--Nakano vanishing theorem.
\end{proof}

Next, we recall the notion of $g_n^r$ on a smooth proper curve.
\begin{defn}
Let $C$ be a smooth proper curve over an algebraically closed field $k$.
We say $C$ \emph{has $g_n^r$} if there exists a divisor $D$ on $C$ of degree $n$ and a vector subspace $V \subset H^0(C,D)$ of dimension $r+1 \geq 1$.
\end{defn}

\begin{defn}
\label{defn:family}
Let $B$ be a locally Noetherian scheme, and $f \colon X \rightarrow B$ be a smooth projective morphism.
\begin{enumerate}
\item 
We say $f$ is a \emph{(relative) prime Fano threefold of genus 7} when any geometric fiber of $f$ is a prime Fano threefold of genus $7$.
\item 
We say $f$ is a \emph{(relative) curve of genus 7 without $g_4^1$} when any geometric fiber of $f$ is a smooth proper curve of genus $7$ without $g_4^1$.
\end{enumerate}
\end{defn}

\subsection{Spinor tenfolds}
\label{subsection:spin10}

First, we recall the definition of quadratic form over general base schemes.

\begin{defn}
{\cred Let $B$ be a locally Noetherian scheme, 
$V$ a locally free sheaf of rank $10$ on $B$.
A \emph{quadratic form} on $V^{\vee}$ (over $B$) is a locally direct summand morphism
\[
\varphi \colon L \hookrightarrow S^2 V
\]
from a line bundle $L$ on $B$.
We note that $\varphi$ corresponds to the morphism 
\[
q_{\varphi} \colon S_{2} V := (S^2 V)^{\vee} \rightarrow L^{\vee}
\]
which defines a quadratic form in the usual sense on a fiber over each point $b \in B$.
We say $\varphi$ is \emph{non-degenerate} when $q_{\varphi}$ is non-degenerate on a fiber over each point $b \in B$.
}
\end{defn}

{\cred In the following, we recall the definition of spinor tenfolds}.
Let $B$ be a locally Noetherian scheme, 
$V$ a locally free sheaf of rank $10$ on $B$, and $L$ a line bundle on $B$.
{\cred Let $\varphi \colon L \hookrightarrow S^2 V$ be a non-degenerate quadratic form on $V^{\vee}$.}
Let $\OGr_B (V, \varphi)$ be the isotropic Grassmannian, which is a closed subscheme of the 
{\cred $\Gr_{B} (V,5)$}
whose $T$-valued points {\cred for any $B$-scheme $T$} are given by the set of isomorphism classes of surjections $t$ from $V_{T}$ to some locally free sheaf $E$ of rank $5$ on $T$ such that $S^2 t$ kills $\varphi_{T} (L_{T})$.
Then the Stein factorization of the structure morphism $p \colon \OGr_B(V,\varphi) \rightarrow B$ is 
\[
\OGr_B (V, \varphi) \rightarrow B' \rightarrow B,
\]
where $B' \rightarrow B$ is a finite \'{e}tale covering of degree 2.
Indeed, as in \cite{Mukaicurve}, when $B$ is an algebraically closed field, this holds true.
Therefore, even in the general case, every geometric fiber of a finite morphism $B' \rightarrow B$ is just two points, so $B' \rightarrow B$ is finite \'{e}tale of degree $2$.
In the following, we suppose that the following assumption is satisfied.

\begin{quote}
\textbf{Assumption (a):} \ 
{\cred $\pi \colon B' \rightarrow B$ splits, i.e., the map $\cO_B \to \pi_* \cO_{B'}$ splits.}
\end{quote}

{\cred In the following, suppose that $B$ is connected,} we denote connected components of $\OGr_B (V,\varphi)$ by $\Sigma_{+} (V,\varphi)$ and $\Sigma_{-} (V, \varphi)$.
Let $M$ be the relative Pl\"{u}cker class of $\OGr_B (V, \varphi),$ which is an element of $\Pic_{\OGr_{B} (V, \varphi)/B} (B)$.
Moreover, we suppose that the following assumption is satisfied.

\begin{quote}
\textbf{Assumption (b):} \  
There exists an ample line bundle $H_{+}$ on $\Sigma_{+} (V, \varphi)$ such that
$M|_{\Sigma_{+} (V, \varphi)} \simeq H_{+}^{\otimes 2}$. 
\end{quote}

\begin{lem}
\label{lem:pmindep}
Let $B$, $V$, $L$, and $\varphi$ be as above.
In particular, we suppose that $B$ is connected and the Assumptions (a) and (b) are satisfied. 
Then $M|_{\Sigma_{-} (V, \varphi)}$ can be also divided by $2$ in $\Sigma_{-} (V, \varphi)$.
\end{lem}

\begin{proof}
Let $\OFl_{B}((V,\varphi),5,4)$ be the orthogonal flag variety parametrizing orthogonal flags $V \twoheadrightarrow L \twoheadrightarrow L'$ with $\rank L =5$ and $\rank L'=4$.
Let $\OGr_{B} ((V,\varphi), 4)$ be the orthogonal Grassmannian variety parametrizing rank 4 isotropic quotients.
The natural morphism 
\[
p\colon \OFl_{B}((V,\varphi),5,4) \rightarrow \OGr_{B}(V,\varphi)
\]
is equal to the projective bundle
\[
p \colon
\P_{\OGr_{B}(V,\varphi)} (\mathcal{U}^{\vee})
\rightarrow
\OGr_{B}(V,\varphi),
\]
where $\mathcal{U}$ is rank 5 the universal quotient bundle.
By Assumption (a), $p$ decomposes into 
\[
p_{\pm} \colon
\P_{\Sigma_{\pm}(V,\varphi)} (\mathcal{U}_{\pm}^{\vee}) \rightarrow \Sigma_{\pm}(V,\varphi),
\]
where $\mathcal{U}_{\pm}$ is the restriction $\mathcal{U}|_{\Sigma_{\pm}(V,\varphi)}$.
Then the natural morphism
\[
\P_{\Sigma_{\pm}(V,\varphi)} (\mathcal{U}_{\pm}^{\vee})
\hookrightarrow 
\OFl_{B}((V,\varphi),5,4) \rightarrow \OGr_{B}((V,\varphi),4)
\]
is an isomorphism.
We put
\[
H_{-} := \frac{1}{2} M  \in \Pic_{\Sigma_{-}(V,\varphi)/B} (B).
\]
Here, the inclusion follows from the proof of \cite[Corollary 2.3]{Kuznetsovfamily}. Note that, the Kodaira vanishing for geometric fibers of $\Sigma_{\pm}$ holds by \cite[Proposition 2 and Theorem 2]{Mehta-Ramanathan} (cf.\ Lemma \ref{lem:mukaincohomology}).
Then by construction, we have $2H_{-} \in \Pic (\Sigma_{-} (V,\varphi))$.
Note that, by the above argument, we have two distinct fibrations
\[
q_{\pm} \colon
\OGr_{B} (V, \varphi) \rightarrow \Sigma_{\pm} (V, \varphi)
\]
induced by $p_{\pm}$.
Since $\OGr((V,\varphi),4)$ (resp.\ $\Sigma_{\pm}(V,\varphi)$) has a relative Picard rank 2 (resp.\ $1$), by restricting on a fiber of $q_{\pm}$ and using an adjunction, we can show that
\[
\cO_{\OGr((V,\varphi),4)} (-K_{\OGr((V,\varphi),4)/B} )= 5 q_{+}^{\ast}(H_{+}) + 5 q_{-}^{\ast} (H_{-}).
\]
Combining with $H_+ \in \Pic (\Sigma_{+}(V,\varphi))$ and $2 H_- \in \Pic(\Sigma_{-}(V,\varphi))$, we have $H_{-} \in \Pic (\Sigma_{-})$, and it finishes the proof.
\end{proof}

We denote the half of $M|_{\Sigma_{-}(V,\varphi)}$ given by Lemma \ref{lem:pmindep} by $H_{-}$.
{\cred By \cite[Section 1]{Mukaicurve}, $H_{+}$ and $H_{-}$ are very ample on each geometric fiber.
Since $H^1 (H_{+}) = H^1 (H_{-})=0$ by the Kodaira vanishing theorem (\cite[Proposition 2 and Theorem 2]{Mehta-Ramanathan}) 
(cf.\ Lemma \ref{lem:mukaincohomology}),
$S_{\pm} := p_{\ast} H_{\pm}$ is a locally free sheaf by the cohomology and base change theorem, and
they define embeddings
\[
\Sigma_{\pm} (V, \varphi) \hookrightarrow \P_{B} ( S_{\pm} ).
\]
}
{\cred In this setting, the projective dual $\Sigma_{\pm} (V, \varphi)^{\vee}$ of $\Sigma_{\pm} (V, \varphi) \hookrightarrow \P_{B} ( S_{\pm} )$ is defined as the image of the composition 
\[
\P_{\Sigma_{\pm} (V, \varphi)} (N_{\Sigma_{\pm} (V, \varphi)/\P_{B}(S_{\pm})}) \hookrightarrow
\P_{\P_{B}(S_{\pm})} (T_{\P_{B}(S_{\pm})/B}) \rightarrow \P_{B} (S_{\pm}^{\vee}),
\]
where the first map is induced by the normal bundle exact sequence, and the second map is induced by the Euler sequence.
By the above argument using the cohomology and the base change theorem, for any geometric point $b$ in $B$, the base change $\Sigma_{\pm} (V, \varphi)^{\vee} \times_{B} b$ is isomorphic to the projective dual of 
$\Sigma_{\pm} (V|_{b}, \varphi|_{b}) \hookrightarrow \P_{b} ( S_{\pm} |_{b})$.
}

\begin{prop}[cf.\ {\cite[Proposition 2.7]{Mukaicurve}}]
\label{prop:spinorduality}
Let $B$, $V$, $L$, and $\varphi$ be as above. In particular, we suppose that $B$ is connected and the Assumptions (a) and (b) are satisfied. 
Then the projective dual $\Sigma_{+}^{\vee}$ of $\Sigma_{+} \hookrightarrow \P_{B} (S_{+})$ is isomorphic to 
$\Sigma_{-} \hookrightarrow \P_{B} (S_{-})$.
\end{prop}

\begin{proof}
We use the notation in Lemma \ref{lem:pmindep}.
We have the diagram
\[
\begin{tikzcd}
&   \P_{\Sigma_{+}(V, \varphi)}(\mathcal{U}_{+}^{\vee}) \ar[d] \ar[r, "\sim"] & \OGr_{B} ((V,\varphi),4) \ar[r, "\sim"] & \P_{\Sigma_{-}(V, \varphi)}(\mathcal{U}_{-}^{\vee}) \ar[d] & \\
\P_{B} (S_{+}) & \Sigma_{+} (V, \varphi)\ar[l,hook'] && \Sigma_{-} (V,\varphi) \ar[r,hook] & \P_{B}(S_{-}) 
\end{tikzcd}
\]
Since $\mathcal{U}_{\pm}^{\vee}$ is isomorphic to 
the twisted normal bundle $N_{\Sigma_{\pm}(V,\varphi)/\P_{B}(S_{\pm})} (-2H_{\pm})$ by \cite[Corollary 2.4]{Mukaicurve} (see also the proof of Lemma \ref{lem:mukaincohomology} (9)) and the relative Picard rank of $\OGr_{B}((V,\varphi),4)$ is $2$, the proposition follows from the above diagram.
\end{proof}

\begin{prop}
\label{prop:assumptionab}
Let $V$ be a $10$-dimensional vector space over a field $k$ of characteristic different from $2$. 
Let $q \colon S_{2} V^{\vee} \rightarrow k$.
be a {\cred non-degenerate} quadratic form on $V^{\vee}$.
We denote the corresponding morphism $k \hookrightarrow S^2 V$ by $q_\varphi$.
\begin{enumerate}
\item 
Assumption (a) for $(V, \varphi_{q})$ is satisfied if and only if $\disc q $ is trivial.
\item
Suppose that Assumption (a) for $(V, \varphi_{q})$ is satisfied.
{\cora Moreover, suppose that $k$ is of characteristic $0$.}
Then Assumption (b) for $(V, \varphi_{q})$ is satisfied if and only if $C_0 (q) \simeq M_{16} (k) \times M_{16} (k).$
\end{enumerate}
\end{prop}

\begin{proof}
(1) The assertion follows from \cite[p.366, footnote]{Wang}.

(2) First, suppose that Assumption (b) is satisfied.
{\cora Then the action of $\Spin (V, q) (k)$ on $\Sigma_{+}$ induces a projective representation of 
$\Spin (V,q)$ on $S_{+} = H^0 (H_{+})$. 
Since $\Spin (V,q)$ is a semi-simple simply connected algebraic group in characteristic $0$, a $\G_{m}$-central extension of $\Spin (V,q)$ is split.
Therefore, we can take a lift $\rho \colon \Spin (V,q) \rightarrow \GL (S_{+})$.
Moreover, the base change $\rho_{\overline{k}}$ is the half-spinor representation of $\Spin (V_{\overline{k}}, q_{\overline{k}})$ (cf. \cite[Section 1]{Mukaicurve}).
By taking a derivation, we have a morphism of Lie algebras
\[
d \rho \colon \mathrm{Cl}_{2} (q) \rightarrow \End (S_{+}).
\]
Note that $\rho_{\overline{k}}$ is induced by the half-spinor Clifford representation $\rho' \colon C_{0}(q_{\overline{k}}) \rightarrow \End (S_{+, \overline{k}})$ and the base change $d\rho_{\overline{k}}$ equals to $\rho'|_{\mathrm{Cl}_{2} (q_{\overline{k}})}$.
Therefore, since $\rho'$ is a $\overline{k}$-algebra homomorphism and $k \oplus \mathrm{Cl}_{2}(q_{\overline{k}})$ generates $C_{0} (q)$ as a $k$-algebra, $\rho'$ is also defined over $k$.
}
In particular, by \cite[Chapter V, Theorem 2.5]{Lam}, we have $C_{0} (q) \simeq M_{16} (k) \times M_{16} (k).$
Next, suppose that we have $C_{0} (q) \simeq M_{16} (k) \times M_{16} (k).$
In this case, by \cite[Chapter V, Theorem 2.5]{Lam}, we have
$C(q) \simeq M_{32} (k) $.
Therefore, there exists a representation of $C(q)$ on a $32$-dimensional vector space $S$ whose restriction to $C_{0} (q)$ decomposes into $S_{+}^{\vee} \bigoplus S_{-}^{\vee}$, where $S_{+}$ and $S_{-}$ are $16$-dimensional vector spaces over $k$.
For any $v \in V_{\overline{k}}$, let $\varphi_{v} \in \End (S_{\overline{k}})$ be a corresponding endomorphism.
For any Lagrangian subspace $U \subset V_{{\cred \overline{k}}}$, 
let $s_{U}$ be a pure spinor corresponding to $U$, i.e., $s_{U}\in S_{+ \overline{k}} \cup S_{- \overline{k}}$ such that $\varphi_{u} (s_{U}) =0$ for any $u \in U.$
Then the map
\[
\Sigma_{+}(V,\varphi_{q})_{\overline{k}} \rightarrow \P (S_{+ \overline{k}})\simeq \P^{15}_{\overline{k}}
\]
defined by
\[
U \mapsto [s_{U}]
\]
is an embedding (cf.\ \cite[Section 1]{Mukaicurve}).
By the Galois descent, this map descends to an embedding
\[
\Sigma_{+} (V,\varphi_q) \hookrightarrow \P (S_{+}).
\]
The hyperplane section of $\P (S_{+})$ defines an ample class $H_{+}$, as desired.
\end{proof}

\subsection{Dual construction as a family}

Recently, Kuznetsov \cite{Kuznetsovfamily} studied relative prime Fano threefolds of genus $7$ over $\Q$-schemes.
In this section, we confirm that such a family corresponds to a relative curve of genus $7$ without $g_4^1$ (although it might be well-known to experts).

We state the curve analogue of Kuznetsov's result in the following.

\begin{prop}
\label{prop:curvesigmafamily}
Let $B$ be a {\cred connected locally Noetherian} {\cora $\Q$-}scheme, and $f \colon C \rightarrow B$ a curve of genus 7 without $g_4^1.$
Let
\[
\iota \colon
C \hookrightarrow \P_{B}(f_{\ast} \cO_{C} (K_{C/B}))
\]
be the canonical embedding over $B$, {\cred $\mathcal{I}_{C}$ the ideal sheaf of $C$ in $\P_{B}(f_{\ast} \cO_{C} (K_{C/B}))$,} and $p$ the projection from $\P_{B}(f_{\ast} \cO_{C} (K_{C/B}))$ to $B$.
\begin{enumerate}
\item
The coherent sheaf $W_C :=p_{\ast} (\mathcal{I}_{C} (2))$ is a locally free sheaf of rank 10 on $B$, and the kernel $F_C$ of a natural map
\[
S^2 W_C \rightarrow p_{\ast} (\mathcal{I}_{C}^{2} (4)) 
\]
{\cred defines a non-degenerate quadratic form $\varphi_{C} \colon F_{C} \hookrightarrow S^2 W_{C}$ on $W_{C}^{\vee}$.}
\item
We put 
\[
E_C:= (N_{C/\P_{B}(f_{\ast}\cO_{C} (K_{C/B}))}^{\ast} \otimes \cO_{C} (2 K_{C/B})),
\]
which is a locally free sheaf on $C$ of rank $5$.
Then the morphism $f^{\ast} W_C \rightarrow E_C$ is a surjection and
defines a morphism 
\[
\iota_{C} \colon C \rightarrow \Gr_B (W_C,5) 
\]
which is an embedding factoring through $\OGr_B (W_C, \varphi_{C})$.
\item
The orthogonal Grassmannian $\OGr_B (W_C, \varphi_{C})$ satisfies the Assumptions (a) and (b) in Subsection \ref{subsection:spin10}. In particular, $\OGr_B (W_{C}, \varphi_{C})$ decomposes into two connected components $\Sigma_{C,+}$ and $\Sigma_{C,-}$, and we have locally free sheaves $S_{C,+}$ and $S_{C,-}$ on $B$ with an embedding
\[
\Sigma_{C, \pm} \hookrightarrow \P_{B}(S_{C,\pm}).
\]
Moreover, $\iota_{C}$ factors through one connected component $\Sigma_{C,+}$, and there exists a locally free quotient $S_{C,+} \twoheadrightarrow A_{C}$ of rank $7$ such that 
\[
C = \Sigma_{C,+} \cap \P_{B} (A_{C}) \subset \P_{B} (S_{C,+}).
\]
\item 
Let $f' \colon C' \rightarrow B$ be a curve of genus $7$ without $g_4^1$, and $g \colon C \rightarrow C'$ be an isomorphism over $B$.
Then there exists an isomorphism $\widetilde{g} \colon \P_{B} (S_{C, +}) \simeq \P_{B} (S_{C',+})$ such that $\widetilde{g}$ restricts to $\P_{B}(A_{C}) \simeq \P_{B}(A_{C'})$ and induces isomorphisms $\Sigma_{C,+} \simeq \Sigma_{C',+}$ and $g \colon C \simeq C'$.
\end{enumerate}
\end{prop}

\begin{proof}
(1) {\cred Note that the sheaf $\mathcal{I}_{C} (2)$ is flat over $B$ since so are $C$ and $\P_{B} (f_* \cO_{C}(K_{C/B}))$.
We have 
$H^1 (I_{C_{k(b)}}(2)) =0$
for any point $b \in B$ by Noether's theorem, where $I_{C_{k(b)}}$ is the ideal sheaf of the closed embedding $\iota \times_{B} \Spec k(b)$.
Therefore, the coherent sheaf $W$ is a locally free sheaf of rank 10 by the cohomology and base change theorem and Noether's theorem (see \cite[(3.1)]{Mukaicurve}).}
{\cora Next, we shall show that $F_{C}$ is a locally free sheaf of rank 1. We can show that the map $S^2(W_{C}\otimes_{\cO_B} k(b)) \rightarrow H^0(I_{C_{k(b)}}^2(4))$
is surjective with $1$-dimensional kernel, and $H^1 (I_{C_{k(b)}}^2(4)) =0$ by 
\cite[Main theorem, Section 4, and Corollary 5.3]{Mukaicurve} and \cite[Corollary 4.4]{Kuznetsovspinor}.
Note that since the conormal bundle $\mathcal{I}_{C}/\mathcal{I}_{C}^2$ is locally free on $C$, $\mathcal{I}_{C}^2$ is flat over $B$.
Therefore, by the cohomology and base change theorem, $p_*(\mathcal{I}^2_{C}(4))$ is locally free.
Since $S^2 W_{C} \rightarrow p_*(\mathcal{I}^2_{C} (4))$ is surjective, 
$F_{C}$ is torsion-free and locally generated by one section, which means that $F_{C}$ is a rank $1$ locally direct summand of $S^2 W_{C}$.
Moreover, $F_{C}$ defines non-degenerate quadratic form $\varphi_{C}$ since so it is on any fiber (\cite[Theorem 4.2]{Mukaicurve}).
}

(2) The morphism $p^{\ast} W_{C} \rightarrow E_{C}$ is surjective by \cite[Section 3]{Mukaicurve}.
Moreover, $\iota_{C}$ is an embedding by \cite[Proposition 3.3]{Mukaicurve}, and $\iota_{C}$ factors through $\OGr_B (W_{C}, \varphi_{C})$ by the definition of $\varphi_{C}$ (cf.\ \cite[Proof of Theorem 0.4]{Mukaicurve}).

(3) We denote the structure morphism of $\OGr_B (W_{C}, \varphi_{C})$ by $p_{\OGr}$.
Since $C \hookrightarrow \OGr_B (W_{C}, \varphi_{C})$ induces the morphism
\[
p_{\OGr \ast} \cO_{\OGr_B (W_{C}, \varphi_{C})} \rightarrow f_{\ast} \cO_{C} = \cO_{B}
\]
which is a section of $\cO_{B} \rightarrow p_{\OGr \ast} \cO_{\OGr_B (W, \varphi_{C})}$, Assumption (a) is satisfied. 
We denote the component containing $C$ by $\Sigma_{C,+}$.
Let $H_{C} \in \Pic_{\Sigma_{C,+}/B} (B)$ be the relative fundamental class (cf.\ \cite[Corollary 2.6]{Kuznetsovfamily}).
Let $\beta$ be the image of $H$ via the natural map 
\[
\Pic_{\Sigma_{C, +}/B} (B) \rightarrow H^2_{\et} (B, \G_{m}).
\]
Since any geometric fiber of $\Sigma_{C,+}/B$ is a Fano variety of Picard rank 1 and index $8$ {\cred by Proposition \ref{prop:MukaivssplitMukai} (see also \cite[Proposition 2.1]{Mukaicurve})},
we have
$\beta^8 =1$
by \cite[Corollary 2.17]{Kuznetsovfamily}.
We denote the structure morphism of $\Sigma_{C,+}$ by $p^{+}$, and the ideal sheaf of $C$ in $\Sigma_{C,+}$ by $\mathcal{I}^{+}$.
{\cred We consider
\[
p^+_* (\mathcal{I}^{+} \otimes H_{C})
\]
as a $\beta$-twisted sheaf (cf.\ \cite[Subsection 2.2]{Kuznetsovfamily}).
Since $H^1(\mathcal{I}^{+} \otimes H_{C}|_{\Sigma_{C,+, k(b)}})=0$ for any point $b\in B$ by Lemma \ref{lem:mukaincohomology}, this sheaf is a $\beta$-twisted locally free sheaf.
Moreover, this sheaf is of rank $9$ (cf.\ \cite[Proposition 2.2]{Mukaicurve}), and we have $\beta^{9} =1$.
Therefore, we have $\beta=1$ and $H_{C}$ comes from an actual line bundle, i.e., Assumption (b) is satisfied.
The sheaf
\[
S_{C,+} := p_{\ast}^+ H_{C}
\]
is a locally free sheaf by the Kodaira vanishing (\cite[Proposition 2 and Theorem 2]{Mehta-Ramanathan}) and the cohomology and base change theorem.
Moreover, this sheaf is of rank 16 (see \cite{Mukaicurve}, see also Lemma \ref{lem:mukaincohomology}).
We put
\[
A_{C} := \coker ( p_{\ast}^+ (\mathcal{I}^{+} \otimes H_{C}) \rightarrow  p_{\ast}^+ H_{C})
\]
which is a locally free sheaf of rank $7$.
Then we have the desired equality by definition.}

(4) The morphism $g$ induces a natural isomorphism
\[
f'_{\ast} \cO_{C'} ( K_{C'/B} )\simeq f'_{\ast} g_{\ast} g^{\ast} \cO_{C'}( K_{C'/B}) \simeq f_{\ast} \cO_{C}( K_{C/B}),
\]
so we have an isomorphism
\[
\P_{B} (f_{\ast} \cO_{C}( K_{C/B})) \simeq \P_{B} (f'_{\ast} \cO_{C'}( K_{C'/B})).
\]
By this isomorphism, we have isomorphisms $W_C \simeq W_{C'}$ and $F_C \simeq F_{C'}$ compatible with $\varphi_{C}$ and $\varphi_{C'}$.
In particular, we obtain an isomorphism \[ \Gr_B (W_{C}, 5) \simeq \Gr_B (W_{C'}, 5)
\]
compatible with $\iota_{C}$ and $g$.
This isomorphism clearly restricts to $g_{+} \colon \Sigma_{C,+} \simeq \Sigma_{C',+}$ and $g_+$ induces an isomorphism
\[
\P_{B} (p_{C, +,\ast} H_{C}) \simeq \P_{B}( p_{C', +, \ast} H_{C'}),
\]
since $g_{+}^{\ast} (H_{C'}) \otimes H_{C}^{-1}$ is relatively trivial over $B$. 
This isomorphism is no other than the desired isomorphism $\widetilde{g}$.
\end{proof}

{\cora
\begin{rem}
\label{rem:basedvr}
In the above proposition, we assume that the base scheme $B$ is a $\Q$-scheme since we use the surjectivity of $S^2(W_{C}\otimes_{\cO_B} k(b)) \rightarrow H^0(I_{C_{k(b)}}^2(4))$ and the vanishing of $H^1 (I_{C_{k(b)}}^2(4))$. We do not know whether the same hold true also in positive characteristic.

However, if $B$ is a regular connected scheme of dimension less than 2, (which is not necessarily $\Q$-scheme,) then all the statements in the above proposition hold.
Only the part showing that $F_{C}$ is a locally direct summand of $S^2 W_C$ is non-trivial. To show this, we may assume that $B$ is a spectrum of a discrete valuation ring.
The kernel of $S^2(W_{C}\otimes_{\cO_B} k(b)) \rightarrow H^0(I_{C_{k(b)}}^2(4))$ is 1-dimensional as in the above proof.
Moreover, the sheaf $p_* \mathcal{I}^2_{C}(4)$ is torsion-free since it is a subsheaf of $p_* \mathcal{I}^2_{C}(m)$ with sufficiently large $m$, which is a locally free sheaf by the cohomology and base change theorem.
Thus the image of $S^2 W_{C} \rightarrow p_*(\mathcal{I}^2_{C} (4))$ is torsion-free and therefore locally free.
Moreover, $F_{C} \subset S^2 W_{C}$ is also torsion free and hence $F_C$ is a rank $1$ locally direct summand of $S^2 W_{C}$.
\end{rem}
}

\begin{prop}
\label{prop:Xsigmafamily}
Let $B$ be a {\cred connected locally Noetherian} $\Q$-scheme, and $f\colon X \rightarrow B$ a prime Fano threefold of genus $7$.
Let 
\[
\iota \colon
X \hookrightarrow \P_{B}( f_{\ast} \cO_{X} (- K_{X/B}))
\]
be the anti-canonical embedding over $B$ (which exists by cohomology and base change theorem and Definition-Proposition \ref{defn-prop:Xtocurveflop}.(1)),
{\cred $\mathcal{I}_{X}$ an ideal sheaf of $X$ in $\P_{B}(f_{\ast} \cO_{X} (-K_{X/B}))$,}
and $p$ a projection from
$\P_{B}(f_{\ast} \cO_{X}(-K_{X/B}))$ to $B$.
\begin{enumerate}
\item
The coherent sheaf $W_{X}:= p_{\ast} (\mathcal{I}_{X} (2))$ is a locally free sheaf of rank 10 on $B$,
and the kernel $F_{X}$ of a natural map
\[
S^{2} W_{X} \rightarrow p_{\ast} (\mathcal{I}_{X}^{2} (4))
\]
{\cred defines a non-degenerate quadratic form $\varphi_{X} \colon F_{X} \hookrightarrow S^2 W_{X}$ on $W_{X}^{\vee}$.}
\item 
We put 
\[
E_{X} := N_{X/\P_{B}(f_{\ast}\cO_{X}(-K_{X/B}))}^{\ast} \otimes \cO_{X}(-2 K_{X/B}),
\]
which is a locally free sheaf of rank $5$.
The morphism $f^{\ast} W_{X} \rightarrow E_{X}$ defines a morphism 
\[
\iota_{X} \colon X \rightarrow \Gr_B (W_{X}, 5) 
\]
which is an embedding factoring through $\OGr_B(W_{X}, \varphi_{X})$.
\item
The similar statements as in Proposition \ref{prop:curvesigmafamily} (3) and (4) holds true.
\end{enumerate}
\end{prop}

\begin{proof}
It is similar to the proof of Proposition \ref{prop:curvesigmafamily}.
See \cite[Proposition 5.11 and Remark 5.12]{Kuznetsovfamily}.
Since we should reduce it to the result over {\cred algebraically closed field in characteristic $0$} (cf.\ \cite[Section 6]{MukaiSugakuShintenkaiEnglish}), we need to work over $\Q$-schemes (see Proposition \ref{prop:Mukainfoldsigmafamily} for certain generalization).
\end{proof}

\begin{defn-prop}
\label{defn-prop:dualvariety}
\begin{enumerate}
\item 
Let $B$ be a {\cred connected locally Noetherian} $\Q$-scheme, and $f\colon C \rightarrow B$ a curve of genus 7 without $g_4^1$.
We take an embedding 
\[
C = \Sigma_{C,+} \cap \P_{B} (A_{C}) \hookrightarrow \Sigma_{C, +} \hookrightarrow \P_{B} (S_{C, +})
\]
given in Proposition \ref{prop:curvesigmafamily}.
Let $\Sigma_{C, +}^{\vee} \hookrightarrow \P_{B} (S_{C,+}^{\vee})$ be the projective dual of 
$
\Sigma_{C,+}\hookrightarrow \P_{B} (S_{C,+}).
$
Then 
\[
\beta(C) := \Sigma_{C,+}^{\vee} \cap \P_{B} (A_{C}^{\perp}) \subset \P_{B} (S_{C,+}^{\vee})
\]
is a prime Fano threefold of genus 7.
{\cred Even when the base scheme $B$ is \emph{non-connected} locally Noetherian $\Q$-scheme, we can define $\beta (C)$ by discussing each connected component separately.
We call $\beta(C)$ the \emph{dual} of $C$.
}
\item
Let $B$ be a {\cred connected locally Noetherian} $\Q$-scheme, and $f\colon X \rightarrow B$ a prime Fano threefold of genus 7.
We take an embedding 
\[
X = \Sigma_{X, +} \cap \P_{B} (A_{X})
\hookrightarrow \Sigma_{X, +} \hookrightarrow \P_{B} (S_{X, +})
\]
given in Proposition \ref{prop:Xsigmafamily}.
Then
\[
\Gamma(X) := \Sigma_{X,+}^{\vee} \cap \P_{B} (A_{X}^{\perp}) \subset \P_{B} (S_{X,+}^{\vee})
\]
is a curve of genus 7 without $g_4^1$.
{\cred Even when the base scheme $B$ is \emph{non-connected} locally Noetherian $\Q$-scheme, we can define $\Gamma (X)$ by discussing each connected component separately.
We call $\Gamma(X)$ the \emph{dual} of $X$.}
\end{enumerate}
\end{defn-prop}

\begin{proof}
(1) Since the morphism 
\[
\P_{\Sigma_{C,+}} (N_{\Sigma_{C,+}/\P_{B}(S_{C,+})})
\rightarrow \P_{B} (S_{C,+}^{\vee})
\]
is a projective bundle onto the image, the projective dual $\Sigma_{C,+}^{\vee}$ is flat over $B$, so $\Sigma_{C,+}^{\vee}$ is smooth over $B$ by Proposition \ref{prop:spinorduality}.
Then $\beta(C)$ is smooth over $B$ by \cite[Theorem 2.5.8]{Fu} and \cite[Lemma 3.3]{Kuznetsovspinor}.
Note that, for any geometric point $b$ of $B$, $\cO(-K_{\Sigma_{C,+,b}^{\vee}/ \kappa (b)}) = \cO(8)$, where $\cO(1)$ is a hyperplane section of $\P_{\kappa(b)} (S_{C,+ (\kappa(b))}^{\vee})$.
Indeed, this follows from Proposition \ref{prop:spinorduality} and \cite[(3.7)]{Kuznetsovfamily} in characteristic $0$, and can be verified by taking a lift even in positive characteristic.
Then the computation of index and degree follows from the adjunction.

(2) The assertion follows from the same argument as above and \cite[Proposition 2.2]{Mukaicurve}.
\end{proof}

\begin{prop}
\label{prop:dualcompatiblejacobian}
Let $k$ be a field of characteristic $0$.
Let $f \colon X \rightarrow \Spec k$ be a prime Fano threefold of genus $7$.
Let \[
 C = \Gamma (X) \rightarrow \Spec k
\]
be as in Proposition \ref{defn-prop:dualvariety}.
Then we have an isomorphism
\[
(\Ab_{X/k}^2, \Theta_{X}) \simeq (J_{C/k},\Theta_{C}).
\]
In particular, we have an isomorphism of $\Gal(\overline{k}/k)$-representations
\[
H^3_{\et} (X_{\overline{k}}, \Q_{\ell} (2)) \simeq H^1_{\et} (C_{\overline{k}}, \Q_{\ell} (1)).
\]
\end{prop}
\begin{proof}
This follows from the proof of \cite[Proposition 8.4]{KuznetsovProkhorov} and \cite[Subsection 6.2]{Kuznetsovhyperplane}.
We include the proof for the sake of completeness.
Let $X \hookrightarrow \Sigma_{+} \subset\P(S_{+})$ and $C \hookrightarrow \Sigma_{-} \subset \P(S_{-})$ be as in Definition-Proposition \ref{defn-prop:dualvariety}, i.e., we have $S_{-} = S_{+}^{\vee}$ and $\Sigma_{-} = \Sigma_{+}^{\vee}$.
Let $\mathcal{U}_{+}$ (resp.\ $\mathcal{U}_-$) be the universal rank 5 quotient bundle on $\Sigma_{+}$ (resp.\ $\Sigma_{-}$).
Let $p_{+}$ (resp.\ $p_-$) be the projection to $\Sigma_+$ (resp.\ $\Sigma_-$) from $\Sigma_{+} \times_{k} \Sigma_{-}$. 
Then there exists a natural morphism 
\[
e \colon p_+^* \cO_{\Sigma_+} \otimes p_-^* \mathcal{U}_-^{\vee} 
\rightarrow 
p_+^* \mathcal{U}_+ 
\otimes p_-^* \cO_{\Sigma_{-}}
\]
(see \cite[\S1 and Subsection 6.2]{Kuznetsovhyperplane} for the precise definition).
We put
\[
\mathcal{E} := \coker e |_{X \times C}.
\]
By \cite[Theorem 2.46 and Theorem 4.19]{Kuznetsovhyperplane}, the Fourier-Mukai functor with the Fourier-Mukai kernel $\mathcal{E}_{\overline{k}}$ gives a fully faithful functor 
\[
\Phi_{\mathcal{E}_{\overline{k}}} \colon D^{\mathrm{b}} (C_{\overline{k}}) \rightarrow D^{\mathrm{b}} (X_{\overline{k}})
\]
which induces a semi-orthogonal decomposition \footnote{We can also show that $D^{\mathrm{b}} (X) = \langle D^{\mathrm{b}} (C), \cO_{\Sigma+}, \mathcal{U}_+ \rangle$. Indeed, the fully faithfulness of $\Phi_{\mathcal{E}}$ can be reduced to that of $\Phi_{\mathcal{E}_{\overline{k}}}$ by using \cite[Proposition 1.49 and Proposition 3.17]{Huybrechts}, and we have the one inclusion $\supset$. The other inclusion can be reduced to the case where $k=\overline{k}$ by \cite[Lemma 1.61]{Huybrechts}.}
\[
D^{\mathrm{b}} (X_{\overline{k}}) =
\langle
D^{\mathrm{b}} (C_{\overline{k}}), \cO_{\Sigma_{+ \overline{k}}}, \mathcal{U}_{+, \overline{k}} 
\rangle.
\]

Then by \cite[Proposition 8.4]{KuznetsovProkhorov}, the second Chern class $c_2 (\mathcal{E})$ gives maps
\[
\begin{tikzcd}
\A^2(X_{\overline{k}}) \arrow[r,yshift=0.6ex, "c_2(\mathcal{E})"] & \A^1(C_{\overline{k}}) \arrow[l,yshift=-0.5ex,"c_2(\mathcal{E})"]\\
\end{tikzcd}
\]
inducing an isomorphism 
\[
\varphi \colon (\Ab^2_{X/k, \overline{k}}, \Theta_{X, \overline{k}}) \simeq (J_{C/k, \overline{k}}, \Theta_{C, \overline{k}})
\]
(see also Definition-Theorem \ref{defn-thm:representative}).
Recall that $\A^2 (X_{\overline{k}}) \rightarrow \Ab^{2}_{X/k} (\overline{k})$ is a $\Gal(\overline{k}/k)$-equivariant isomorphism by Proposition \ref{prop:representativeetale} 
and Remark \ref{rem:representative}, 
and it is well-known that the same holds for $\A^1(C_{\overline{k}})$ and $J_{C/k}$.
Since $c_{2} (\mathcal{E})$ is $\Gal(\overline{k}/k)$-invariant, $\varphi$ descends to a desired isomorphism
\[
(\Ab_{X/k}^2, \Theta_{X}) \simeq (J_{C/k},\Theta_{C}).
\]

The later statement follows from Proposition \ref{prop:representativeetale} and Remark \ref{rem:representative}.
\end{proof}

\begin{defn}
\begin{enumerate}
\item 
We define the moduli stack of smooth curves $\mathcal{M}_{\mathrm{sm}}$ over $\Z$ as the fibered category with objects relative smooth projective curves $f \colon X \rightarrow B$ and morphisms defined by Cartesian diagrams.
We define $\mathcal{M} \subset \mathcal{M}_{\mathrm{sm}}$ over $\Z$ as the substack consisting of objects $X \rightarrow B$ that are relative curves of genus 7 without $g_4^1$. 
\item 
We define the moduli stack of Fano schemes $\mathcal{F}_{\mathrm{Fano}}$ over $\Z$ as the fibered category with objects Fano schemes $f \colon X \rightarrow B$ and morphisms defined by Cartesian diagrams.
We define $\mathcal{F} \subset \mathcal{F}_{\mathrm{Fano}}$ as the substack consisting of objects $X \rightarrow B$ that are prime Fano threefolds of genus $7$.
\item 
We define the moduli stack of principally polarized abelian schemes $\mathcal{A}$ over $\Z$ as the fibered category with objects principally polarized abelian schemes $(f \colon X \rightarrow B, \Lambda)$.
\end{enumerate}
\end{defn}

\begin{thm}
\label{thm:moduliarithmetictorelli}
\begin{enumerate}
\item
{\cora
The construction $\beta$ gives a morphism of stacks
\[
\beta \colon \mathcal{M}_{\Q} \rightarrow \mathcal{F}_{\Q}
\]
over $\Q$.
}
\item 
The construction $\Gamma$ gives a morphism of stacks
\[
\Gamma \colon \mathcal{F}_{\Q} \rightarrow \mathcal{M}_{\Q}
\]
over $\Q$.
\item 
Morphisms $\beta$ and $\Gamma$ are mutually inverse isomorphisms.
In particular, we have an isomorphism of stacks
\[
\mathcal{M}_{\Q} \simeq \mathcal{F}_{\Q}.
\]
\item
Let $j \colon \mathcal{M}_{\Q} \rightarrow \mathcal{A}_{\Q}$ be a period map defined by Jacobian of curves, which is an immersion.
Let $k$ be a field of characteristic $0$.
Let $f \colon X \rightarrow \Spec k$ be an object of $\mathcal{F}_{\Q}$ over $\Spec k$.
Then $j \circ \Gamma (f)$ is isomorphic to $(\Ab^{2}_{X/k}, \Theta_{X})$.
\end{enumerate}
\end{thm}

\begin{proof}
(1) and (2) follows from Propositions \ref{prop:spinorduality}, \ref{prop:curvesigmafamily}, \ref{prop:Xsigmafamily}, and \ref{defn-prop:dualvariety}.
{\cora Note that, for any isomorphism $X \simeq Y$ in $\mathcal{M}_{\Q} (B)$, we can associate $\beta (X) \simeq \beta (Y)$ in $\mathcal{A}_{\Q} (B)$ by Proposition \ref{prop:curvesigmafamily} (4), and the same holds for $\Gamma$ by Proposition \ref{prop:Xsigmafamily} (3).
}
{\cora Then, by the construction, we have}
$
\Gamma \circ \beta \simeq \mathrm{id}
$
and 
$
\beta \circ \Gamma \simeq \mathrm{id},
$
hence we have (3).
(4) follows from Proposition \ref{prop:dualcompatiblejacobian}. 
\end{proof}

\begin{rem}
\label{rem:separatedness}
\begin{enumerate}
\item
$\mathcal{M}$ and $\mathcal{A}$ are smooth Deligne-Mumford stacks that are separated and of finite type over $\Z$.
By Theorem \ref{thm:moduliarithmetictorelli}, $\mathcal{F}_{\Q}$ is also a smooth Deligne-Mumford stack which is separated and of finite type.
The reason \cite{Javanpeykar-Loughran:GoodReductionFano}'s proof of the (ordinary) Shafarevich conjecture did not go through for prime Fano threefolds of genus 7 is that their method could not establish this separatedness. 
On the other hand, this separatedness follows also from recent advances in $K$-stability (\cite[Theorem C]{Abban-Zhuang} and \cite[Theorem 1.1 (3)]{Blum-Xu}).
However, for the cohomological Shafarevich conjecture, this separatedness alone does not suffice.

\item 
{\cora As in Remark \ref{rem:basedvr}, we can define $\beta$ on $\mathcal{M} (B)$ for any regular (possibly mixed characteristic) scheme $B$ with $\dim B \leq 1$.
In that sense, we expect that $\beta$ extends to $\mathcal{M}$.
On the other hand, to extend $\Gamma$ to a mixed characteristic base scheme, we need the classification of prime Fano threefolds of genus $7$ in positive characteristic, which is not known yet.
However, in Proposition \ref{prop:mixedtwo-ray}, we construct $C \in \mathcal{M}_{\mathrm{sm}} (k)$, which is the candidate of $\Gamma (X)$ on a $k$-valued point $X$, when $X$ admit (possibly ramified) characteristic 0 lift.
Even though we do not know whether $C \in \mathcal{M} (k)$ in general,
if $X = \beta (C')$ for $C' \in \mathcal{M} (k)$ in the above sense, we have $C \simeq C'$ by the separatedness of $\mathcal{M}_{\mathrm{sm}}$.
}
\end{enumerate}
\end{rem}

To conclude this subsection, we state two direct corollaries of Theorem \ref{thm:moduliarithmetictorelli}.
The first one is a generalization of \cite[Paragraph after Theorem 8.6]{MukaiBrill} over an arbitrary field of characteristic $0$.

\begin{cor}
\label{cor:usualarithmetictorelli}
Let $k$ be a field of characteristic $0$.
Let $X$ and $Y$ be prime Fano threefolds of genus 7 over $k$.
Let $(\Ab^2_{X/k}, \Theta_{X})$ and $(\Ab^2_{Y/k}, \Theta_{Y})$ be their algebraic representatives and polarizations given in Theorem \ref{defn-thm:representative} and Theorem \ref{defn-thm:representativepolarization}.
Suppose that there exists an isomorphism $\phi \colon (\Ab^2_{X/k}, \Theta_{X}) \simeq  (\Ab^2_{Y/k}, \Theta_{Y})$ as principally polarized abelian varieties.
Then $X$ and $Y$ are isomorphic over $k$.
\end{cor}

\begin{proof}
This is a direct corollary of Theorem \ref{thm:moduliarithmetictorelli}.
\end{proof}

\begin{cor}
\label{cor:autfaithfulaction}
Let $k$ be an algebraically closed field of characteristic $0$, and $X$ a prime Fano threefold of genus $7$ over $k$.
For any prime number $\ell$, the natural action of $\Aut (X)$ on $H^3_{\et} (X, \Q_{\ell})$ is faithful.
\end{cor}

\begin{proof}
Let $C$ be the dual of $X$ as in Definition-Proposition \ref{defn-prop:dualvariety}.
We shall show that the isomorphism
$(\Ab^2_{X/k} \Theta_{X}) \simeq (J_{C/k}, \Theta_{C})$  
in Proposition \ref{prop:dualcompatiblejacobian} is compatible with an isomorphism $\Aut (X) \simeq \Aut (C)$  in Theorem \ref{thm:moduliarithmetictorelli} via the natural action of $\Aut (X)$ (resp.\ $\Aut (C)$) on $\Ab^2_{X/k}$ (resp.\ $J_{C/k}$).
For any $g \in \Aut (X)$, as in the proof of Proposition \ref{prop:curvesigmafamily}, $g$ extends to the automorphism $\widetilde{g}$ on $\Sigma_{+} \subset \P (S_{+})$.
Then the dual automorphism $\widetilde{g}^{\vee}$ on $\Sigma_{-} \subset \P (S_{-})$ restricts to the automorphism $g^{\vee}$ on $C$, which is the image of $g$ via the isomorphism $\Aut (X) \simeq \Aut (C)$.
Then we can show that $(g \times g^{\vee})^{\ast} \mathcal{E} \simeq \mathcal{E}$, where $\mathcal{E}$ is as in Proposition \ref{prop:dualcompatiblejacobian}.
Therefore, the morphism
\[
c_{2} (\mathcal{E}) \colon \A^2 (X) \rightarrow \A^1 (C)
\]
satisfies $c_{2} (\mathcal{E}) \circ g = g^{\vee} \circ c_{2} (\mathcal{E})$.
By Proposition \ref{prop:representativeetale} and Remark \ref{rem:representative}, the same holds for the isomorphism $\Ab^2_{X/k} \simeq J_{C/k}$.

Now the result follows from \cite[Theorem 1.13]{Deligne-Mumford} and \cite[Section 19, Theorem 3]{Mumford}, Proposition \ref{prop:representativeetale} and Remark \ref{rem:representative}.
\end{proof}

\subsection{Cohomological Shafarevich conjecture}

As an application of Theorem \ref{thm:moduliarithmetictorelli},
we prove the cohomological Shafarevich conjecture for
prime Fano threefold of genus 7 over
a finitely generated field of characteristic $0$.

\begin{thm}
\label{thm:cohomshaf}
Let $K$ be a finitely generated field of characteristic $0$,
and $R$ be a finitely generated normal $\Z$-algebra
with $K = \Frac R$. We fix a prime number $\ell$ and suppose that $\ell \in R^{\times}$.
Consider the set $\Shaf(K,R)$
as in Theorem \ref{introMainTheoremshaf}, i.e., we set
\[
\Shaf(K,R) :=
\left\{X   \left| 
\begin {array}{l}
X \colon \textup{prime Fano threefold of genus 7 over $K$}, \\
H^3_{\et}(X_{\overline{K}}, \Q_{\ell}) \colon \textup{unramified}\\
\textup{ at any height $1$ prime } \mathfrak{p} \in \Spec R
\end{array}
\right.
\right\}/K\textup{-isom}.
\]
Then the set $\Shaf (K,R)$ is finite.
\end{thm}

\begin{proof}
By Theorem \ref{thm:moduliarithmetictorelli}, the problem is reduced to the finiteness of the following:
\[
\Shaf_{\mathrm{ab}}^{\mathrm{pol}}(K,R)
 :=
\left\{(A,\lambda)   \left| 
\begin {array}{l}
A \colon \textup{$7$-dimensional abelian variety over $K$}, \\
\lambda \colon \textup{principal polarization on $A$},\\
H^1_{\et}(A_{\overline{K}}, \Q_{\ell}) \colon \textup{unramified}\\
\textup{ at any height $1$ prime } \mathfrak{p} \in \Spec R
\end{array}
\right.
\right\}/K\textup{-isom}.
\]
Since 
\[
\Shaf_{\mathrm{ab}}(K,R)
:= \left\{A   \left| 
\begin {array}{l}
A \colon \textup{$7$-dimensional abelian variety over $K$}, \\
H^1_{\et}(A_{\overline{K}}, \Q_{\ell}) \colon \textup{unramified}\\
\textup{ at any height $1$ prime } \mathfrak{p} \in \Spec R
\end{array}
\right.
\right\}/K\textup{-isom},
\]
is a finite set by \cite[Chapter VI, Theorem 2]{Faltings-Wuestholz} and \cite{Martin-Deschamps}, and the natural morphism
\[
\Shaf_{\mathrm{ab}}(K,R) \rightarrow \Shaf_{\mathrm{ab}}^{\mathrm{pol}}(K,R)\]
is finite-to-one by \cite[Theorem 18.1]{Milne}, it finishes the proof.
\end{proof}

\section{Good reduction and cohomology}
\label{section:goodredcriteria}
In this section, we give the counter-example to
the N\'eron--Ogg--Shafarevich criterion for prime Fano threefolds of genus $7$.
The equal-characteristic $0$ case is easily deduced from Theorem \ref{thm:moduliarithmetictorelli}, but the mixed characteristic case is more subtle.
To this end, first, we recall another construction of $C$ from $X$, the so-called two-ray game construction.

\subsection{Two-ray game construction}
In this subsection, we recall the construction from $X$ to $C$ given by two-ray games.

\begin{defn-prop}
\label{defn-prop:Xtocurveflop}
Let $k$ be an algebraically closed field of characteristic $0$, and $X$ a prime Fano threefold of genus $7$ over $k$.
\begin{enumerate}
\item 
The anti-canonical divisor $-K_{X}$ is very ample and it induces an embedding $X \hookrightarrow \P^{8}$.
\item
Let $q \subset X$ be a sufficiently general conic in the sense of \cite[Theorem 6.5]{Iliev-Markushevich}.
Let $\widetilde{X}' \rightarrow X$ 
be the blow-up at the conic $q$.
Then there exists the flop $\widetilde{X}' \dashrightarrow \widetilde{Q}$ 
and a map
$\widetilde{Q} \rightarrow Q$ which is a blow-up at a smooth curve $\Gamma_{\mathrm{f}} \subset Q \subset \P^4$ of genus $7$ and degree $10$.
Here, $Q$ is a smooth quadric threefold in $\P^4$.
We denote the curve $\Gamma_{\mathrm{f}}$ by $\Gamma_{\mathrm{f}}(X)$.
\item 
Assume that $k = \C$.
In the case of (2) (resp.\ (3)), we have an isomorphism 
\[(J(X), \Theta(X)) \simeq (J(\Gamma_{\mathrm{f}}(X)), \Theta (\Gamma_{\mathrm{f}}(X)))
\]
as principally polarized abelian varieties. Here, $\Theta(\Gamma_{\mathrm{f}}(X))$ is the theta polarization on the Jacobian.
\item 
The isomorphism class of a curve $\Gamma_{\mathrm{f}} (X)$ does not depend on the choice of $q \subset X$.
\end{enumerate}
\end{defn-prop}

\begin{proof}
(1) The assertion follows from \cite[Corollary 4.1.13]{Iskovskikh-Prokhorov}.

(2) See, e.g.,  \cite[Theorem 6.3 and Theorem 6.5]{Iliev-Markushevich}.

(3) See the proof of \cite[Lemma 6.8]{Iliev-Markushevich}
(see also \cite[Section 8]{MukaiBrill}).

(4) The assertion is essentially known, but we provide a proof for the reader's convenience.
We take two sufficiently general conics $q_1, q_2 \subset X$. 
We take a finitely generated subfield $k'$ of $k$ so that $q_1, q_2 \subset X$ and the construction in (2) are defined. 
Let $\Gamma_{\mathrm{f}, 1} (X)$ and $\Gamma_{\mathrm{f}, 2} (X)$ be obtained curves over $k'$.
By taking an embedding $k' \hookrightarrow \C$ and applying (4) and the Torelli theorem for curves, we have an isomorphism $\Gamma_{\mathrm{f},1} (X)_{\C} \simeq \Gamma_{\mathrm{f},2} (X)_{\C}$.
Therefore, there exists an isomorphism $\Gamma_{\mathrm{f},1} (X)_{k} \simeq \Gamma_{\mathrm{f},2} (X)_{k}$ and it completes the proof.
\end{proof}

\begin{rem}
\label{rem:suffgeneral}
In Definition-Proposition \ref{defn-prop:Xtocurveflop} (2), the condition ``sufficiently general", which means satisfying $(\ast)$ and $(\ast\ast)$ in \cite[Subsection 4.1]{Iskovskikh-Prokhorov}, is actually automatic for smooth conics $q \subset X$.
Indeed, the condition $(\ast)$ is automatic by \cite[Corollary 4.4.3]{Iskovskikh-Prokhorov}.
Moreover, the condition $(\ast \ast)$ is also automatic by the same argument as in \cite[Proposition 4.2, and Proposition 7.3]{Tanaka2} (or see \cite{Jahnke-Peternell-Radloff}).
\end{rem}

\begin{prop}
\label{prop:relationGamma}
Let $k, X$ be as in Definition-Proposition \ref{defn-prop:Xtocurveflop}.
Then we have 
\[
\Gamma (X) \simeq \Gamma_{\mathrm{f}}(X).
\]
\end{prop}

\begin{proof}
This is \cite[Corollary 6.12]{Iliev-Markushevich}.
\end{proof}

\begin{rem}
There exists another construction from $C$ to $X$, so-called non-commutative Brill--Noether theory (\cite{MukaiBrill}).
Since we do not need this construction, we omit it.
\end{rem}

\subsection{Good reduction}
In this subsection, we define good reduction of Fano varieties.

\begin{defn}
\label{defn:goodreduction}
Let $R$ be a discrete valuation ring, $K$ a fraction field of $R$, and $k$ a residue field of $R$.  
Let $X$ be a Fano variety over $K$.
We say \emph{$X$ has good reduction at $R$} if and only if there exists a smooth projective scheme $\mathcal{X}$ over $R$ such that $\mathcal{X}_{K} \simeq X$ holds and $-K_{\mathcal{X}/R}$ is ample.
\end{defn}

\begin{rem}
\label{rem:reductionFano}
We use the same notation as in Definition \ref{defn:goodreduction}.
Assume that a prime Fano threefold $X$ of genus $7$ over $K$ has good reduction (as Fano varieties) at $R$, and let $\pi \colon \mathcal{X} \rightarrow R$ be a smooth projective model as in Definition \ref{defn:goodreduction}. In this case, the following hold. 
\begin{enumerate}
\item
The special fiber $\mathcal{X}_{k}$ is also a prime Fano threefold of genus $7$ over $k$. 
Indeed, the Picard rank of $\mathcal{X}_{\overline{k}}$ is $1$ by \cite[Theorem 1.1]{Gounelas-Javanpeykar}.
Moreover, since the intersection number $(-K_{\mathcal{X}_{k}})^3$ is 14, $\mathcal{X}_{k}$ has the index 1 and the genus $7$.
\item 
The anti-canonical bundle $-K_{\mathcal{X}/R}$ is very ample.
Indeed, $-K_{X/K}$ is very ample with $h^0 (-K_{X/K}) =9$ by Definition-Proposition \ref{defn-prop:Xtocurveflop} and \cite[Theorem 1.1 and Theorem 1.2]{Tanaka1}, and so is $-K_{\mathcal{X}_{k}/k}$ by the same reason.
Therefore, $\pi_{*} (\cO(-K_{\mathcal{X}/R}))$ is free module of rank $9$ over $R$ whose basis induces an embedding $\mathcal{X} \hookrightarrow \P^8_{R}$. 
\end{enumerate}
\end{rem}

In general, good reduction as Fano varieties (the above definition) and good reduction as smooth varieties are different concepts.
For instance, degree 4 del Pezzo surfaces in \cite[Remark 4.6]{Scholl} have good reduction as smooth varieties, but do not have good reduction as Fano varieties at some finite places.
The following proposition ensures that they are equivalent if the second Betti number of the generic fiber is equal to $1$: 

\begin{prop}
Let $R, K, k$ be as in Definition \ref{defn:goodreduction}.
Let $X$ be a Fano variety of genus $7$ over $K$ such that $b_{2} (X_{\overline{K}}) =1$. 
If there exists a smooth projective scheme $\mathcal{X}$ over $R$ such that $\mathcal{X}_{K} \simeq X$, then $\mathcal{X}$ satisfies the condition of Definition \ref{defn:goodreduction}.
\end{prop}

\begin{proof}
By the smooth proper base change theorem and Proposition \ref{prop:basic}, $H^2 (\mathcal{X}_{k}, \Q_{\ell}) $ is 1-dimensional.
In particular, $\mathcal{X}_{k}$ has a Picard rank $1$.
Since $h^0(-K_{\mathcal{X}_{k}}) \geq h^0(-K_{X}) = 9$, $-K_{\mathcal{X}_{k}}$ is ample.
It implies that $-K_{\mathcal{X}/R}$ is ample, and it finishes the proof. 
\end{proof}

\subsection{Counter-example to the N\'eron--Ogg--Shafarevich criterion}
In this subsection, we show that an analogue of Neron--Ogg--Shafarevich criterion does not hold for prime Fano threefolds of genus 7.

\begin{lem}
\label{lem:reductioncurveschar0}
Let $k$ be an algebraically closed field of characteristic $0$.
Then there exists a Dedekind domain $\cO_{K}$ of finite type over $k$ with fraction field $K$, a closed point $\p \in \Spec \cO_{K}$ 
and a smooth proper curve $\mathcal{C}$ over $\cO_{K,\p}$ of genus $7$ such that 
\begin{itemize}
\item 
The geometric generic fiber $\mathcal{C}_{\overline{K}}$ does not have $g_4^1$.
\item 
The geometric special fiber $\mathcal{C}_{k} := \mathcal{C} \otimes_{\cO_{K,\p}} \kappa(\p)$ has $g_4^1$, does not have $g_3^1$ or $g_6^2$, and is not hyperelliptic (We call such a curve a \emph{generic curve with $g_4^1$}).
\end{itemize}
\end{lem}

\begin{proof}
This follows from the standard argument on the moduli space of curves.
First, note that there exists at least one generic curve with $g_4^1$ of genus $7$ over $k$ (see \cite{Mukaicurve3} for example).
Therefore, on the coarse moduli space $M_{7}$ of smooth curves of genus 7,
the locus $V$ 
corresponding to generic curves with $g_4^1$ of genus $7$ is locally closed and non-empty.
We take closed points $P \in V$ and $Q \in M_{7} \setminus V$.
Then we can take an irreducible curve $X \subset M_{7}$ passing $P$ and $Q$.
By replacing $X$ with a finite cover and taking normalization, we have a smooth curve $Y$ over $k$ with a morphism $\phi\colon Y \rightarrow M_{7}$ whose image contains $P$ and $Q$ such that $\phi$ factors through the fine moduli space $\mathcal{M}_{7,L}$ with a suitable level structure $L$.
By pulling back the universal curve over $\mathcal{M}_{7,L}$, we have a relative curve of genus $7$ over $Y$. 
By taking localization of $Y$ at $\p$, we have a desired family.
\end{proof}

\begin{lem}
\label{lem:reductioncurves}
Let $p$ be a prime number.
There exists a number field $K$, a finite place $\mathfrak{p}$ of $K$ with $\mathfrak{p} | p$, and a smooth proper curve $\mathcal{C}$ of genus $7$ over $\cO_{K,\p}$ such that the following hold.
\begin{itemize}
\item 
The curve $\mathcal{C}_{\overline{K}}$ does not have $g_4^1$.

\item 
The geometric special fiber $\mathcal{C}_{\overline{k}}$ is a generic curve with $g_4^1$ in the sense of Lemma \ref{lem:reductioncurveschar0}.
\end{itemize}
Here, $\cO_{K,\mathfrak{p}}$ is the localization of the integer ring of $K$ at $\mathfrak{p}$ and $k$ is its residue field.
\end{lem}

\begin{proof}
Let $q$ be a prime number other than $p$.
First, we can take a smooth proper curve $C_{1}$ (resp.\ $C_2$) of genus 7 over $\overline{\F}_{p}$ (resp.\  $\overline{\F}_{q}$) with $g_4^1$ which is a generic curve with $g_4^1$  (resp.\ without $g_4^1$).
We remark that there exists at least one such curve (see \cite[Table 1]{Mukaicurve}. Note that, generic with $g_4^1$ is equivalent to tetragonal without $g_6^2$).

By the argument in \cite[Example 0.9]{MBSkolem}, we have a number field $K$ and finite places $\mathfrak{p}$ and $\mathfrak{q}$ such that the following hold.
\begin{itemize}
\item 
We have $\mathfrak{p} | p$ and $\mathfrak{q} | q$.
\item 
There exists a smooth projective scheme $\mathcal{C}_{1}$ over $\cO_{K, \p}$ and $\mathcal{C}_{2}$ over $\cO_{K,\mathfrak{q}}$ such that $\mathcal{C}_{1,K} \simeq \mathcal{C}_{2,K}$, $\mathcal{C}_{1, \overline{\F}_{p}} \simeq C_{1}$, and $\mathcal{C}_{2, \overline{\F}_{q}} \simeq C_{2}$ hold.
\end{itemize}
Since the locus parametrizing smooth curves with $g_4^1$ is Zariski closed in the moduli space of smooth curves of genus $7$, $\mathcal{C}_{1,\overline{K}} \simeq \mathcal{C}_{2,\overline{K}}$ is a curve without $g_4^1$.
Therefore, the curve $\mathcal{C}_{1}$ satisfies the desired condition.
\end{proof}

To discuss the two-ray game construction in mixed characteristic, first we give a definition of a flop.

\begin{defn}
\label{defn:flop}
Let $R$ be an excellent discrete valuation ring.
Let $\mathcal{X}$ and $\mathcal{Y}$ be normal Noetherian excellent integral schemes over $R$ of finite dimension with dualizing complexes.
We assume that $\mathcal{X}$ is terminal and $\Q$-factorial.
Let $f \colon \mathcal{X} \rightarrow \mathcal{Y}$ be a projective birational morphism.
\begin{enumerate}
\item 
We say $f$ is a \emph{flopping contraction} if $f$ is an isomorphism in codimension $1$, the relative Picard rank $\rho(\mathcal{X}/Y) =1$, and
$K_{\mathcal{X}/R}$ is numerically $f$-trivial.
\item 
Let $D$ be a $\Q$-Cartier effective $\Q$-divisor on $\mathcal{X}$ such that $-D$ is $f$-ample.
Let $f^{+} \colon \mathcal{X}^{+} \rightarrow \mathcal{Y}$ be a projective birational morphism between normal integral schemes.
We say $f^{+}$ is a \emph{$D$-flop} if $\mathcal{X}^{+}$ is $\Q$-factorial, 
$f^{+}$ is an isomorphism in codimension $1$, $D^{+}$ is $f^{+}$-ample.
Here, $D^{+}$ is a strict transform of $D$ in $\mathcal{X}^{+}$.
Then a $D$-flop is isomorphic to
\[
\Proj_{\mathcal{Y}}\bigoplus_{d\geq 0} \cO_{\mathcal{Y}} (d f_{\ast} (D)) \rightarrow \mathcal{Y},
\]
so a $D$-flop is unique.
\end{enumerate}
\end{defn}

\begin{prop}
Let $f$, $\mathcal{X}$, $\mathcal{Y}$ be as in Definition \ref{defn:flop}.
Assume that $f$ is a flopping contraction, $D$ a $\Q$-Cartier effective $\Q$-divisor on $\mathcal{X}$ which is $f$-anti-ample, and $f^{+}\colon \mathcal{X}^{+} \rightarrow \mathcal{Y}$ a $D$-flop.
Then the following holds.
\begin{enumerate}
\item
We have $\rho(\mathcal{X}^{+}/\mathcal{Y}) =1$.
\item 
For any $\Q$-Cartier effective $f$-anti-ample divisor $D'$ on $\mathcal{X}$, $f^{+}$ is $D'$-flop.
\end{enumerate}
By (2), the notion of $D$-flop is independent of the choice of $D$, and we call $f^{+}$ simply a \emph{flop} of $f$.
\end{prop}

\begin{proof}
(1) We note that $\mathcal{X}^{+}$ and $\mathcal{X}$ are $\Q$-factorial and isomorphic in codimension 1 over $\mathcal{Y}$.
It suffices to show that for any numerically $f^{+}$-trivial $\Q$-Cartier divisor $E^{+}$ on $\mathcal{X}^{+}$, its strict transform $E$ on $\mathcal{X}$ is numerically $f$-trivial
(converse direction can be proved similarly).
Take a normal integral scheme $\mathcal{W}$ with projective birational morphisms $\mathcal{X} \xleftarrow{g} \mathcal{W} \xrightarrow{h} \mathcal{X}^{+}$, we have $h^{\ast} E^{+} = g^{\ast} (g_{\ast} (h^{\ast} E^{+}))$  by the negativity lemma (cf.\ \cite[Lemma 2.14]{BMPSTWW}).
Since $h^{\ast} E^{+}$ is $f^{+}\circ h = f \circ g$-numerically trivial, $g_{\ast} (h^{\ast} E^{+}) = E$ is $f$-numerically trivial, and it finishes the proof.

(2) It suffices to show that the strict transform $D'^{+}$ of $D'$ on $\mathcal{X}^{+}$ is $f^{+}$-ample.
Since $f \neq f^{+}$, $D'^{+}$ is not $f^{+}$-anti-ample.
Moreover. if $D'^{+}$ is $f^{+}$-numerically trivial, $D$ is also $f$-numerically trivial by the argument in (1) and it gives a contradiction.
Combining with (1), $D'^{+}$ is $f^{+}$-ample and it finishes the proof.
\end{proof}

Now we can prove the following mixed characteristic analogue of two-ray games for conics.
Note that, the proof of the existence of flop is similar to \cite[Proposition 10.7]{Tanaka2}. However, since he considers a $W(k)$-lift of purely positive characteristic flop, the situation is a bit different.
Therefore, we include its complete proof.

\begin{prop}
\label{prop:mixedtwo-ray}
Let $R$ be a Henselian excellent discrete valuation ring, $K$ a fraction field of $R$, $k$ a residue field of $R$, and $\varpi$ a uniformizer of $R$. Assume that $K$ is of characteristic $0$.
Let $X$ be a prime Fano threefold of genus $7$ over $K$.
Assume that $X$ has good reduction at $R$, and take a smooth projective scheme $\mathcal{X}$ over $R$ with $\mathcal{X}_{K} \simeq X$.
Then by extending $K$ if necessary, we have the following.
\begin{enumerate}
\item
There exists a smooth relative conic $q \subset \mathcal{X}$ with respect to the anti-canonical embedding $\mathcal{X} \subset \P^8_{R}$.
\item
We denote the blow-up of $\mathcal{X}$ at $q$ by $b \colon \widetilde{\mathcal{X}} \rightarrow \mathcal{X}$.
Then the anti-canonical divisor $-K_{\widetilde{\mathcal{X}}/R}$ is semi-ample, and the corresponding fiber space morphism $f \colon \widetilde{\mathcal{X}} \rightarrow \mathcal{Z}$ is a flopping contraction.
Moreover, the base change $f_{k}$ (resp.\ $f_{K}$) of $f$ is also a flopping contraction.
\item
The flop $f^+ \colon \widetilde{\mathcal{Q}} \rightarrow \mathcal{Z}$ of $f$ exists and $\widetilde{\mathcal{Q}}$ is smooth over $R$.
Moreover, the base change $(f^+)_{k}$ (resp.\ $(f^+)_{K}$) of $f^{+}$ is a flop of $f_{k}$ (resp.\ $f_{K}$). 
\item
There exists a morphism $g \colon \widetilde{\mathcal{Q}} \rightarrow \mathcal{Q}$, where $\mathcal{Q}$ is a relative smooth quadric threefold in $\P^{4}_{R}$.
Moreover, $g$ is a blow-up of $\mathcal{Q}$ at $\mathcal{C} \subset \mathcal{Q}$.
Here, $\mathcal{C} \subset \mathcal{Q} \subset \P^{4}_{R}$ is a smooth projective relative curve over $R$ whose genus is $7$ and degree is $10$.
Moreover,
the geometric generic fiber $\mathcal{C}_{\overline{K}}$ satisfies 
$
\mathcal{C}_{\overline{K}} \simeq \Gamma_{\mathrm{f}}(X_{\overline{K}})
$
\item 
The generic fiber $\mathcal{C}_{\overline{K}}$ is a smooth proper curve of genus $7$ without $g_4^1$.
Moreover, if $k$ is of characteristic 0, so is the special fiber $\mathcal{C}_{\overline{k}}$.
\item
Even if $k$ is of positive characteristic,
the special fiber $\mathcal{C}_{\overline{k}}$ is a smooth proper curve of genus 7 which is not a generic curve with $g_4^1$ in the sense of Lemma \ref{lem:reductioncurveschar0}.
\end{enumerate}
\end{prop}

\begin{proof}
We only consider the case where $K$ is mixed characteristic $(0,p)$ since the equal-characteristic $0$ case is easier.

(1) Let $F(X)$ (resp.\ $S(X)$) be the Hilbert scheme of conics (resp.\ lines) on $X$. 
Let $\widetilde{F}(X)$ (resp.\ $\widetilde{S}(X)$) be the universal conic (resp.\ universal line) over $X$. We shall use the same notation for the model $\mathcal{X}$ and the special fiber $\mathcal{X}_{k}$.
The variety $F(X)$ is a smooth projective variety of dimension $2$ over $K$ (
\cite[Theorem 5.3]{Kuznetsovderived}, see also the proof of Proposition \ref{prop:basic} (4)).
Moreover, by \cite[Proposition 5.4]{Tanaka2}, we have $\dim S(\mathcal{X}_{k}) =1$.
We define the proper closed subscheme $Z_{k} \subset \mathcal{X}_{k}$ as the image of 
$\widetilde{S}(\mathcal{X}_{k})$ via the projection $\widetilde{S}(\mathcal{X}_{k}) \rightarrow \mathcal{X}_{k}$.
On the other hand, let $F(X)' \subset F(X)$ be the non-smooth locus of the projection $\widetilde{F}(X) \rightarrow F(X)$,
and $\widetilde{F}(X)' \subset \widetilde{F}(X)$ the preimage of $F(X)'$, which is a closed subscheme of dimension less than $3$.
Let $W \subset X$ be the image of $\widetilde{F}(X)'$ via the projection $\widetilde{F} (X) \rightarrow X$.
Let $W_{k} \subset \mathcal{X}_{k}$ be the special fiber of the Zariski closure of $W$ in $\mathcal{X}$, which is a closed subscheme of dimension less than 3.
Now by extending $K$ if necessarily, we  can choose a point $x_{k} \in (\mathcal{X}_{k} \setminus (Z_{k} \cup W_{k}))(k)$.
Since $\mathcal{X}$ is smooth and $R$ is Henselian, we can choose a lift $x \in \mathcal{X} (R)$ of $X$.
Then by the choice of $x$, $x_{K} \in \mathcal{X} (K) = X(K)$ lies on a smooth conic $q_{K} \subset X$.
We define $q \subset \mathcal{X}$ as the Zariski closure of $q_{K}$ in $\mathcal{X}$. 
Then the special fiber $q_{k} \subset \mathcal{X}_{k}$ is a (possibly singular) conic passing through $x_{k}$.
However, by the choice of $x_{k}$ again, there is no line passing through $x_{k}$, so $q_{k}$ is also a smooth conic.
Now we have a desired smooth relative conic $q \subset \mathcal{X}$.

(2) We denote the blow up of $\mathcal{X}$ at $q$ by $\widetilde{\mathcal{X}}$.
The line bundle $\cO_{\widetilde{\mathcal{X}}_{k}} (- K_{\widetilde{\mathcal{X}}_{k}})$ is base point free and big by \cite[Proposition 7.2]{Tanaka2}, and the same holds true for $\cO_{\widetilde{\mathcal{X}}_{K}} (- K_{\widetilde{\mathcal{X}}_{K}})$ (see \cite[Proposition 4.4.1]{Iskovskikh-Prokhorov}).
Therefore, by \cite[Theorem 1.2]{Witaszek},
the line bundle $\cO_{\widetilde{\mathcal{X}}} (- K_{\widetilde{\mathcal{X}}/R})$ is semi-ample, and we obtain the fiber space $f \colon \widetilde{\mathcal{X}} \rightarrow \mathcal{Z}$ associated with the line bundle.
Then the special fiber $f_{k}$ only contracts a curve \cite[Proposition 7.3 and Proposition 7.8]{Tanaka2},  
and by \cite[Proposition 4.4.1]{Iskovskikh-Prokhorov}, the generic fiber $f_{K}$ is also a flopping contraction.
Therefore, $f$ is also a flopping contraction. 
To prove $f_{k}$ is a flopping contraction, it suffices to show that $\mathcal{Z}_{k}$ is normal, which is equivalent to $f_{k \ast} \cO_{\widetilde{\mathcal{X}}_{k}} = \cO_{\mathcal{Z}_{k}}$.
Let $\widetilde{\mathcal{X}}_{k} \xrightarrow{h_{1}} V \xrightarrow{h_{2}} \mathcal{Z}_{k}$ be the Stein factorization of $f_{k}$, where $h_1$ is a flopping contraction given in \cite[Proposition 7.3 and Proposition 7.8]{Tanaka2}, and $h_{2}$ is a normalization. 
Since $\mathcal{Z}_{k}$ is reduced, and we have an injection
\[
\cO_{\mathcal{Z}_{k}} \hookrightarrow f_{k \ast} (\cO_{\widetilde{\mathcal{X}}_{k}}) = h_{2,\ast} (\cO_{V}). 
\]
Since we have a commutative diagram
\[
 \xymatrix{
\cO_{\mathcal{Z}} \ar[r]^{\simeq} \ar[d]
& f_{\ast} \cO_{\widetilde{\mathcal{X}}}  \ar[d]\\
\cO_{\mathcal{Z}_{k}}
 \ar@{^{(}->}[r] &
f_{k \ast} \cO_{\widetilde{\mathcal{X}}_{k}},
}
\]
it suffices to show that the right vertical arrow is surjective.
The multiplication by a uniformizer $\varpi \in R$ induces an exact sequence
\begin{equation}
\label{eqn:exactspecial}
0 \rightarrow \cO_{\widetilde{\mathcal{X}}} \rightarrow \cO_{\widetilde{\mathcal{X}}} \rightarrow \cO_{\widetilde{\mathcal{X}_{k}}} \rightarrow 0.
\end{equation}
By taking $R^i f_{\ast}$ of (\ref{eqn:exactspecial}), it suffices to show 
\[
R^1 f_{\ast} \cO_{\widetilde{\mathcal{X}}} =0.
\]
Since $\mathcal{Z}_{K}$ has terminal singularity and of characteristic 0, it has rational singularity and we have 
\begin{equation}
\label{eqn:genericrational}
R^i f_{K \ast} \cO_{\widetilde{\mathcal{X}}_{K}} = 0
\end{equation}
for $i>0$.
On the other hand, since $h_{1}$ is a flopping contraction, we have $R^i h_{1 \ast} \cO_{\widetilde{\mathcal{X}}_k} = 0$ by \cite[Proposition 6.10]{Tanakaflop}.
Since $h_2$ is finite, we have $R^i f_{k \ast} \cO_{\widetilde{\mathcal{X}}_k} = 0$ for $i>0$.
By taking $Rf_{\ast}$ of (\ref{eqn:exactspecial}) again, the multiplication by $\varpi$ on  $R^i f_\ast (\cO_{\widetilde{\mathcal{X}}})$ is surjective.
Combining with (\ref{eqn:genericrational}), we have \[
R^i f_{\ast} \cO_{\widetilde{\mathcal{X}}} = 0
\]
for $i>0$, and it completes the proof.

(3) We prove the assertion (3), namely, the existence of a flop.
Note that, for a faithfully flat extension of Henselian discrete valuation rings $R \subset S$, the base change $f_{S}$ is also flopping contraction since the assumptions are preserved under base change and the relative conic $q_{R}$ is still smooth.
Since the existence of a flop is equivalent to the finite generation of some relative log canonical algebra, we may assume that $k$ is algebraically closed.
Let $P \in \mathcal{Z}_{k} \subset \mathcal{Z}$ be a closed point. We assume that $\mathcal{Z}$ is singular at $P$. In particular, $P \in \mathcal{Z}_{k}$ is also a singular point. 
Let $\mathcal{Z}':= \widehat{\cO_{\mathcal{Z},P}}$ be a completion of $\cO_{\mathcal{Z}}$ at $P$, and we put $\widetilde{\mathcal{X}}' := \widetilde{\mathcal{X}} \times_{\mathcal{Z}} \mathcal{Z}'$.
Note that we have
\[
(\widehat{\cO_{\mathcal{Z},P}})_{k} := \widehat{\cO_{\mathcal{Z},P}} \otimes_{R} k = \widehat{\cO_{\mathcal{Z}_{k},P}},
\]
where $\widehat{\cO_{\mathcal{Z}_{k},P}}$ is a completion of $\cO_{\mathcal{Z}_{k}}$ at $P$.
By \cite[Lemma 6.12]{Tanakaflop},
there exist an injective $k$-algebra homomorphism
\[
 h \colon
 k[[x,y,z]] \hookrightarrow \widehat{\cO_{\mathcal{Z}_{k},P}},
\]
and $w \in \widehat{\cO_{\mathcal{Z}_{k},P}}$ such that we have $\widehat{\cO_{\mathcal{Z}_{k},P}} = k[[x,y,z]] \oplus k[[x,y,z]]w$ via $h$.
We take a lift $\widetilde{w} \subset \widehat{\cO_{Z,P}}$.
By taking lifts of $h(x), h(y), h(z) \in (\widehat{\cO_{\mathcal{Z},P}})_k$, we have an $R$-algebra homomorphism
\[
\widetilde{h} \colon R[[x,y,z]] \rightarrow  \widehat{\cO_{Z,P}}.
\]
Take a lift $\widetilde{w} \in \widehat{\cO_{Z,P}}$ of $w$, and consider the $R[[x,y,z]]$-module map
\begin{align*}
&\Phi\colon R[[x,y,z]]^{\oplus 2} \rightarrow \widehat{\cO_{Z,P}} \\
& (a,b) \mapsto \widetilde{h} (a) + \widetilde{h}(b)\widetilde{w}.
\end{align*}
Since $R[[x,y,z]]$ is $\varpi$-complete and $\widehat{\cO_{Z,P}}$ is $\varpi$-separated, this map is surjective.
Moreover, since $\widehat{\cO_{Z,P}}$ is $\varpi$-torsion free, we have $\ker \Phi \otimes_{R[[x,y,z]]} k[[x,y,z]] =0$, it implies $\ker \Phi \subset \bigcap \varpi^n R[[x,y,z]] =0$.
Therefore, $\Phi$ is an isomorphism.
We put 
\[
W := \Spec R[[x,y,z]].
\]
We have a map of normal $R$-schemes
\[
\alpha \colon
\widetilde{\mathcal{X}}' \xrightarrow{f'}
\mathcal{Z}' 
\xrightarrow{\beta}
W,
\]
where $\beta$ is $\widetilde{h}^{\ast}$.

Let $E$ be a prime divisor on $\widetilde{\mathcal{X}}'$ such that $-E$ is $f'$-ample (note that $\widetilde{\mathcal{X}}'$ is regular).
We put $E_{\mathcal{Z}'} := f'_{\ast} E$ and $E_{W} := \beta_{\ast} E_{\mathcal{Z}'}$.
Since $E_{\mathcal{Z}'}$ is non-$\Q$-Cartier and $E_{W}$ is Cartier, we have $E_{\mathcal{Z}'} \neq a \beta^{*} E_{W}$ for any $a\in \Q$, and we have
\[
\beta^{\ast} E_{W} = E_{\mathcal{Z}'} + E'_{\mathcal{Z}'},
\]
where $E'_{\mathcal{Z}}$ is the Galois conjugate (with respect to $\beta$) of $E_{\mathcal{Z}'}$.
Let $E'$ be the strict transform of $E'_{\mathcal{Z}'}$ on $\widetilde{\mathcal{X}}'.$
Then we have
\[
\alpha^{*} E_{W} = E+ E'.
\]
Since $\alpha^{\ast} E_{W}$ is $f'$-trivial, $E'$ is $f'$-ample.
We note that the Galois conjugation induces
\[
H^0(\mathcal{Z}', \cO (d E_{\mathcal{Z}'}')) \simeq  H^0(\mathcal{Z}', \cO(d E_{\mathcal{Z}'}))
\]
for any $d \in \Z_{\geq 0}$.
Now we have
\begin{equation}
\label{eqn:Galoisconj}
\widetilde{\mathcal{X}}' \simeq  \Proj \bigoplus_{d \geq 0} H^{0}(\mathcal{Z}', \cO(dE'_{\mathcal{Z}'})) \simeq \Proj \bigoplus_{d \geq 0} H^{0}(\mathcal{Z}', \cO(dE_{\mathcal{Z}'})) =: \widetilde{\mathcal{Q}}'.
\end{equation}
Therefore, by the argument in \cite[Proposition 6.6]{Kollar-Mori}, the flop $f^{+}$ of $f$ exists over an open neighborhood of a singular point $P \in \mathcal{Z}$, and its pull-back via $\mathcal{Z}' \rightarrow \mathcal{Z}$ equals 
$f'^{+} \colon \widetilde{\mathcal{Q}}' \rightarrow \mathcal{Z}'$.
Since $f$ is isomorphic over the smooth locus of $\mathcal{Z}$, a flop exists around smooth points.
Moreover, on the generic fiber, the existence of a flop is well-known (see \cite[Theorem 6.14]{Kollar-Mori} for example. Note that in our setting, the Picard number is preserved under field extension, and we can easily reduce the problem to the case of $\C$) and the flop $\widetilde{\mathcal{Q}}_{K}$ of $\widetilde{\mathcal{X}}_{K}$ is smooth (see \cite[Theorem 2.4]{Kollarflop}).
Therefore, the flop $f^{+} \colon \widetilde{\mathcal{Q}} \rightarrow \mathcal{Z}$
of $f$ exists.
Note that, the Galois action $\sigma$ induces a birational map $\mathcal{Z}' \dashrightarrow \mathcal{Z}'$ over $W$, which is isomorphic in codimension $1$ since $\mathcal{Z}' \rightarrow W$ is finite.
By the argument in the proof of \cite[Theorem 6.13]{Tanakaflop}, the base change $\beta_{k}$ of $\beta$ is also finite Galois of degree 2.
Therefore, $\sigma$ induces the associated Galois action $\sigma_{k} \colon \mathcal{Z}'_{k} \dashrightarrow \mathcal{Z}'_{k}$ which is isomorphic in codimension 1.
Since we have a commutative diagram
\[
 \xymatrix{
\widetilde{\mathcal{X}}' \ar[r]^{\simeq} \ar[d]_{f'}
& \widetilde{\mathcal{Q}}'  \ar[d]^{f'^{+}}\\
\mathcal{Z}' \ar@{..>}[r]^{\sigma} 
&\mathcal{Z}',
}
\]
by (\ref{eqn:Galoisconj}),
restricting to the special fiber, we have a commutative diagram
\[
 \xymatrix{
\widetilde{\mathcal{X}}'_{k} \ar[r]^{\simeq} \ar[d]_{f'_k}
& \widetilde{\mathcal{Q}}'_{k}  \ar[d]^{f'^{+}_{k}}\\
\mathcal{Z}'_{k} \ar@{..>}[r]^{\sigma_{k}} 
&\mathcal{Z}'_{k}.
}
\]
Since $f'_{k}$ and $\sigma_k$ are isomorphic in codimension 1,
$f'^{+}_{k}$ is also isomorphic in codimension 1. 
Therefore, the special fiber of a flop $f_{k}$ is also isomorphic in codimension 1.
Let $D$ be a prime divisor on $\mathcal{X}$ such that $-D$ is $f$-ample. Note that $-D|_{\mathcal{X}_{k}}$ is $f_{k}$-ample.
Then the strict transform $D^{+}$ of $D$ on $\widetilde{\mathcal{Q}}$ is $f^{+}$-ample, and so $D^{+}|_{\widetilde{\mathcal{Q}}_{k}}$ is $f^{+}_{k}$-ample.
Moreover, $\widetilde{\mathcal{Q}}_{k}$ is regular since its completion at closed points are regular.
Therefore, $f^{+}_{k}$ is a flop of $f_{k}$.

(4) By \cite[Theorem 1.1]{Kollarsmooth} (see also \cite[Theorem 7.4]{Tanaka2}), 
there exists a 
$K_{\widetilde{\mathcal{Q}}_{k}}$-negative extremal ray contraction $g_{k}\colon \widetilde{\mathcal{Q}}_{k} \rightarrow Q$, 
and $Q$ is smooth.
We fix a very ample line bundle $H$ on $Q$, and we
put $M := g_{k}^{\ast} (H)$.
Since $H^2(\widetilde{\mathcal{Q}}_{k}, \cO_{\widetilde{\mathcal{Q}}_{k}}) =0$, we can take a lift $L$ of line bundle $M$ on $\widetilde{\mathcal{Q}}$.
Then $L_{k}:=L|_{\widetilde{\mathcal{Q}}_{k}} = M$ is semi-ample and $L_{k}\otimes \cO_{\widetilde{\mathcal{Q}}_{k}} (-K_{\widetilde{\mathcal{Q}}_{k}})$ is ample.
Therefore, $L_{K}$ is nef and $L_{K} \otimes \cO_{\widetilde{\mathcal{Q}}_{K}} (-K_{\widetilde{\mathcal{Q}}_{K}})$ is ample, so $L_{K}$ is base point free by the base point free theorem.
By \cite[Theorem 1.2]{Witaszek}, $L$ is semi-ample, and we have the associated fiber space morphism
\[
g \colon \widetilde{\mathcal{Q}} \rightarrow \mathcal{Q}.
\]
Since $L_{K} \otimes \cO_{\widetilde{\mathcal{Q}}_{K}} (-K_{\widetilde{\mathcal{Q}}_{K}})$ is ample and $L_{K}$ is semi-ample, $g_{K}$ is $K_{\widetilde{\mathcal{Q}}_{K}}$-negative extremal contraction.
Since the Stein factorization of the base change $g_k$ is given by $\widetilde{\mathcal{Q}}_{k} \rightarrow Q \rightarrow \mathcal{Q}$ and $Q$ is smooth, we can prove $g_{k \ast} \cO_{\widetilde{\mathcal{Q}}_{k}} = \cO_{\mathcal{Q}_{k}}$
in a similar way to (2) (we can use \cite{Chatzistamatiou-Rulling} since $Q$ is smooth).
By Takeuchi's result (\cite{Takeuchi} ,see \cite[Theorem 4.4.11]{Iskovskikh-Prokhorov}, the base change $g_{\overline{K}}$ is given by the blow-up of a quadratic threefold $\mathcal{Q}_{\overline{K}} \subset \P_{\overline{K}}$ at a genus 7 degree 10 curve $C_{\overline{K}}$.
In particular, we have $(-K_{\mathcal{Q}_{K}})^3 = (-K_{\mathcal{Q}_{k}})^3 =54$.
and by \cite[Theorem 2.16]{Tanaka1} ,\cite[Theorem 7.4]{Tanaka2} and \cite[Theorem 1.1]{Kollarsmooth}, $g_{\overline{k}}$ is also a blow-up of a quadratic threefold $ \mathcal{Q}_{\overline{k}} \subset \P_{\overline{k}}^4$ at a genus $7$-curve $C_{\overline{k}}$.
In particular, $\mathcal{Q}$ is smooth over $R$.
By extending $K$ if necessary, we may assume that $\mathcal{Q}_{k}$ is a quadric threefold in $\P^4_{k}$.
Since $H^2(\cO_{\mathcal{Q}_k}) = 0$, we can take a lift $L \in \Pic (\mathcal{Q})$ of a very ample line bundle on $\mathcal{Q}_{k}$ corresponding to a hyperplane in $\P^4_{k}$, and by the cohomology and base change theorem, $L$ defines an embedding $\mathcal{Q} \hookrightarrow \P^4_{R}$, which embeds $\mathcal{Q}$ as a relative smooth quadric threefold.

Let $\mathcal{C} \subset \mathcal{Q}$ be a non-isomorphic locus of $g$.
Then by the above consideration, $\mathcal{C}$ is a relative smooth proper curve of genus 7.
The isomorphism $\mathcal{C}_{\overline{K}} \simeq \Gamma_{\mathrm{f}}(X_{\overline{K}})$ clearly follows from the construction and Definition-Proposition \ref{defn-prop:Xtocurveflop} (5).

(5) The assertion follows from Proposition \ref{prop:relationGamma}.

(6) We may assume that $k=\overline{k}$, and we put $C := \mathcal{C}_{\overline{k}}$ and $Q:=\mathcal{Q}_{\overline{k}}$.
Note that, the embedding $C \subset Q \subset \P_{\overline{k}}^{4}
$
is given by a linear system $\cO(K_{C}-P_{1}-P_{2})$ for closed points $P_{1}, P_{2} \in C$.
In other words, This embedding is given by the projection of the canonically embedded curve $C \subset \P^{6}_{\overline{k}}$ from the line $\ell$ passing through $P_{1}$ and $P_{2}$.

Assume that $C$ is a generic curve with $g_4^1$.
In this case, by \cite[Proposition 6.3]{Mukaicurve}, the union of $4$-secant planes of $C$ with respect to the canonical embedding $C \hookrightarrow \mathbb{P}^{6}_{\overline{k}}$ is a quartic threefold $V$ in $\mathbb{P}^6$ (in \cite[Section 6]{Mukaicurve}, $V$ is written as $W$).
Let $Q' \subset \mathbb{P}^4_{\overline{k}}$ be a projection of $V \subset \mathbb{P}^6_{\overline{k}}$ from the line $\ell$.
Since $V$ is a quartic covered by $4$-secant planes, $Q'$ is a quadric containing $\P^{2}_{\overline{k}}$, so $Q'$ is different from $Q$.
Therefore, $Q'$ defines a member of the linear system $|2H - C|$ on $Q$. Here, $|2H - C|$ is a linear system consisting of a member of $|2H|$ containing $C$, and $H$ is a hyperplane section of $Q \subset \P^4_{\overline{k}}$.
In particular, we have
$H^0 (\widetilde{Q}, \cO_{\widetilde{Q}} (2 g_{\overline{k}}^{\ast}(H)- E_{C})) \neq 0$.
Here, we put $\widetilde{Q} := \widetilde{\mathcal{Q}}_{\overline{k}}$, and $E_{C}$ is the $g_{\overline{k}}$-exceptional divisor.
On the other hand, the pseudo-effective cone of $\widetilde{Q}$ is generated by two rays, the $g_{\overline{k}}$-exceptional divisor $E_{C}$ and the 
$b_{\overline{k}} \circ f_{\overline{k}}^{-1} \circ f^{+}_{\overline{k}}$-exceptional divisor $E^{+}$.
By applying \cite[Theorem 4.4.11]{Iskovskikh-Prokhorov} to the geometric generic fiber and taking specialization, we have
\[
E_{C} \sim 5(-K_{\widetilde{Q}})-3 E^{+} = -5 (-g_{\overline{k}}^{\ast} (3H) + E_{C}) - 3 E^{+}.
\]
Therefore, we have $E^{+} \sim 5 g_{\overline{k}}^{\ast} (H) -2 E_{C}$.
Then the numerical class $2g_{\overline{k}}^{\ast}(H)-E_{C}$ is equal to $\frac{2}{5} E^{+} - \frac{1}{5} E_{C}$, which contradicts the description of the pseudo-effective cone. This shows that $C$ is not a general curve with $g_4^1$ and it finishes the proof.
\end{proof}

\begin{thm}
\label{thm:grcounterexample}
Let $p$ be a prime number (resp.\ $p:=0$).
There exists a number field (resp.\ a function field of a curve over an algebraically closed field) $K$, finite place $\p$ of $K$ with residue characteristic $p$, and a prime Fano threefold $X$ of genus $7$ over $K$ such that the following hold.
\begin{enumerate}
\item 
The Galois representation $H^3_{\et} (X_{\overline{K}}, \Q_{\ell})$ is unramified at $\p$ for any prime number $\ell \notin \p$.
\item 
The Fano variety $X$ does not admit potentially good reduction at $\cO_{K, \p}$.
More precisely, for any finite extension $L$ of $\Frac (\widehat{\cO_{K, \p}})$, the Fano variety $X_{L}$ does not admit good reduction at its valuation ring $R$.
\end{enumerate}
\end{thm}

\begin{proof}
Let $K$, $\p$, $\mathcal{C}$ be as in Lemma \ref{lem:reductioncurves} (resp.\ Lemma \ref{lem:reductioncurveschar0}).
By extending $K$ and taking a finite place above $\p$ if necessary, we may take a prime Fano threefold $X$ of genus 7 over $K$ such that $X_{\overline{K}} \simeq \beta (\mathcal{C}_{\overline{K}})$.

(1) By extending $K$ and take a finite place above $\p$ if necessarily, 
we have the following:
There exists $\widetilde{X}' \rightarrow X$ which is a blow-up at conic, $\widetilde{Q} \rightarrow Q$ which is a blow-up at a smooth curve $C$ of genus $7$ on $Q$,
and a flop $\widetilde{X}' \dashrightarrow \widetilde{Q}$.
This is a $K$-rational version of construction in Definition-Proposition \ref{defn-prop:Xtocurveflop}.
By extending $K$ if necessarily again, we may assume that $\mathcal{C}_{K} \simeq C$ by Proposition  \ref{prop:relationGamma}.
Then we have an isomorphism as $\Gal(\overline{K}/K)$-modules
\[
H^3_{\et}(X_{\overline{K}}, \Q_{\ell}) \simeq H^3_{\et} (\widetilde{X}'_{\overline{K}}, \Q_{\ell}) \simeq H^3_{\et} (\widetilde{Q}_{\overline{K}}, \Q_{\ell}) \simeq H^1_{\et}(C_{\overline{K}}, \Q_{\ell})(-1)
\]
by the argument as in \cite[Corollary 4.12]{Kollarflop} and the blow-up formula.
Since $C \simeq \mathcal{C}_{K}$ has good reduction at $\cO_{K,\p}$, $H^1_{\et} (C_{\overline{K}},\Q_{\ell})$ is unramified at $\p$, and so is $H^3_{\et} (X_{\overline{K}}, \Q_{\ell})$.

(2) Suppose by contradiction that $X_{L}$ admits a good reduction at $R$, i.e., there exists a smooth projective scheme $\mathcal{X}$ over $R$ such that $\mathcal{X}_{L} \simeq X_{L}$.
We denote the residue field of $R$ by $l$, which is a finite extension of $k$.
By Proposition \ref{prop:mixedtwo-ray},
if we extend $L$, we have a smooth proper curve $\mathcal{C}'$ of genus $7$ over $R$ satisfying the following:
\begin{itemize}
\item 
The geometric special fiber $\mathcal{C}'_{\overline{l}}$ is not a generic curve with $g_4^1$ in the sense of Lemma \ref{lem:reductioncurveschar0}. Moreover, if $l$ is of characteristic $0$, $\mathcal{C}'_{\overline{l}}$ has no $g_4^1$.
\item 
We have
$\mathcal{C}'_{\overline{L}} \simeq \Gamma_{\mathrm{f}}(X_{\overline{L}})$.
\end{itemize}
Since we also have $X_{\overline{L}} = \beta (\mathcal{C}_{\overline{L}})$ by Proposition \ref{prop:relationGamma}, we have an isomorphism $\mathcal{C}'_{L} \simeq \mathcal{C}_{L}$ by extending $L$ if necessary.
Moreover, $\mathcal{C}_{R} := \mathcal{C} \otimes_{\cO_{K,\p}} R$ is a smooth proper curve of genus 7 over $R$ whose geometric special fiber $\mathcal{C}_{\overline{l}}$ is a generic curve with $g_4^1$ by the choice of $\mathcal{C}$.
By the uniqueness of a smooth projective model of $\mathcal{C}_{L}$ (\cite[Lemma 1.12]{Deligne-Mumford}), we have $\mathcal{C}_{R} \simeq \mathcal{C}'$ and it gives a contradiction since $\mathcal{C}_{\overline{l}}$ (resp.\ $\mathcal{C'}_{\overline{l}}$) is a generic curve with $g_4^1$ (resp.\ not a generic curve with $g_4^1$).
\end{proof}

\section{On Mukai $n$-folds of genus 7}
\label{Section:Mukain}

In this section, we shall study the arithmetic properties of Mukai $n$-folds that are linear sections of spinor tenfolds. A standard reference of such varieties is \cite{Kuznetsovspinor}. 
In this section, we freely use the notations in
Section \ref{subsection:spin10}.

Throughout this section, $n$ is an integer with $3 \leq n \leq 10$, unless otherwise stated. 

\subsection{Mukai $n$-folds of genus 7}

\begin{defn-prop}
Let $B$ be a {\cred connected} scheme, $V^{(B)}$ a free sheaf of rank 10 on $B$, $L^{(B)}$ a free module of rank $1$ on $B$, and $\varphi \colon L^{(B)} \hookrightarrow S^2 V^{(B)}$ be a split injection.
Suppose that we can choose basis $e_{1}, \ldots, e_{10}$ of $V^{(B)}$ such that 
the quadratic form $q_{\varphi}$ associated with $\varphi$ is given by 
\[
q_{\varphi} (\sum x_i e_i) = x_1 x_2 + x_3 x_4 + \cdots + x_9 x_{10}.
\]
Here, we identify $V^{(B)}$ with its dual by $e_{1}, \ldots, e_{10}$.
In this case, Assumptions (a) and (b) in Section \ref{subsection:spin10} are satisfied. 
Indeed, we have
\[
S_{+}^{(B)} := (\bigwedge_{\mathrm{even}} U_{\infty}^{(B)})^{\vee}, S_{-}^{(B)} := (\bigwedge_{\mathrm{odd}} U_{\infty}^{(B)})^{\vee},
\]
where 
\[
U_{\infty}^{(B)} := \cO_{B} e_{2} \oplus \cO_{B} e_{4} \oplus \cdots \oplus \cO_{B} e_{10}.
\]
See \cite{Mukaicurve} for details.
We denote the components of orthogonal Grassmannian in this case by $\Sigma_{+}^{(B)}$ and $\Sigma_{-}^{(B)}$.
\end{defn-prop}

\begin{defn}
\label{defn:Mukaivariety}
Let $B$ be a locally Noetherian scheme.
Let $n$ be an integer with $1 \leq n \leq 10$.
\begin{enumerate}
\item 
Let $S_{+}^{(B)} \twoheadrightarrow L$ be a locally free quotient over $B$ of rank $6+n$.
{\cred We denote the decomposition of $B$ into connected components by $\sqcup_{i \in I} B_{i}.$
If for any $i$, the intersection
\[
X(L)_{i} := \P_{B_{i}}(L|_{B_{i}}) \cap \Sigma_{+}^{(B_{i})}
\]
is smooth and dimensionally transverse over $B_{i}$, i.e., the dimension of any geometric fiber of $X(L)_{i}$ over $B_{i}$ is $n$, we denote the disjoint union $\sqcup_{i} X(L)_{i}$ over $B$ by $X(L)$, and 
we call it a \emph{split Mukai $n$-folds of genus $7$ over $B$}.}

\item 
A smooth projective scheme $X$ over $B$ is called a \emph{Mukai $n$-fold of genus $7$ over $B$} when any fiber over a geometric point $k$ on $S$ is a split Mukai $n$-fold of genus $7$ over $k$.

\item 
Suppose that $3 \leq n \leq 10$.
A smooth projective scheme $X$ over $B$ is called a \emph{Fano $n$-fold of coindex 3 and genus 7 over $B$}, when any geometric fiber is a (smooth) Fano variety of Picard rank $1$, index $n-2$, and $(\frac{-K_{X}}{n-2})^n =12$.

\end{enumerate}
\end{defn}

\begin{rem}
\label{rem:mukaindual}
Suppose that there exists a locally free quotient $S_{-}^{(B)} \twoheadrightarrow L'$ over $B$ of rank $10 - n$
in the setting of Definition \ref{defn:Mukaivariety}.
Then $X(L'^{\perp})$ is a split Mukai $n$-fold of genus $7$ over $B$ if and only if the intersection
$ \P(L') \cap \Sigma_{-} (V_{B}, \varphi)$
is smooth and dimensionally transverse by \cite[Lemma 3.3]{Kuznetsovspinor} (see the proof of Definition-Proposition \ref{defn-prop:dualvariety}).
\end{rem}

\begin{prop}
\label{prop:MukaivssplitMukai}
Let $B$ be a locally Noetherian scheme. Let $n$ be an integer with $3 \leq n \leq 10$.
Consider the following conditions:
\begin{enumerate}
\item 
$X$ is a split Mukai $n$-fold of genus 7 over $B$.
\item 
$X$ is a Mukai $n$-fold of genus $7$ over $B$.
\item 
$X$ is a Fano $n$-fold of coindex 3 and genus $7$ over $B$.
\end{enumerate}
Then we have $(1) \Rightarrow (2) \Rightarrow (3)$.
When $B$ is a spectrum of an algebraically closed field of characteristic $0$, we have $(1) \Leftrightarrow (2) \Leftrightarrow (3)$. Furthermore, when $B$ is a scheme over $\Q$, (2) and (3) are equivalent.
\end{prop}

\begin{proof}
First, we show $(1) \Rightarrow (2) \Rightarrow (3)$. Since $(1) \Rightarrow (2)$ is trivial, we show $(2) \Rightarrow (3)$.
We may assume that $B = \Spec k$ for an algebraically closed field $k$.
{\cred Let $X \subset \P(S_{+}^{(k)})$ be the embedding in the definition of split Mukai $n$-folds.
It suffices to show the following: 
\begin{enumerate}
\item[(a)]
$b_{1} (X) =0$ and $b_{2} (X)$=1. 
\item[(b)]
The hyperplane section $\cO_{X}(1)$ is an ample generator of $\Pic (X)$.
\item[(c)]
$\cO_{X}(-K_{X}) \simeq \cO_{X}(n-2)$, and $(\cO_{X}(1))^n =12$.
\end{enumerate}
}
{\cred
We suppose that $n \leq 9$.
Note that, by the construction, there exists a split Mukai $(n+1)$-fold $Y$ over $k$ containing $X$ as a hyperplane section.
Indeed, if we put $X = \P(L) \cap \Sigma_{+}$, then we can choose a codimension 1 subspace 
$L' \subset L^{\perp} \subset \P^{15 \vee}$ such that 
$L' \cap \Sigma_{+}^{\vee} \subset \P^{15 \vee}$ is smooth and dimensionally transverse. Then 
$Y= \P(L'^{\perp}) \cap \Sigma_{+}$ satisfies the desired condition.
Therefore, by the adjunction and weak Lefschetz, it suffices to show the above statements for $n=10.$
In this case, (a) is well-known in characteristic 0 (see \cite[Corollary 3.8]{Kuznetsovspinor} for example), and the general case can be shown by taking a lift of the split quadratic form.
Moreover, (c) follows from \cite[Proposition 2.1]{Mukaicurve}, 
and (b) follows from (c).
}

The second statement follows from \cite{Mukaimukain-fold}, \cite{Amb99} and \cite{Mella}.
The third statement clearly follows from the second statement.
\end{proof}

\begin{lem}
\label{lem:mukaincohomology}
Let $k$ be a field, $n$ an integer with {\cred $1 \leq n \leq 10$}, and 
{\cora 
\[
X \subset \Sigma_{+} := \Sigma_{+}^{(\Spec k)} \subset \P^{15} := \P (S_{+}^{(\Spec k)})
\]
be a split Mukai $n$-fold over $k$.
Let $\cO_{X}(1)$ be the hyperplane line bundle on $X \subset \P^{15}$.
}
Then we have the following.
\begin{enumerate}
\item
$h^i(\cO_{X} (j)) =0$ for $i>0$ and $j\geq 3-n$. 
\item 
$h^0(\cO_{X}(1)) =6+n$.
{\cora Moreover, $\cO_{X} (1)$ is an ample generator of $\Pic (X)$ if $3 \leq n \leq 10$.}
\end{enumerate}
Let $I_{X}$ be an ideal sheaf of $X$ with respect to the embedding $X \hookrightarrow \P^{5+n}$ defined by $\cO_{X}(1)$.
\begin{enumerate}
\setcounter{enumi}{2}
\item 
{\cred $H^0 (\cO_{\P^{5+n}}(i)) \rightarrow H^0 (\cO_{X}(i))$ is surjective for any $i \geq 0$.}
\item
$h^0 (I_{X} (2)) =10$ {\cred and $h^i (I_{X} (j)) =0$ for $i>0$ and $j\geq 3-n$}.
\item 
The kernel of the natural morphism
$S^{2} H^0 (I_{X} (2)) \rightarrow H^{0} (I^{2}_{X}(4))$
is 1-dimensional.
{\cora Moreover, the kernel of the natural morphism 
$S^{2} H^0 (I_{X} (2)) \rightarrow H^{0} (S^2 N_{X/\P^{5+n}}^{*}(4))$ 
is also $1$-dimensional.
}
\item 
The natural map
$ H^{0}(I_{X} (2)) \otimes_{k} \cO_{X} \rightarrow N_{X/\P^{5+n}}^{\ast} (2)$
is surjective.
{\cred
\item 
Let $X\subset \Sigma_{+} := \Sigma_{+}^{(\Spec k)}$ be the embedding in Definition \ref{defn:Mukaivariety}, and $I_{X}^{+}$ be the ideal sheaf of $X$ on $\Sigma_{+}$.
Then we have $H^1(I_{X}^{+}(1)) =0$.
}
\cred{\item 
Suppose that $k$ is of characteristic $0$.
Then the natural morphism $S^2 H^0 (I_{X} (2)) \rightarrow H^0(I_{X}^2 (4))$ is surjective.
Moreover, we have $H^i (I_{X}^2(4)) =0$ for $i>0$.
}
\cora{\item 
Suppose that $4 \leq n \leq 10$.
Then we have 
\begin{enumerate}
\item 
$H^i (S^2 N^{\ast}_{X/\P^{5+n}} (4)) =0$ for $i>0$.
\item
The natural morphism $S^2 H^0 (I_{X} (2)) \rightarrow H^0(S^2 N^{\ast}_{X/\P^{5+n}} (4))$ is surjective.
\end{enumerate}
}
\end{enumerate}
\end{lem}

\begin{proof}
{\cora  We may assume that $k$ is an algebraically closed field.}

First, we prove (1) and (2).
{\cora 
The latter statement in (2) follows from the argument in Proposition \ref{prop:MukaivssplitMukai}.
We shall show the remaining statements by descending induction on $n$.}
We consider the case where $n=10$. 
By \cite[Proposition 2 and Theorem 2]{Mehta-Ramanathan}, the Kodaira type vanishing theorem holds true for $X$.
Moreover, $X$ admits a lift to $W(k)$ as a Mukai $10$-fold, and $h^0(\cO(1)) =16$ holds true in the generic fiber (\cite[Subsection 2.2]{Kuznetsovspinor} for example).
By the invariance of $\chi(\cO (1))$, we have $h^0(\cO_{X}(1)) =16$.

Next, we show the general case.
{\cred As in the proof of Proposition \ref{prop:MukaivssplitMukai}, there exists a Mukai $(n+1)$-fold $Y$ over $k$ containing $X$ as a hyperplane section.}
We consider the exact sequence
\[
0 \rightarrow  \cO_{Y} (j-1) \rightarrow \cO_{Y}(j)  \rightarrow \cO_{X} (j) \rightarrow 0
\]
for $j \geq 3-n$.
Since $j-1 \geq 3-(n+1)$, by the induction hypothesis and the above exact sequence, we have $h^i (\cO_{X} (j)) =0$ for $i>0.$
Moreover, if $n \geq 2$, we have $h^0 (\cO_{X} (1)) =6+n$ by the induction hypothesis and the exact sequence for $j=1$ since we have $H^1 (\cO_{Y}) =0$.
Even when $n=1$, the same strategy works since in this case $Y$ has $b_{1} =0$ and a trivial canonical sheaf (see the proof of Proposition \ref{prop:MukaivssplitMukai}), i.e., $Y$ is a K3 surface.

Next, we show (3).
{\cred
This is trivial for $i=0,1$.
We shall show this by induction on $i$. We fix $i\geq 2$.
Note that the $n=1$ case follows from Noether's theorem. 
Assume that $n\geq 2$, and we assume by double induction that
the results for Mukai $(n-1)$-folds hold.
We can take a Mukai $(n-1)$-fold $Z$ over $k$ contained in $X$ as a hyperplane section.
We have a commutative diagram
\[
\xymatrix{
0 \ar[r] & \cO_{\P^{5+n}} (i-1) \ar[r] \ar[d] & \cO_{\P^{5+n}} (i) \ar[r] \ar[d] & \cO_{\P^{4+n}} (i) \ar[d] \ar[r] & 0 \\
0 \ar[r] & \cO_{X}(i-1) \ar[r] & \cO_{X}(i) \ar[r] & \cO_{Z} (i) \ar[r] &0
}
\]
where horizontal sequences are exact.
By (1), $H^0$ of horizontal sequences are also exact.
The induction hypothesis on $i$ and $n$ ensures that $H^0$ of left and right vertical arrows are surjective.
Therefore, $H^0$ of the middle vertical arrow is surjective as desired.

Next, we shall show (4).
Since we have an exact sequence
\begin{equation}
\label{eqn:ideal}
0 \rightarrow I_{X} (j) \rightarrow \cO_{\P^{5+n}}(j) \rightarrow \cO_{X} (j) \rightarrow 0,
\tag{$*$}
\end{equation}
the second statement follows from (1) and (3).
We shall show that $h^0(I_{X}(2))=10$.
This is well-known in characteristic $0$ (see \cite[Corollary 4.3]{Kuznetsovspinor} for example).
First, note that we can take a flat $W(k)$-lifting $\mathcal{X}$ of $X$.
Indeed, if we have $X = X(L) = \P_{\Spec k} (L) \cap \Sigma_{+}^{(\Spec k)}$ for quotient vector space $q \colon S_{+}^{(\Spec k)} \twoheadrightarrow L$,
we can take the free quotient $S_{+}^{(\Spec W (k))} \twoheadrightarrow \mathcal{L}$ which lifts $q$.
Then 
\[
\mathcal{X} := \P_{\Spec W(k)}(\mathcal{L}) \cap \Sigma_{+}^{(\Spec W(k))}
\]
is a desired lift.
Then by the invariance of $\chi (I(2))$ and the second statement, the equality $h^0 (I_{X} (2)) =10$ can be reduced to the case of characteristic $0$.
}

Next, we show (5).
{\cred We can take a Mukai $10$-fold $Y=\Sigma_{+}$ of genus $7$ containing $X$ and a Mukai $1$-fold $Z$ of genus $7$ contained in $X$.
Therefore, by \cite[Corollary 2.5]{Mukaicurve}, the restriction $H^0(I_{Y} (2)) \rightarrow H^0 (I_{Z} (2))$ is an isomorphism.
By (4), the restriction induces 
\[
H^0(I_{Y} (2)) \simeq H^0 (I_{X} (2)) \simeq H^0 (I_{Z} (2)).
\]
Therefore, the dimensions of the kernels in the statement are at least 1 by \cite[Corollary 1.11]{Mukaicurve} and at most 1 by \cite[Corollary 5.3]{Mukaicurve}.
}

Next, we show (6).
As in the proof of (5), this can be reduced to the case where $n=10$.
In this case, the result follows from \cite[Proposition 2.4]{Mukaicurve}.

{\cred
Next, we show (7). By the proof of (1) and (2), the restriction morphism $H^0 (\cO_{\Sigma_{+}} (1)) \rightarrow H^0 (\cO_{X} (1))$ is surjective. 
Then (7) follows from the exact sequence
\[
0 \rightarrow I^{+}_{X}(1) \rightarrow \cO_{\Sigma_{+}}(1) \rightarrow \cO_{X}(1) \rightarrow 0
\]
and (1).
}

{\cred
(8) follows from \cite[Corollary 4.4]{Kuznetsovspinor}.
}

{\cora
Next, we show (9), note that, by (4) and (5), (b) is equivalent to $h^0 (S^2 N^{\ast}_{X/\P^{5+n}} (4)) =54$.

First, we show (a) and (b) for the case where $n=10$, i.e., $X = \Sigma_{+}$.
By \cite[Corollary 2.4]{Mukaicurve},
The surjection $H^0(I_{\Sigma_{+}}(2)) \otimes_{k} \cO_{\Sigma_{+}} \rightarrow N^{\ast}_{\Sigma_{+}/\P^{15}} (2)$ in (6) induces the embedding $\Sigma_{+} \hookrightarrow \Gr (10,5)$.
In particular, the universal rank 5 quotient sheaf on $\Sigma_+$ is isomorphic to $N^{\ast}_{\Sigma_{+}/\P^{15}} (2)$.
Therefore, the projective bundle 
\[
\pi \colon \mathcal{D} := \P_{\Sigma_{+}} (N^{\ast}_{\Sigma_+/\P^{15}}(2)) \rightarrow \Sigma_+
\]
is isomorphic to the natural projection from a connected component of an orthogonal flag variety $\OFl_{+} (10,5,1) \rightarrow \Sigma_+$.
In particular, the Kodaira vanishing theorem holds for $\mathcal{D}$ by \cite[Proposition 2 and Theorem 2]{Mehta-Ramanathan}.
Let $\cO_{\mathcal{D}}(1)$ be the the tautological quotient line bundle on $\mathcal{D}$.
Then $\cO_{\mathcal{D}}(1)$ is isomorphic to $\cO_{\mathcal{D}} (H_{Q})$, where $H_Q$ is the pull-back of a hyperplane section of a hyperquadric 8-fold $\OGr(10,1) = \mathbb{Q}^8 \subset \P^{9}$ by the natural projection
\[
p \colon \mathcal{D} \rightarrow \mathbb{Q}^{8} \subset \P^{9}.
\]
Moreover, we have
\[
R\pi_* (\cO_{\mathcal{D}}(2 H_Q)) = S^2 N_{\Sigma_+/\P^{15}}^{*}(4).
\]
Therefore, it suffices to show that 
\[
H^i (\cO_{\mathcal{D}}(2H_Q)) =0 \textup{ and } h^0(\cO_{\mathcal{D}}(2H_Q)) =54
\] 
for $i>0.$
Note that we have
\[
\cO_{\mathcal{D}} (K_{\mathcal{D}}) \simeq  \cO_{\mathcal{D}}(-5) \otimes \pi^* ((\cO_{\Sigma_+} (K_{\Sigma_+})) \otimes_{k} \det (N_{\Sigma_+/\P^{15}}^{*}(2))) =
\cO_{\mathcal{D}} (-5 H_{Q} - 6 H_{S}),
\]
where $H_{S}$ is the pull-back of a hyperplane section of $\Sigma_{+} \subset \P^{15}$ by $\pi$.
Therefore, the first vanishing follows from the Kodaira vanishing theorem.
On the other hand, we have
\[
H^0(\cO_{\mathcal{D}} (2H_{Q})) =H^0(\cO_{\mathbb{Q}^8} (2)) =54.
\]
Note that the map $p$ is a flat fiber space morphism with $R^j p_* \cO_{\mathcal{D}} =0$ for any $j>0$ since for any closed point $v \in \mathcal{D}$, the fiber $p^{-1}(v) \subset \OFl_{+}(10,5,1)=\mathcal{D}$ is isomorphic to a component of the orthogonal Grassmannian $\OGr_{+}(8,4)$.

Next, we show (a) when $5 \leq n \leq 9$.
Let $X \subset \Sigma_{+}$ be a Mukai $n$-fold.
We put $\mathcal{D}_{X} := \pi^{-1} (X)$.
We shall show that 
\begin{equation}
H^i (\cO_{\mathcal{D}_{X}} (2H_{Q} + j H_{S}) ) =0 
\tag{i}
\end{equation}
for any $j\geq 5-n$ and $i>0$ by descending induction on $n$.
Note that, this vanishing holds when $n=10$ by the above argument.
As in the proof of (1) and (2), there exists a Mukai $(n+1)$-fold $Y \subset \Sigma_{+}$ over $k$ containing $X$.
Then the desired vanishing follows from the exact sequence
\begin{equation}
0 \rightarrow \cO_{\mathcal{D}_Y} (2H_{Q} + (j-1) H_{S})
\rightarrow \cO_{\mathcal{D}_Y} (2H_{Q} + j H_{S})
\rightarrow \cO_{\mathcal{D}_X} (2H_{Q} + j H_{S}) \rightarrow 0.
\tag{$**$}
\end{equation}
Therefore, we have 
\[
H^i (\cO_{\mathcal{D}_{X}} (2H_{Q})) = H^i (S^2 N^{\ast}_{X/\P^{5+n}} (4)) =0
\]
for $i>0$ by (i) as desired.

Next, we show (a) when $n=4$.
To this end, we shall show that 
\begin{equation}
H^i (\cO_{\mathcal{D}_{X}} (2H_{Q} + (4-n) H_{S}) ) =0 
\tag{ii}
\end{equation}
hold for any Mukai $n$-fold $X$ with $4\leq n \leq 10$ and $i>0$.
By the exact sequence ($**$) for $j= 4-n$ and (i), we can reduce it to the case where $n=10$.
By the Serre duality, it suffices to show that
$H^{14-i} (-7H_{Q}) =0$.
Then the desired vanishing follows from the Kodaira vanishing on $\mathbb{Q}^8$, which follows from \cite[Proposition 2 and Theorem 2]{Mehta-Ramanathan}.
Now we put $n:=4$.
Then we have
\[
H^i (\cO_{\mathcal{D}_{X}} (2H_{Q})) = H^i (S^2 N^{\ast}_{X/\P^{5+n}} (4)) =0
\]
by (ii) as desired.

Finally, we show (b) when $4 \leq n \leq 9$.
To this end, we show
\[
H^0 (\cO_{\mathcal{D}_{X}} (2H_{Q}+ jH_{S})) =0
\]
for $4-n \leq j \leq -1$ and $5 \leq n \leq 9$.
Note that, we have
\[
H^0 (\cO_{\mathcal{D}} (2 H_{Q} + j H_{S})) =0
\]
for $j\leq -1$ since $H_{Q}$ and $H_{S}$ span the pseudo-effective cone of $\mathcal{D}$.
Then the desired vanishing follows from the descending induction and the exact sequence ($**$), (i), and (ii).
Then the equality $h^0 ( \cO_{\mathcal{D}_{X}} (2 H_{Q})) =54$ can be reduced to the case $n=10$ by the exact sequence ($**$) again.
Therefore, we have
$h^0 (S^2 N^{\ast}_{X/\P^{5+n}} (4)) =54$
for any $ 4 \leq n \leq 10$.
}
\end{proof}

\begin{prop}
\label{prop:Mukainfoldsigmafamily}
{\cora
Let $1 \leq n \leq 9$ be an odd integer.
Let $B$ be a locally Noetherian connected scheme, and $f \colon X \rightarrow B$ a Mukai $n$-fold of genus $7$ over $B$.
Suppose one of the following.
\begin{enumerate}
\item[(a)]
$B$ is a $\Q$-scheme.
\item[(b)]
$B$ is a regular scheme of dimension $1$, or a spectrum of a field.
\item[(c)]
$n \geq 4$.
\end{enumerate}
Then the following hold.
}
\begin{enumerate}
\item 
The relatively ample generator $H \in \Pic_{X/B} (B)$ comes from an actual line bundle which is relatively very ample.
\item
Let 
\[
\iota \colon X \hookrightarrow \P_{B} (f_{\ast} H)
\]
be the embedding by $H$,
{\cred $\mathcal{I}_{X}$ an ideal sheaf of $X$ in $\P_{B}(f_{\ast} H)$,}
and $p$ a projection from $\P_{B} (f_{\ast} H)$ to $B$.
Then the coherent sheaf $W_{X}:= p_{\ast} (\mathcal{I}_{X} (2))$ is a locally free sheaf of rank 10 on $B$.
{\cora Moreover, in the cases (a) and (b) (resp,\ the case (c)), the kernel $F_{X}$ of a natural map
\[
S^{2} W_{X} \rightarrow p_{\ast} (\mathcal{I}_{X}^2 (4))
\]
\[
\textup{(resp. } 
S^{2} W_{X} \rightarrow p_{\ast} (S^2 N^*_{X/\P_{B}(f_{*}H)}(4))
\textup{)}
\]
defines a non-degenerate quadratic form $\varphi_{X} \colon F_{X} \hookrightarrow S^2 W_{X}$ on $W_{X}^{\vee}$.
}
\item 
We put $E_{X} := N_{X/\P_{B}(f_{\ast}H)}^{\ast} \otimes \cO_{X}(2)$, which is a locally free sheaf of rank $5$.
Then the morphism $p^{\ast} W_{X} \rightarrow E_{X}$ defines a morphism 
\[
\iota \colon X \rightarrow \Gr_B (W_{X}, 5) 
\]
which is an embedding factoring through $\OGr_B(W_{X}, \varphi_{X})$.
\item
The statements similar to Proposition \ref{prop:curvesigmafamily} (3) and (4) hold true.
\end{enumerate}
\end{prop}

\begin{proof}
We prove (1).
Let $\beta$ be the image of $H$ via the natural map
\[
\Pic_{X/B} (B) \rightarrow H^2_{\et}(B,\G_{m}).
\]
Since $X$ is of index $n-2$, we have $\beta^{n-2} =1$.
Moreover, since $f_{\ast} H$ is a $\beta$-twisted vector bundle of rank $6+n$ by Lemma \ref{lem:mukaincohomology}.(2), we have $\beta^{6+n}=1$. 
Therefore, the order of $\beta$ divides $n-2$ and $8$, and we get $\beta =1.$
{\cred This shows that $H$ comes from a line bundle on $X$.}
Moreover, $H$ is relatively very ample by Lemma \ref{lem:mukaincohomology}.(1)-(2) and the cohomology and base change theorem.

Then (2)-(4) basically follows from the same argument as in Proposition \ref{prop:curvesigmafamily}.
{\cora The only non-trivial point is where we prove that $F_{X}$ is a locally free direct summand of $S^2 W_X$.
In the case (a), this follows from the same argument as in Proposition \ref{prop:curvesigmafamily}.(1) by using Lemma \ref{lem:mukaincohomology}.(5), (8).
The case (b) follows from the same argument as in Remark \ref{rem:basedvr}.
The case (c) follows from Lemma \ref{lem:mukaincohomology}.(5), (9) by replacing $p_{\ast} (\mathcal{I}_{X}^{2} (4))$ with $p_{\ast} (S^2 N^*_{X/\P_{B}(f_*H)}(4))$ in the proof of Proposition \ref{prop:curvesigmafamily}.(1).
}
\end{proof}

\begin{prop}
\label{prop:splitsigmaisom}
Let $B$ be a {\cred connected locally Noetherian} scheme. 
Let $f \colon X \rightarrow B$ be a split Mukai $n$-fold of genus $7$ over $B$.
{\cora Suppose that one of the conditions (a)-(c) in Proposition \ref{prop:Mukainfoldsigmafamily} holds.
Then the statements (1)-(4) in Proposition \ref{prop:Mukainfoldsigmafamily} hold true for $X$.}
\end{prop}

\begin{proof}
This follows from the same argument as in Proposition \ref{prop:Mukainfoldsigmafamily} by using Lemma \ref{lem:mukaincohomology}.
Note that, since $f \colon X \rightarrow B$ is split,
several Brauer classes are automatically trivial.
\end{proof}

\begin{defn}
Let $R$ be a discrete valuation ring, and $K$ the fraction field of $R$.
\begin{enumerate}

\item 
Let $X$ be a Mukai $n$-fold of genus $7$ over $K$.
We say \emph{$X$ has good reduction as Mukai $n$-folds of genus $7$ at $R$} when there exists a Mukai $n$-fold $\mathcal{X}$ of genus $7$ over $\Spec R$ such that $\mathcal{X}_{K} \simeq X$.

\item
Let $X$ be a split Mukai $n$-fold of genus $7$ over $K$.
We say \emph{$X$ has good reduction as split Mukai $n$-folds of genus $7$ at $R$} when
there exists a split Mukai $n$-fold $\mathcal{X}$ of genus $7$ over $\Spec R$ such that $\mathcal{X}_{K} \simeq X$.

\end{enumerate}
\end{defn}

\subsection{Mukai tenfolds of genus 7}
In this subsection, we show that the Shafarevich conjecture for Mukai tenfolds of genus $7$ holds true.
Note that, since Mukai tenfolds of genus $7$ are flag schemes in the sense of \cite[Definition 3.1]{Javanpeykar-Loughran:flag},
this result is just a direct application of Javanpeykar--Loughran's results in \cite{Javanpeykar-Loughran:flag}.
We also show another finiteness result for Mukai tenfolds with $\cO(1)$.

\begin{thm}
\label{thm:Mukai10shaf}
Let $K$ be a number field, $S$ a finite set of finite places of $S$.
Then the set of $\cO_{K,S}$-isomorphism classes of Mukai tenfolds of genus $7$ over $\cO_{K,S}$ is a finite set.
In particular, the set of $K$-isomorphism classes of Mukai tenfolds of genus $7$ over $K$ with good reduction as Mukai tenfolds at the localization at any finite place $v \notin S$ is a finite set.
\end{thm}

\begin{proof}
This is a corollary of \cite[Theorem 3.6]{Javanpeykar-Loughran:flag}.
\end{proof}

\begin{rem}
In Theorem \ref{thm:Mukai10shaf}, we cannot treat the general good reduction since we do not have an analogue of Proposition \ref{prop:MukaivssplitMukai} $(2) \Leftrightarrow (3)$ in positive characteristic.
\end{rem}

\begin{rem}
\label{rem:mukai10}
We shall show that there exist infinitely many isomorphism classes of Mukai tenfolds over $\Q$. 
Let $K$ be a field, and
fix a vector space $V$ of rank 10 over $K$.

We take $K=\Q$ in the following.
Let $S$ be a finite set of prime numbers 
such that $\# S$ is even.
For places $v$ on $\Q$, we set 
\[
\epsilon_{v} :=
\left\{
\begin{array}{ll}
-1, & v \in S \textup{; and}  \\
1,& \textup{otherwise.}
\end{array}
\right.
\]
Then there exists a quadratic form $q_{S}$ on $V^{\vee}$ with  $\disc q \in \Q^{\times 2}$, signature (10,0), and Hasse invariants $(\epsilon_{v})_v$ (see \cite[Chapter IV, Proposition 7]{SerreArithmetic}).
The simple computations on Hilbert symbols show that, for any $a \in \Q^{\times}$, Hasse invariants of the quadratic form $a q_{S}$ at prime numbers are the same as those in $q_S$ (cf. \cite[Chapter III, Theorem 1]{SerreArithmetic}).
For any $S$, 
we can take a Mukai tenfold $\Sigma_{+} (V, \varphi_{q_{S}})$ over $\Q$ by Proposition \ref{prop:assumptionab}.
Moreover, by the argument as in Proposition \ref{prop:Mukainfoldsigmafamily}, we can recover the $8$-dimensional hyperquadric defined by $q_{S}=0$ from $\Sigma_{+} (V, \varphi_{q_{S}})$.
Therefore, this construction gives infinitely many Mukai tenfolds over $\Q.$
\end{rem}

\begin{defn}
Let $X$ be a Mukai tenfold $X$ of genus $7$ over a field $k$.
We say $X$ is \emph{with $\cO (1)$} when there exists an ample line bundle $H$ on $X$ which is a generator of $\Pic (X_{\overline{k}})$.
\end{defn}

\begin{prop}
\label{prop:mukaitenO1finite}
Let $K$ be a number field.
Then the number of $K$-isomorphism classes of Mukai tenfolds of genus 7 over $K$ with $\cO (1)$ is exactly $2^r$, in particular, it is finite.
Here, $r$ is the number of real places of $F$.
\end{prop}

\begin{proof}
By the proof of Proposition \ref{prop:Mukainfoldsigmafamily}, Mukai tenfolds of genus $7$ over $K$ with $\cO(1)$ is given by $\Sigma_{\pm} (V, \varphi_{q})$ for some quadratic form $q$ on a vector space $V$ of rank $10$ satisfying Assumptions (a) and (b) in subsection \ref{subsection:spin10} (see the notation in Proposition \ref{prop:assumptionab}).
Note that $\Sigma_{+} (V, \varphi_{q})$ and $\Sigma_{-} (V,\varphi_{q})$ are isomorphic to each other by the action of $\mathrm{O} (V,q) (K)$. 
Moreover, as in the argument in Remark \ref{rem:mukai10}, $\Sigma_{+} (V, \varphi_{q_{1}})$ and $\Sigma_{+} (V, \varphi_{q_{2}})$ are isomorphic if and only if $q_{1}$ and $q_{2}$ are $K$-similar.
Combining with Proposition \ref{prop:assumptionab},
the set of $K$-isomorphism classes of Mukai tenfolds over $K$ with $\cO(1)$ is isomorphic to the following set.
\[
A:=
\left\{
q 
\left| 
\begin{array}{l}
q \colon 
\textup{quadratic form over $K$}, \\
q \textup{ is of rank $10$},\\
\disc q  = 1,\\
C_{0}(q) \simeq M_{16} (K) \times M_{16} (K)
\end{array}
\right\}\right.
/ K\textup{-similar}.
\]
By \cite[Chapter V, Proposition 3.20]{Lam} (cf.\ \cite[Chapter V, Theorem 2.5 and (3.12)]{Lam}), for any quadratic form $K$ of rank 10 with trivial determinant, $C_{0}(q) \simeq M_{16} (K) \times M_{16} (K)$ holds if and only if the Hasse invariant of $q$ on any place $v$ on $K$ is trivial.
By \cite[Theorem 72:1]{OMeara} (cf.\ \cite[Remark 66:5]{OMeara}), quadratic forms of rank $10$ over $K$ with trivial discriminant and Hasse invariants correspond to integers $(s_{v})_{v}$ with $0 \leq s_{v} \leq 10$, $(-1)^{s_{v}} =1$, and $(-1)^{\frac{s_{v}(s_{v}-1)}{2}} =1$, where $v$ varies in all real places of $K$.
Here, the correspondence is given by taking negative indices of quadratic forms.
These conditions are equivalent to $s_{v} \in \{1,5,9 \}$.
For tuples $(s_{v})_{v}$ and $(t_{v})_{v}$,
let $q_{(s_{v})_{v}}$ and $q_{(t_{v})_{v}}$ be corresponding quadratic forms over $k$.
Note that $q_{(s_{v})_{v}}$ and $q_{(t_{v})_{v}}$ are $K$-similar if and only if $s_{v} \in \{t_{v}, 10-t_{v}\}$ for any $v$.
Indeed, since $c q_{(t_{v})_{v}}$ for $c \in K^{\times}$ corresponds to 
$(5- \sgn (v(c))(5- t_{v}))_{v}$, the ``only if" part is clear.
Here, 
\[
\sgn \colon \R^{\times} \rightarrow \{\pm 1\}
\]
is the signature function.
To show the ``if" part, it suffices to show that the map
\[
K^{\times} \rightarrow \{\pm 1 \}^{r} ; c \mapsto (\sgn (v(c)))_{v}
\]
is surjective. This follows from the approximation theorem.
Therefore, the cardinality of $A$ is exactly $2^r$.
\end{proof}

\subsection{Mukai ninefolds of genus 7}
In this section, we prove the stronger finiteness result for Mukai ninefolds of genus $7$.
The fundamental result is the rigidity of Mukai ninefolds.

\begin{prop}
Let $k$ be an algebraically closed field of characteristic $0$.
\begin{enumerate}
\item 
The number of $k$-isomorphism classes of Mukai ninefolds is just $1$.
\item 
Let $X$ be a Mukai ninefold over $k$. Then we have $H^1 (T_{X}) =0$.
\end{enumerate}
\end{prop}

\begin{proof}
(1) The assertion follows from the self-duality Proposition \ref{prop:spinorduality} and Igusa's results (\cite{Igusa}) on the transitivity of the action of $\Spin_{10}$ on $\P(S_{+}) \setminus \Sigma_{+}$ (cf.\ \cite[Proposition 1.13]{Mukaicurve} or \cite[Section 5]{Kuznetsovspinor}).

(2) The assertion follows from (1) and $H^2 (T_X) =0$, which follows from the Akizuki--Nakano vanishing.
\end{proof}

Here we briefly recall a construction of the Mukai ninefold of genus 7 \cite[\S5.4]{Kuznetsovspinor}.
For later use, we give the construction over an arbitrary field $k$.
Let $(W, q_W \in S^2W)$ be a $7$-dimensional orthogonal space over $k$.
Then we have the orthogonal flag variety $\OFl(W;3,1)$, which parametrizes the flags $W_3 \to W_1$ of $1$-dimensional (resp.\ $3$-dimensional) isotropic quotients $W_1$ (resp.\ $W_3$) of $W$.
This variety admits two projections $p_1 \colon \OFl(W;3,1) \to \OGr(W,1)$ and $p_2 \colon \OFl(W;3,1) \to \OGr(W,3)$.
By the definition, $\OGr(W,1)$ is a hyperquadric $\Q^5$ in $\P(W)$ and $p_2$ is a $\P^2$-bundle given by the projectivization of the universal quotient bundle on $\OGr(W,3)$.

It is well known that, over an algebraically closed field,  $\OGr(W,3)$ is isomorphic to a quadric $\Q^6 \subset \P(V)$ in the projectivized $8$-dimensional spinor representation $V$ associated with $q_W$, and $p_1$ is a $\P^3$-bundle given by the projectivization of the spinor bundle $E$, which is defined as $((p_1)_*(p_2)^* \cO_{\Q^6}(1))$ \cite{Ott88}.
Thus, over a non-closed field $k$, $\OGr(W,3)$ is a $k$-form of a $6$-dimensional hyperquadric $\Q^6$.
In the following, we assume that $\cO_{\Q^6}(1)$ is defined over the base field $k$.
Then, by the construction, $E$ is also defined over $k$.

Let $\widetilde X$ be the projectivization $\P_{\Q^5}(\cO_{\Q^5}(1) \oplus E) \to \Q^5$.
Then the tautological line bundle of this projectivization is generated by global sections and defines a morphism $\varphi \colon \widetilde X \to X \subset \P(W \oplus V)$.
The restriction of $\varphi$ to the subbundle $\P_{\Q^5}(E) \subset \P_{\Q^5}(\cO_{\Q^5}(1) \oplus E)$ is the morphism $\P_{\Q^5}(E) \to \Q^6$, which is identified with $p_2 \colon \OFl(W;3,1) \to \OGr(W,3)$ via the isomorphism $\P_{\Q^5}(E) \simeq \OFl(W;3,1)$.

\begin{prop}
Use the same notations as above.
Then $X$ is a Mukai ninefold of genus $7$, and $\varphi$ is a blow up of $X$ along a $6$-dimensional hyperquadric.

 Conversely, if $X$ is the Mukai ninefold of genus 7 over an algebraically closed field of characteristic $0$,
 then we have:
\begin{enumerate}
 \item $X$ is a two-orbits variety under the action of the automorphism group, 
 \item The blow up $\widetilde X$ of $X$  along the closed orbit, which is $\Q^6$, is isomorphic to the projective bundle
 \[
 \P(\cO_{\Q^5}(1) \oplus E) \to \Q^5.
 \]
\end{enumerate}
 \end{prop}

\begin{proof}
Note that Mukai ninefold of genus $7$ is the horospherical variety whose associated triple is \cite[Theorem~0.1~(ii)~2]{Pas09} (see, e.g. \cite[Proposition 1.16]{GPPS02}).
Therefore, this follows from \cite[Subsection 1.4, Subsection 1.5, and Subsection 4.3]{GPPS02}.
\end{proof}

The above construction gives the following correspondence.

\begin{prop}
\label{prop:Mukaininequadic}
Let $k$ be a field over $\Q$, and $W$ a $7$-dimensional vector space over $k$.
Then there is a bijection between 
\[
\{\textup{Mukai ninefolds of genus $7$ over }k \}/k\textup{-isom}
\]
and
\[
\left\{
\textup{hyperquadrics } \Q^5 \subset \P (W) \textup{ over $k$} \left|
\begin{array}{l}
\OGr(W,3) \textup{ admits } \cO_{\OGr(W,3)}(1) \\
\textup{which is a generator of } \Pic (\OGr(W,3)_{\overline{k}})
\end{array}
\right\}\right./k\textup{-isom}.
\]
\end{prop}

\begin{proof}
Note that, for a Mukai ninefold $X$ of genus 7 over $k$, a quadric 6-fold $\Q^6$ contained in $X_{\overline{k}}$ is unique by \cite[Lemma 5.10]{Kuznetsovspinor}.
Therefore, by Galois descent, we have a unique $k$-form of $\Q^6$ contained in $X$.
Note that the ample generator $H \in \Pic(X)$ is defined over $k$.
Thus the ample generator $\cO_{\Q^6}(1)$ is also defined over $k$. 
The remaining argument can be easily reduced to the case where $k = \overline{k}.$
\end{proof}

\begin{thm}
\label{thm:numberofmukai9}
Let $K$ be a number field.
Then the number of $K$-isomorphism classes of Mukai ninefolds over $K$ is equal to $2^r$ (in particular finitely many),
where $r$ denotes the number of real places of $K$.
\end{thm}

\begin{proof}
By Proposition \ref{prop:Mukaininequadic},
the set of $K$-isomorphism classes of Mukai ninefolds over $K$ is isomorphic to the following set:
\[
A:=
\left\{
q 
\left| 
\begin{array}{l}
q \colon 
\textup{quadratic form over $K$}, \\
q \textup{ is of rank $7$},\\
C_{0}(q) \simeq M_{8} (K)
\end{array}
\right\}\right.
/ K\textup{-similar}
\]
Indeed, $\OGr(W,3)$ admits a line bundle $\cO_{\OGr(W,3)} (1)$ if and only if $C_{0} (q) \simeq M_{8} (K)$ by an argument similar to the proof of Proposition \ref{prop:assumptionab} {\cred and \cite[Chapter V, Theorem 2.4]{Lam}}.
Any element in $A$ admits a representative $q$ with $\det q \in K^{\times 2}$, so we restrict to such forms.
By \cite[Chapter V, Proposition 3.20]{Lam} (cf.\ \cite[Chapter V, Theorem 2.4 and (3.12)]{Lam}),
for such a $q$, $C_{0} (q) \simeq M_{8} (K)$ holds if and only if
the Hasse invariant of $q$ on any place $v$ of $K$ is trivial.
By \cite[Theorem 72:1]{OMeara} (cf.\ \cite[Remark 66:5]{OMeara}),
quadratic forms of rank 7 over $K$ with trivial determinants and Hasse invariants correspond to integers (negative indices) $(s_{v})_{v}$ with
$0 \leq  s_{v} \leq 7$, $(-1)^{s_v} =1$, and $(-1)^{\frac{s_{v}(s_{v}-1)}{2}} =1$, where $v$ varies in all the real places of $K$.
These conditions are equivalent to $s_{v} \in \{0,4 \}$, so the number of $(s_{v})_{v}$ is exactly $2^r$.
Note that, corresponding quadratic forms have indices $(7-s_{v}, s_{v})_{v}$, so they are not similar to each other.
\end{proof}

\begin{rem}
\label{rem:mukai9}
Let $k$ be a field.
\begin{enumerate}
\item 
Suppose $k$ is of characteristic different from $2$.
Then every split Mukai ninefold is isomorphic to each other by \cite[Proposition 2]{Igusa}, where the transitivity of $\SO (q)$ on $\P^{15} \setminus \Sigma_{+}$ is proved when $q$ is split.

\item
Assume that $k$ is a number field.
By taking a linear hypersurface in $\P^{15}$, we can construct Mukai ninefolds over $k$ from a Mukai tenfold over $k$ with $\cO(1)$.
Moreover,  by the proof of Proposition \ref{prop:Mukainfoldsigmafamily}, 
Mukai tenfold over $k$ can be recovered from Mukai ninefold over $k$.
Therefore, by Proposition \ref{prop:mukaitenO1finite} and Theorem \ref{thm:numberofmukai9}, all possible hyperplanes of a given Mukai tenfold with $\cO(1)$ are isomorphic over $k$ when $k$ is a number field.
In other words, when $k$ is a number field, the transitivity of $\SO (10)$ on $\P^{15} \setminus \Sigma_{+}$ is obtained (without the splitting assumption of $q$ as in \cite{Igusa}). 
\end{enumerate}
\end{rem}

Note that, the finiteness in Theorem \ref{thm:numberofmukai9} can be also proved by the following:

\begin{prop}
\label{prop:autmukai9}
Let $k$ be a field of characteristic $0$, and $X$ a split Mukai ninefold over $k$.
Then we have 
\[
\Aut (X) \simeq ((\Spin(7)\times \GL_1(k))/\{\pm1\}) \ltimes  V.
\]
Here, 
\begin{itemize}
 \item $\{\pm 1\} \subset \Spin(7) $ is the central finite subgroup such that $\Spin(7)/\{\pm 1\} \simeq \SO(7)$.
 \item $\{\pm 1\}$ acts diagonally on $\Spin(7) \times \GL_1(k)$.
 \item $V$ is the $8$-dimensional spin representation of $\Spin(7)$.
\end{itemize}
\end{prop}

\begin{proof}
This follows from \cite[Proof of Lemma~1.17]{Pas09}.
(Note that, in \cite[Theorem~1.11 or Lemma~1.17]{Pas09}, the result is stated up to some finite groups;
in our case, the isomorphism is obtained by applying directly the argument given in the proof at \cite[page 974]{Pas09}.)
\end{proof}

\subsection{Mukai sixfold of genus 7 --- Counter-examples to Shafarevich conjecture}

The main theorem of this subsection is the following:

\begin{thm}
\label{thm:counterexampleshaf}
There exists a number field $K$ such that 
\[
\left\{X_{\overline{K}}  \left|
\begin{array}{l}
X \colon \textup{split Mukai sixfold of genus 7 over } K,\\
\textup{there exists a split Mukai sixfold }\mathcal{X} \textup{ over } \cO_{K} \\
\textup{such that } \mathcal{X}_{K} \simeq X
\end{array}
\right.
\right\}/\overline{K}\textup{-isom}
\]
is an infinite set.
In particular, the Shafarevich conjecture  for Mukai sixfolds does not hold (even when we only treat good reduction as split Mukai sixfolds).
\end{thm}

\begin{proof}
Let $\Gr(S_{-}^{(\Spec \Z)}, 4)$ be the Grassmannian over $\Z$ parametrizing rank 4 locally free quotient sheaves of $S_{-}^{(\Spec \Z)}$.
Let $W \subset \Gr (S_{-}^{(\Spec \Z)},4)$ be the closed subscheme consisting of rank 4 locally free quotients $S_{-}^{(\Spec \Z)} \twoheadrightarrow L$
such that 
\[
\emptyset \neq \P_{\Spec \Z}(L) \cap \Sigma_{-}^{(\Spec \Z)}  \subset \P_{\Spec \Z} (S_{-}^{(\Spec \Z)})
\]
Then the relative dimension of $\Gr (S_{-}^{(\Spec \Z)},4)$ (resp.\ $W$) over $\Z$ is 48 (resp.\ 46).
We put $U:= \Gr (S_{-}^{(\Spec \Z)},4) \setminus W$.
Clearly, we have the action of $\Spin(V^{(\Spec \Z)})$ on the open subscheme $U$.

We can show that there exists a number field $K$ such that 
\[
A :=
\left\{ 
u \in U(K) \left| u \textup{ extends to an integral point }\widetilde{u} \in  U( \cO_{K})
\right.
\right\}
\]
is Zariski-dense in $U$. 
Indeed, we can take an open sub-torus $T \subset \Gr (S^{(\Spec \Z)}_{-},4)_{\Q}$, and we can apply Hassett--Tschinkel's results \cite[Proposition 4.1]{Hassett-Tschinkel} for $(T, W|_{T})$ so that there exists a number field $K$ such that
\[
\left\{ 
u \in T(K) \left| u \textup{ extends to an integral point } U( \cO_{K})
\right.
\right\}
\]
is Zariski-dense in $T$, and so $A$ is Zariski-dense in $U$ for such a $K$.

Now we want to show that the set
\[
\{ u_{\overline{K}} \in U_{\overline{K}} (\overline{K})| u_{\overline{K}} \textup{ comes from } u \in A 
\} / \Spin(V)(\overline{K})
\]
is an infinite set.
Suppose for a contradiction that it is a finite set, then the base changes of $A$ can be covered by finitely many $\Spin(V)_{\overline{K}}$-orbits. Since the dimension of $\Spin(V)_{\overline{K}}$ is $45$, which is less than $48$, it contradicts the density of $A$.

For $u \in U(K)$, let $X_{u} := X(L_u^{\perp})$ be the split Mukai sixfold over $K$ given by $S_{-}^{(\Spec K)} \twoheadrightarrow L_u$ corresponding to $u$.
Then for any $u \in A$, $X_{u}$ has good reduction as split Mukai sixfolds of genus 7 at any finite place of $K$.
Indeed, for $S_{-}^{(\Spec \cO_{K})} \twoheadrightarrow L_{\widetilde{u}}$, $X_{\widetilde{u}} := X(L_{\widetilde{u}}^{\perp})$ gives a split Mukai sixfold over $\cO_{K}$ with the generic fiber $X_{u}$ by Remark \ref{rem:mukaindual}.

By Proposition \ref{prop:splitsigmaisom}, for $u, u' \in U(K)$, $X_{u, \overline{K}}$ and $X_{u',\overline{K}}$ are isomorphic if and only if $u_{\overline{K}}$ and $u'_{\overline{K}}$ lie on the same $\Spin (V^{(\Spec \overline{K})})$-orbit.
Since the dimension of $\Spin (V^{(\Spec \overline{K})})$ is 45, which is less than $\dim U =48$, such a number field $K$ satisfies the desired conditions.
\end{proof}

\subsection{Fontaine-type results for Mukai $n$-folds of genus 7}

In this subsection, we discuss the existence/non-existence of Mukai $n$-folds over $\Z$ for $3 \leq n \leq 10$.

Our goal is to prove the following theorem.

\begin{thm}
\label{thm:MukainoverZ}
\begin{enumerate}
\item 
For $n \in \{5,6,7,8,9,10\}$, there is a split Mukai $n$-fold of genus 7 over $\Z$.
\item 
There is no Mukai $3$-fold of genus $7$ over $\Q$ which admits good reduction at $\Z_{p}$ for any prime number $p$. 

\item 
There is no Mukai $4$-fold of genus 7 over $\Q$ which admits good reduction as Mukai $4$-folds of genus $7$ at $\Z_{p}$ for any prime number $p$.
\end{enumerate}
\end{thm}

\begin{rem}
\label{rem:Mukai1folds,2-folds}
Note that, Mukai $1$-folds (resp. Mukai $2$-folds) of genus $7$ are
curves of genus $7$ without $g_4^1$ (resp.\ Brill--Noether general K3 surfaces of genus $7$, see \cite[Section 4]{MukaiSugakuShintenkaiEnglish}).
Therefore, there is no family over $\Z$ in these cases by \cite[Theorem 1]{Fontaine} or \cite[\S 7.6 Theorem]{Abrashkin}.
\end{rem}

\medskip
\noindent {\bf Proof of (1)}
Let $\Sigma_{\Z} \subset \P_{\Z} (S)$ be the closed subscheme defined by 10 equations as in \cite[(0.1)]{Mukaicurve}.
Here, $S$ is a free $\Z$-module of rank 16, and we take a basis of $S^{\vee}$
\[
\xi_{\phi}, \xi_{12}, \ldots, \xi_{45}, \xi_{1234}, \ldots, \xi_{2345}.
\]
Then by \cite{Mukaicurve}, $\Sigma$ is a split  Mukai $10$-fold of genus 7 over $\Z$.

We define elements $v_{1}, \ldots, v_{5} \in \P_{\Z}(S)$ as the following:
\begin{eqnarray*}
v_{1} \colon \Z (\xi_{\phi} +\xi_{2345}) \hookrightarrow S^{\vee},\\
v_{2} \colon \Z (\xi_{13} + \xi_{45}) \hookrightarrow S^{\vee},\\
v_{3} \colon \Z (\xi_{14} + \xi_{25}) \hookrightarrow S^{\vee}, \\
v_{4} \colon \Z (\xi_{15} + \xi_{23}) \hookrightarrow S^{\vee}, \\
v_{5} \colon \Z (\xi_{12} + \xi_{1345}) \hookrightarrow S^{\vee}. 
\end{eqnarray*}
For $1 \leq i\leq 5$, we put
\[
V_{i}:= \image v_{1} \oplus \cdots \oplus \image v_{i} \subset S^{\vee}.
\]
Then by direct computation using the equations \cite[(0.1)]{Mukaicurve}, we can show that $\P_{\Z}(V_{i}^{\vee}) \subset \P_{\Z}(S)$  does not intersects with $\Sigma_{\Z} \subset \P_{\Z}(S)$.
Therefore, the dual
\[
X_{i} := \Sigma_{\Z}^{\vee} \cap \P_{\Z}(V_{i}^{\vee \perp}) \subset \P_{\Z}(S^\vee)
\]
is a split Mukai $(10-i)$-fold of genus $7$ over $\Z$ by Remark \ref{rem:mukaindual}. 
It finishes the proof.

\medskip
\noindent {\bf Proof of (2)}
This directly follows from \cite[Theorem 1]{Fontaine} or \cite[\S 7.6 Theorem]{Abrashkin} (see also Proposition \ref{prop:basic}).

\medskip
\noindent {\bf Proof of (3)}
Suppose by contradiction that there exists a Mukai $4$-fold $X$ of genus $7$ over $\Q$ with good reduction as Mukai 4-folds of genus 7 at any completion $\Z_{p}$.
Let $\mathcal{X}_{p}$ be a Mukai $4$-fold over $\Z_{p}$ with $\mathcal{X}_{p,\Q_{p}} \simeq X_{\Q_{p}}.$
Since $\Br (\Z_p)$ is trivial, we have 
\[ 
\Pic_{\mathcal{X}_{p}/\Z_{p}} (\Z_{p}) =
\Pic (\mathcal{X}_{p}).
\]

Therefore, as in the proof of Proposition \ref{prop:Mukainfoldsigmafamily}, the relatively ample generator $H_{p} \in \Pic_{\mathcal{X}_{p}/\Z_p} (\Z_{p})$ is an actual line bundle and relatively very ample.
Since the relatively ample generator $H \in \Pic_{X/\Q} (\Q)$ restricts to the actual line bundle $H_{p}$, by the injectivity of $\Br(\Q) \rightarrow \Pi_{p} \Br (\Q_{p})$, $H$ is also an actual line bundle. 
Moreover, by the same argument as in Proposition \ref{prop:curvesigmafamily}, we have a free $\Z_{p}$-module $W_{\mathcal{X}_{p}}$ with a quadratic form $\varphi_{\mathcal{X}_{p}}$,
the connected component $\Sigma_{\mathcal{X}_{p},+}$ of the corresponding orthogonal Grassmannian $\OGr_{\Z} (W_{\mathcal{X}_{p}}, \varphi_{\mathcal{X}_{p}})$,
a free $\Z_p$-module $S_{\mathcal{X}_{p},+}$ of rank 16, a free quotient $S_{\mathcal{X}_{p}, +} \twoheadrightarrow A_{\mathcal{X}_{p}}$ of rank 10, and an equality as a subscheme
\[
\mathcal{X}_{p} = \Sigma_{\mathcal{X}_{p},+} \cap \P_{\Z_p} (A_{\mathcal{X}_{p}}) \subset \P_{\Z_p} (S_{\mathcal{X}_{p},+}). 
\]
Moreover, the restrictions of $W_{\mathcal{X}_{p}}$, $\varphi_{\mathcal{X}_{p}}$, $\Sigma_{\mathcal{X}_{p}, +}$, and $S_{\mathcal{X}_{p}} \twoheadrightarrow A_{\mathcal{X}_{p}}$ to $\Q_{p}$ have $\Q$-forms $W_{X}, \varphi_{X}, \Sigma_{X,+},$ and $S_{X,+} \twoheadrightarrow A_{X}$ that are independent on $p$.
Let $Y$ be the dual of $X$, i.e.,
\[
Y := \Sigma_{X,+}^{\vee} \cap \P_{\Q} (A_{X}^{\perp}) \subset \P_{\Q} (S_{X,+}^{\vee}).
\]
Then by \cite[Lemma 3.3]{Kuznetsovspinor}, $Y$ is finite \'{e}tale over $\Q$.
Moreover, since $Y_{\Q_{p}}$ admits a finite \'{e}tale scheme model  
\[
\mathcal{Y}_{p} := \Sigma_{\mathcal{X}_{p},+}^{\vee} \cap \P_{\Z_{p}} (A_{\mathcal{X}_p}^{\perp}) \subset \P_{\Z_p} (S_{\mathcal{X}_p,+}^{\vee})
\]
over $\Z_p$ for any $p$, $Y$ is unramified at any $p$.
This implies that $Y$ splits into 12 distinct rational points.
In particular, 
\[
\mathcal{Y}_{\F_{2}} = \Sigma_{\mathcal{X}_2,+ \F_{2}} \cap \P_{\F_{2}} (A_{\mathcal{X}_2, \F_{2}}^{\perp}) \subset \P_{\F_{2}} (S_{\mathcal{X}_2,+,\F_{2}}^{\vee})
\]
consists of 12 distinct $\F_2$-rational points given by dimensionally transverse linear section of Mukai $10$-fold $\Sigma$ with $\cO(1)$ over $\F_{2}$.
Since $\Sigma$ has a rational point by \cite{Esnault}, $\Sigma$ is automatically split.
Therefore, to obtain a contradiction, it suffices to show the following proposition:

\begin{lem}
\label{lem:vectorF2}
Let $\Sigma \subset \P(S_{+})$ be a split Mukai $10$-fold over $\F_{2}$.
For any linear embedding $\P^{5} \hookrightarrow  \P (S_{+})$ such that $\P^5 \cap \Sigma \subset \P(S_{+})$ is a $0$-dimensional scheme $Z$, we have
\[
\# Z (\F_{2}) \leq 8.
\]
\end{lem}

\begin{proof}

Suppose that there exists a $\P^5 = \P (\F_{2}^{\oplus 6}) \hookrightarrow \P (S_{+})$ such that $\# Z (\F_{2}) \geq 9$.
The key observation is the following:
\begin{itemize}
\item
The maximal possible cardinality of a subset $T \subset \F_{2}^{\oplus 6}$ such that all subsets of $T$ consisting of 4 vectors are linearly independent is $8$.
\end{itemize}
This can be shown by easy computation. See \cite[Theorem 2]{DMMS} for the generalization.
Then $\# Z(\F_{2}) \geq 9$ implies that there exists a 4-points $P_{1},\ldots P_{4} \in Z (\F_{2})$ that span a $1$-dimensional or $2$-dimensional projective space.
In particular, there exists a $\P^2 \subset \P(S_{+})$ whose intersection with $\Sigma$ is 0-dimensional and containing at least 4-rational points.
Since $\Sigma_{\overline{\F}_{2}} \subset \P(S_{+,\overline{\F}_{2}})$ has no $4$-secant plane by \cite[Proposition 1.16]{Mukaicurve}, it gives the contradiction.
\end{proof}

\begin{rem}
{\cred The analogue of Lemma \ref{lem:vectorF2} over $\Q$ does not hold. 
More precisely, by using a computer algebra system, we can show the following:
Let $\Sigma \subset \P^{15}_{\Q}$ be the split Mukai tenfold defined by equations \cite[(0.1)]{Mukaicurve}.
Then there exists a linear subspace $\P^{5}_{\Q} \subset \P^{15}_{\Q}$ such that the intersection $\P^{5}_{\Q} \cap \Sigma$ consists of 12 (smooth) $\Q$-rational points.
Therefore, we genuinely need to work over $\F_{2}$.}
\end{rem}

\section{Remarks}
\label{section:Remarks}

\medskip
\noindent {\bf Possible generalization of the moduli description of Fano threefolds of genus $7$}
As we saw in Section \ref{section:fanovscurve}, prime Fano threefolds of genus 7 correspond to smooth proper curves of genus 7 without $g_4^1$.
According to \cite{Kuznetsov-Prokhorovnodal} and \cite[Section 8]{MukaiBrill}, 1-nodal non-factorial Fano threefolds of genus 7 would ideally correspond to smooth proper curves of genus 7 with $g_4^1$.
Therefore, it is interesting to consider a generalization of Theorem \ref{thm:introarithmeticTorelli} for such varieties.

\medskip
\noindent {\bf Shafarevich conjecture for Mukai $8$-folds of genus $7$}
It is known that
there are exactly two isomorphism classes of Mukai $8$-folds of genus $7$ over $\C$.
One with $H^1(T) =0$ is a general Mukai $8$-folds of genus $7$, and another one is a special Mukai $8$-folds of genus $7$ \cite{Kuznetsovspinor}.
If we restrict to general Mukai $8$-folds of genus $7$ (we only treat good reduction as general Mukai $8$-folds of genus $7$), then the same strategy as in \cite[Proposition 4.10]{Javanpeykar-Loughran:GoodReductionFano} works, and the Shafarevich conjecture holds true.
However, we do not know whether the general case holds or not.

\medskip
\noindent {\bf K-stability and arithmetic finiteness}
It is noteworthy that there seems to be some relation (or coincidence) between the K-stability of a given Fano variety $X$ and the arithmetic finiteness of $X$.

Recall that, for $n=10$ or $9$, the number of isomorphism classes of Mukai $n$-folds of genus $7$ over $\C$ is one.
On the other hand, their arithmetic behavior is widely different from each other;
over a fixed number field $K$, there are infinitely many isomorphism classes of Mukai $10$-folds of genus $7$, but finitely many of $9$-folds (Remark~\ref{rem:mukai10} and Theorem~\ref{thm:numberofmukai9}).
They also differ from each other in view of K-stability;
the Mukai $10$-fold of genus $7$ over $\C$, which is a homogeneous variety of $\SO(10)$, admits K\"ahler-Einstein metrics and hence is K-semistable, but the Mukai $9$-fold of genus $7$ over $\C$, which is a horospherical variety of Picard rank one, is K-unstable \cite{Mat72},\cite[\S 5.4.1]{Del20}.

A similar phenomenon happens also for quintic del Pezzo varieties.
A quintic del Pezzo variety $X$ is, by definition, a Fano $n$-fold $X$ of index $n-1$ and of degree $5$.
By the classification result \cite{Fuj90}, the dimension of a quintic del Pezzo variety $X$ is at most $6$, and each $X$ is a linear section of the $6$-dimensional Grassmann variety $\Gr(5,2) \subset \P^9$.
Moreover the number of isomorphism classes of quintic del Pezzo $n$-folds over $\C$ is one when $3 \leq n \leq 6$.
On the other hand, over a fixed number field $K$, there are infinitely many isomorphism classes for quintic del Pezzo $n$-folds when $n=3$ or $6$, but finitely many for $n$-folds with $n =4$, $5$  \cite{DKM}, \cite[\S 2.2]{DK19}.
A quintic del Pezzo $n$-fold $X$ is K-semistable when $n=3$ or $6$, but K-unstable when $n=4$, $5$ \cite{Fuj17}.

Based on this observation, it seems to be also interesting to study the arithmetic behavior of Mukai $8$-folds of genus $7$.
As we saw before, there are only two isomorphism classes of Mukai $8$-folds of genus $7$ over $\C$.
The general one is $K$-semistable, but the other one is $K$-unstable \cite{Del23}, \cite{Kan22}.
Note that, since the automorphism group of the general one is $\PGL(2) \times G_2$, there are infinitely many isomorphism classes over a fixed number field $K$.
It would be interesting to count the number of isomorphism classes of the special Mukai $8$-folds of genus $7$ over a number field.

\medskip



\noindent {\bf Acknowledgements}
The first author was supported by JSPS KAKENHI Grant Number 21K18577, 23K20204, 23K20786, 24K21512, 25K00905.
The second author was supported by JSPS KAKENHI Grant Number 23K12948.
The third author was supported by JSPS KAKENHI Grant number JP22KJ1780.
The fourth author was partially supported
by JSPS Grant-in-Aid for Scientific Research numbers 20H00114, 21H00973, 21K03246, 23H01073. 
The fourth author thanks Department of Mathematics in Kyoto University and Beijing Institute of Mathematical Sciences and Applications for hospitality. The authors wish to express their gratitude to Yuji Odaka for helpful advice on the moduli of prime Fano threefolds of genus $7$.
The authors are also deeply grateful to Shou Yoshikawa for valuable comments on the minimal-model program.
The authors thank Hiromu Tanaka for the comment on the reference of Proposition \ref{prop:basic}. 
Moreover, the authors thank Naoki Imai, Tatsuro Kawakami, Yuya Matsumoto, and Shigeru Mukai for valuable discussion.
Finally, the authors would like to thank anonymous referees for their valuable suggestions.
\printbibliography
\end{document}